\documentclass[reqno]{amsart}

\usepackage[shortlabels]{enumitem}

\usepackage[font=small]{caption}

\usepackage{amsmath,amssymb,amsthm,graphicx,epstopdf,mathrsfs,url}
\usepackage[usenames,dvipsnames]{color}
\definecolor{darkgreen}{rgb}{0.1,0.45,0.1}
\definecolor{darkblue}{rgb}{0.1,0.1,0.4}
\definecolor{darkgrey}{rgb}{0.5,0.5,0.5}

\definecolor{darkred}{rgb}{0.6,0.0,0.0}
\usepackage[colorlinks=true,linkcolor =darkred, citecolor=Blue]{hyperref}
\ifx\figforTeXisloaded\relax \else\global\let\figforTeXisloaded=\relax\fi
\catcode`\@=11
\ifx\ctr@ln@m\undefined\else%
    \immediate\write16{*** Fig4TeX WARNING : \string\ctr@ln@m\space already defined.}\fi
\def\ctr@ln@m#1{\ifx#1\undefined\else%
    \immediate\write16{*** Fig4TeX WARNING : \string#1 already defined.}\fi}
\ctr@ln@m\ctr@ld@f
\def\ctr@ld@f#1#2{\ctr@ln@m#2#1#2}
\ctr@ld@f\def\ctr@ln@w#1#2{\ctr@ln@m#2\csname#1\endcsname#2}
{\catcode`\/=0 \catcode`/\=12 /ctr@ld@f/gdef/BS@{\}}
\ctr@ld@f\def\ctr@lcsn@m#1{\expandafter\ifx\csname#1\endcsname\relax\else%
    \immediate\write16{*** Fig4TeX WARNING : \BS@\expandafter\string#1\space already defined.}\fi}
\ctr@ld@f\edef\colonc@tcode{\the\catcode`\:}
\ctr@ld@f\edef\semicolonc@tcode{\the\catcode`\;}
\ctr@ld@f\def\t@stc@tcodech@nge{{\let\c@tcodech@nged=\z@%
    \ifnum\colonc@tcode=\the\catcode`\:\else\let\c@tcodech@nged=\@ne\fi%
    \ifnum\semicolonc@tcode=\the\catcode`\;\else\let\c@tcodech@nged=\@ne\fi%
    \ifx\c@tcodech@nged\@ne%
    \immediate\write16{}
    \immediate\write16{!!!=============================================================!!!}
    \immediate\write16{ Fig4TeX WARNING:}
    \immediate\write16{ The category code of some characters has been changed, which will}
    \immediate\write16{ result in an error (message "Runaway argument?").}
    \immediate\write16{ This probably comes from another package that changed the category}
    \immediate\write16{ code after Fig4TeX was loaded. If that proves to be exact, the}
    \immediate\write16{ solution is to exchange the loading commands on top of your file}
    \immediate\write16{ so that Fig4TeX is loaded last. For example, in LaTeX, we should}
    \immediate\write16{ say :}
    \immediate\write16{\BS@ usepackage[french]{babel}}
    \immediate\write16{\BS@ usepackage{fig4tex}}
    \immediate\write16{!!!=============================================================!!!}
    \immediate\write16{}
    \fi}}
\ctr@ld@f\def\FigforTeX{F\kern-.05em i\kern-.05em g\kern-.1em\raise-.14em\hbox{4}\kern-.19em\TeX}
\ctr@ld@f\def\W@rnmesoldA#1{\W@rnmesold}
\ctr@ld@f\def\W@rnmesoldAB#1(#2){\W@rnmesold}
\ctr@ld@f\def\W@rnmesold{%
    \immediate\write16{}
    \immediate\write16{!!!=============================================================!!!}
    \immediate\write16{ Fig4TeX WARNING:}
    \immediate\write16{ The file to be compiled is not compatible with the current version}
    \immediate\write16{ of Fig4TeX. To fix that, upgrade the source file (mainly change \BS@ ps*}
    \immediate\write16{ macros by \BS@ fig* macros), or use fig4tex184.tex instead (\BS@ input fig4tex184}
    \immediate\write16{ or \BS@ usepackage{fig4tex184}).}
    \immediate\write16{!!!=============================================================!!!}
    \immediate\write16{}}
\ctr@ln@m\psbeginfig\let\psbeginfig\W@rnmesoldA
\ctr@ln@m\psset\let\psset\W@rnmesoldAB
\ctr@ln@m\pssetdefault\let\pssetdefault\W@rnmesoldAB
\ctr@ln@m\pssetupdate\let\pssetupdate\W@rnmesoldA
\ctr@ln@w{newdimen}\epsil@n\epsil@n=0.00005pt
\ctr@ln@w{newdimen}\Cepsil@n\Cepsil@n=0.005pt
\ctr@ln@w{newdimen}\dcq@\dcq@=254pt
\ctr@ln@w{newdimen}\PI@\PI@=3.141592pt
\ctr@ln@w{newdimen}\DemiPI@deg\DemiPI@deg=90pt
\ctr@ln@w{newdimen}\PI@deg\PI@deg=180pt
\ctr@ln@w{newdimen}\DePI@deg\DePI@deg=360pt
\ctr@ld@f\chardef\t@n=10
\ctr@ld@f\chardef\c@nt=100
\ctr@ld@f\chardef\@lxxiv=74
\ctr@ld@f\chardef\@xci=91
\ctr@ld@f\mathchardef\@nMnCQn=9949
\ctr@ld@f\chardef\@vi=6
\ctr@ld@f\chardef\@xxx=30
\ctr@ld@f\chardef\@lvi=56
\ctr@ld@f\chardef\@@lxxi=71
\ctr@ld@f\chardef\@lxxxv=85
\ctr@ld@f\mathchardef\@@mmmmlxviii=4068
\ctr@ld@f\mathchardef\@ccclx=360
\ctr@ld@f\mathchardef\@dccxx=720
\ctr@ln@w{newcount}\p@rtent \ctr@ln@w{newcount}\f@ctech \ctr@ln@w{newcount}\result@tent
\ctr@ln@w{newdimen}\v@lmin \ctr@ln@w{newdimen}\v@lmax \ctr@ln@w{newdimen}\v@leur
\ctr@ln@w{newdimen}\result@t\ctr@ln@w{newdimen}\result@@t
\ctr@ln@w{newdimen}\mili@u \ctr@ln@w{newdimen}\c@rre \ctr@ln@w{newdimen}\delt@
\ctr@ld@f\def\degT@rd{0.017453 }  % pi/180
\ctr@ld@f\def\rdT@deg{57.295779 } % 180/pi
\ctr@ln@m\v@leurseule
{\catcode`p=12 \catcode`t=12 \gdef\v@leurseule#1pt{#1}}
\ctr@ld@f\def\repdecn@mb#1{\expandafter\v@leurseule\the#1\space}
\ctr@ld@f\def\arct@n#1(#2,#3){{\v@lmin=#2\v@lmax=#3%
    \maxim@m{\mili@u}{-\v@lmin}{\v@lmin}\maxim@m{\c@rre}{-\v@lmax}{\v@lmax}%
    \delt@=\mili@u\m@ech\mili@u%
    \ifdim\c@rre>\@nMnCQn\mili@u\divide\v@lmax\tw@\c@lATAN\v@leur(\z@,\v@lmax)% DY > 9949 DX
    \else%
    \maxim@m{\mili@u}{-\v@lmin}{\v@lmin}\maxim@m{\c@rre}{-\v@lmax}{\v@lmax}%
    \m@ech\c@rre%
    \ifdim\mili@u>\@nMnCQn\c@rre\divide\v@lmin\tw@% DX > 9949 DY
    \maxim@m{\mili@u}{-\v@lmin}{\v@lmin}\c@lATAN\v@leur(\mili@u,\z@)%
    \else\c@lATAN\v@leur(\delt@,\v@lmax)\fi\fi%
    \ifdim\v@lmin<\z@\v@leur=-\v@leur\ifdim\v@lmax<\z@\advance\v@leur-\PI@%
    \else\advance\v@leur\PI@\fi\fi%
    \global\result@t=\v@leur}#1=\result@t}
\ctr@ld@f\def\m@ech#1{\ifdim#1>1.646pt\divide\mili@u\t@n\divide\c@rre\t@n\m@ech#1\fi}
\ctr@ld@f\def\c@lATAN#1(#2,#3){{\v@lmin=#2\v@lmax=#3\v@leur=\z@\delt@=\tw@ pt%
    \un@iter{0.785398}{\v@lmax<}%
    \un@iter{0.463648}{\v@lmax<}%
    \un@iter{0.244979}{\v@lmax<}%
    \un@iter{0.124355}{\v@lmax<}%
    \un@iter{0.062419}{\v@lmax<}%
    \un@iter{0.031240}{\v@lmax<}%
    \un@iter{0.015624}{\v@lmax<}%
    \un@iter{0.007812}{\v@lmax<}%
    \un@iter{0.003906}{\v@lmax<}%
    \un@iter{0.001953}{\v@lmax<}%
    \un@iter{0.000976}{\v@lmax<}%
    \un@iter{0.000488}{\v@lmax<}%
    \un@iter{0.000244}{\v@lmax<}%
    \un@iter{0.000122}{\v@lmax<}%
    \un@iter{0.000061}{\v@lmax<}%
    \un@iter{0.000030}{\v@lmax<}%
    \un@iter{0.000015}{\v@lmax<}%
    \global\result@t=\v@leur}#1=\result@t}
\ctr@ld@f\def\un@iter#1#2{%
    \divide\delt@\tw@\edef\dpmn@{\repdecn@mb{\delt@}}%
    \mili@u=\v@lmin%
    \ifdim#2\z@%
      \advance\v@lmin-\dpmn@\v@lmax\advance\v@lmax\dpmn@\mili@u%
      \advance\v@leur-#1pt%
    \else%
      \advance\v@lmin\dpmn@\v@lmax\advance\v@lmax-\dpmn@\mili@u%
      \advance\v@leur#1pt%
    \fi}
\ctr@ld@f\def\c@ssin#1#2#3{\expandafter\ifx\csname COS@\number#3\endcsname\relax\c@lCS{#3pt}%
    \expandafter\xdef\csname COS@\number#3\endcsname{\repdecn@mb\result@t}%
    \expandafter\xdef\csname SIN@\number#3\endcsname{\repdecn@mb\result@@t}\fi%
    \edef#1{\csname COS@\number#3\endcsname}\edef#2{\csname SIN@\number#3\endcsname}}
\ctr@ld@f\def\c@lCS#1{{\mili@u=#1\p@rtent=\@ne%
    \relax\ifdim\mili@u<\z@\red@ng<-\else\red@ng>+\fi\f@ctech=\p@rtent%
    \relax\ifdim\mili@u<\z@\mili@u=-\mili@u\f@ctech=-\f@ctech\fi\c@@lCS}}
\ctr@ld@f\def\c@@lCS{\v@lmin=\mili@u\c@rre=-\mili@u\advance\c@rre\DemiPI@deg\v@lmax=\c@rre%
    \mili@u\@@lxxi\mili@u\divide\mili@u\@@mmmmlxviii%
    \edef\v@larg{\repdecn@mb{\mili@u}}\mili@u=-\v@larg\mili@u%
    \edef\v@lmxde{\repdecn@mb{\mili@u}}%
    \c@rre\@@lxxi\c@rre\divide\c@rre\@@mmmmlxviii%
    \edef\v@largC{\repdecn@mb{\c@rre}}\c@rre=-\v@largC\c@rre%
    \edef\v@lmxdeC{\repdecn@mb{\c@rre}}%
    \fctc@s\mili@u\v@lmin\global\result@t\p@rtent\v@leur%
    \let\t@mp=\v@larg\let\v@larg=\v@largC\let\v@largC=\t@mp%
    \let\t@mp=\v@lmxde\let\v@lmxde=\v@lmxdeC\let\v@lmxdeC=\t@mp%
    \fctc@s\c@rre\v@lmax\global\result@@t\f@ctech\v@leur}
\ctr@ld@f\def\fctc@s#1#2{\v@leur=#1\relax\ifdim#2<\@lxxxv\p@\cosser@h\else\sinser@t\fi}
\ctr@ld@f\def\cosser@h{\advance\v@leur\@lvi\p@\divide\v@leur\@lvi%
    \v@leur=\v@lmxde\v@leur\advance\v@leur\@xxx\p@%
    \v@leur=\v@lmxde\v@leur\advance\v@leur\@ccclx\p@%
    \v@leur=\v@lmxde\v@leur\advance\v@leur\@dccxx\p@\divide\v@leur\@dccxx}
\ctr@ld@f\def\sinser@t{\v@leur=\v@lmxdeC\p@\advance\v@leur\@vi\p@%
    \v@leur=\v@largC\v@leur\divide\v@leur\@vi}
\ctr@ld@f\def\red@ng#1#2{\relax\ifdim\mili@u#1#2\DemiPI@deg\advance\mili@u#2-\PI@deg%
    \p@rtent=-\p@rtent\red@ng#1#2\fi}
\ctr@ld@f\def\pr@c@lCS#1#2#3{\ctr@lcsn@m{COS@\number#3 }%
    \expandafter\xdef\csname COS@\number#3\endcsname{#1}%
    \expandafter\xdef\csname SIN@\number#3\endcsname{#2}}
\pr@c@lCS{1}{0}{0}
\pr@c@lCS{0.7071}{0.7071}{45}\pr@c@lCS{0.7071}{-0.7071}{-45}
\pr@c@lCS{0}{1}{90}          \pr@c@lCS{0}{-1}{-90}
\pr@c@lCS{-1}{0}{180}        \pr@c@lCS{-1}{0}{-180}
\pr@c@lCS{0}{-1}{270}        \pr@c@lCS{0}{1}{-270}
\ctr@ld@f\def\invers@#1#2{{\v@leur=#2\maxim@m{\v@lmax}{-\v@leur}{\v@leur}%
    \f@ctech=\@ne\m@inv@rs%
    \multiply\v@leur\f@ctech\edef\v@lv@leur{\repdecn@mb{\v@leur}}%
    \p@rtentiere{\p@rtent}{\v@leur}\v@lmin=\p@\divide\v@lmin\p@rtent%
    \inv@rs@\multiply\v@lmax\f@ctech\global\result@t=\v@lmax}#1=\result@t}
\ctr@ld@f\def\m@inv@rs{\ifdim\v@lmax<\p@\multiply\v@lmax\t@n\multiply\f@ctech\t@n\m@inv@rs\fi}
\ctr@ld@f\def\inv@rs@{\v@lmax=-\v@lmin\v@lmax=\v@lv@leur\v@lmax%
    \advance\v@lmax\tw@ pt\v@lmax=\repdecn@mb{\v@lmin}\v@lmax%
    \delt@=\v@lmax\advance\delt@-\v@lmin\ifdim\delt@<\z@\delt@=-\delt@\fi%
    \ifdim\delt@>\epsil@n\v@lmin=\v@lmax\inv@rs@\fi}
\ctr@ld@f\def\minim@m#1#2#3{\relax\ifdim#2<#3#1=#2\else#1=#3\fi}
\ctr@ld@f\def\maxim@m#1#2#3{\relax\ifdim#2>#3#1=#2\else#1=#3\fi}
\ctr@ld@f\def\p@rtentiere#1#2{#1=#2\divide#1by65536 }
\ctr@ld@f\def\r@undint#1#2{{\v@leur=#2\divide\v@leur\t@n\p@rtentiere{\p@rtent}{\v@leur}%
    \v@leur=\p@rtent pt\global\result@t=\t@n\v@leur}#1=\result@t}
\ctr@ld@f\def\sqrt@#1#2{{\v@leur=#2%
    \minim@m{\v@lmin}{\p@}{\v@leur}\maxim@m{\v@lmax}{\p@}{\v@leur}%
    \f@ctech=\@ne\m@sqrt@\sqrt@@%
    \mili@u=\v@lmin\advance\mili@u\v@lmax\divide\mili@u\tw@\multiply\mili@u\f@ctech%
    \global\result@t=\mili@u}#1=\result@t}
\ctr@ld@f\def\m@sqrt@{\ifdim\v@leur>\dcq@\divide\v@leur\c@nt\v@lmax=\v@leur%
    \multiply\f@ctech\t@n\m@sqrt@\fi}
\ctr@ld@f\def\sqrt@@{\mili@u=\v@lmin\advance\mili@u\v@lmax\divide\mili@u\tw@%
    \c@rre=\repdecn@mb{\mili@u}\mili@u%
    \ifdim\c@rre<\v@leur\v@lmin=\mili@u\else\v@lmax=\mili@u\fi%
    \delt@=\v@lmax\advance\delt@-\v@lmin\ifdim\delt@>\epsil@n\sqrt@@\fi}
\ctr@ld@f\def\extrairelepremi@r#1\de#2{\expandafter\lepremi@r#2@#1#2}
\ctr@ld@f\def\lepremi@r#1,#2@#3#4{\def#3{#1}\def#4{#2}\ignorespaces}
\ctr@ld@f\def\@cfor#1:=#2\do#3{%
  \edef\@fortemp{#2}%
  \ifx\@fortemp\empty\else\@cforloop#2,\@nil,\@nil\@@#1{#3}\fi}
\ctr@ln@m\@nextwhile
\ctr@ld@f\def\@cforloop#1,#2\@@#3#4{%
  \def#3{#1}%
  \ifx#3\Fig@nnil\let\@nextwhile=\Fig@fornoop\else#4\relax\let\@nextwhile=\@cforloop\fi%
  \@nextwhile#2\@@#3{#4}}

\ctr@ld@f\def\@ecfor#1:=#2\do#3{%
  \def\@@cfor{\@cfor#1:=}%
  \edef\@@@cfor{#2}%
  \expandafter\@@cfor\@@@cfor\do{#3}}
\ctr@ld@f\def\Fig@nnil{\@nil}
\ctr@ld@f\def\Fig@fornoop#1\@@#2#3{}
\ctr@ln@m\list@@rg
\ctr@ld@f\def\trtlis@rg#1#2{\def\list@@rg{#1}%
    \@ecfor\p@rv@l:=\list@@rg\do{\expandafter#2\p@rv@l|}}
\ctr@ld@f\def\trtlis@rgtok#1{\let@xte={}\let\n@xt\addt@t@xt\addt@t@xt #1}
\ctr@ln@m\M@cro
\ctr@ln@m\n@xt
\ctr@ld@f\def\addt@t@xt#1{\if#1|\let\n@xt\relax\else%
    \if#1,\expandafter\M@cro\the\let@xte|\let@xte={}%
    \else\let@xte=\expandafter{\the\let@xte #1}\fi\fi\n@xt}
\ctr@ln@w{newbox}\b@xvisu
\ctr@ln@w{newtoks}\let@xte
\ctr@ln@w{newif}\ifitis@K
\ctr@ln@w{newcount}\s@mme
\ctr@ln@w{newcount}\l@mbd@un \ctr@ln@w{newcount}\l@mbd@de
\ctr@ln@w{newcount}\superc@ntr@l\superc@ntr@l=\@ne        % Controle impose
\ctr@ln@w{newcount}\typec@ntr@l\typec@ntr@l=\superc@ntr@l % Controle souhaite
\ctr@ln@w{newdimen}\v@lX  \ctr@ln@w{newdimen}\v@lY  \ctr@ln@w{newdimen}\v@lZ
\ctr@ln@w{newdimen}\v@lXa \ctr@ln@w{newdimen}\v@lYa \ctr@ln@w{newdimen}\v@lZa
\ctr@ln@w{newdimen}\unit@\unit@=\p@ % Initialisation a la valeur par defaut.
\ctr@ld@f\def\unit@util{pt}
\ctr@ld@f\def\ptT@ptps{0.996264 }
\ctr@ld@f\def\ptpsT@pt{1.00375 }
\ctr@ld@f\def\ptT@unit@{1} % Initialisation correspondant a la valeur par defaut de \unit@
\ctr@ld@f\def\setunit@#1{\def\unit@util{#1}\setunit@@#1:\invers@{\result@t}{\unit@}%
    \edef\ptT@unit@{\repdecn@mb\result@t}}
\ctr@ld@f\def\setunit@@#1#2:{\ifcat#1a\unit@=\@ne#1#2\else\unit@=#1#2\fi}
\ctr@ld@f\def\d@fm@cdim#1#2{{\v@leur=#2\v@leur=\ptT@unit@\v@leur\xdef#1{\repdecn@mb\v@leur}}}
\ctr@ln@w{newif}\ifBdingB@x\BdingB@xtrue
\ctr@ln@w{newdimen}\c@@rdXmin \ctr@ln@w{newdimen}\c@@rdYmin  % Dimensions de la BoundingBox
\ctr@ln@w{newdimen}\c@@rdXmax \ctr@ln@w{newdimen}\c@@rdYmax
\ctr@ld@f\def\b@undb@x#1#2{\ifBdingB@x%
    \relax\ifdim#1<\c@@rdXmin\global\c@@rdXmin=#1\fi%
    \relax\ifdim#2<\c@@rdYmin\global\c@@rdYmin=#2\fi%
    \relax\ifdim#1>\c@@rdXmax\global\c@@rdXmax=#1\fi%
    \relax\ifdim#2>\c@@rdYmax\global\c@@rdYmax=#2\fi\fi}
\ctr@ld@f\def\b@undb@xP#1{{\Figg@tXY{#1}\b@undb@x{\v@lX}{\v@lY}}}
\ctr@ld@f\def\ellBB@x#1;#2,#3(#4,#5,#6){{\s@uvc@ntr@l\et@tellBB@x%
    \setc@ntr@l{2}\figptell-2::#1;#2,#3(#4,#6)\b@undb@xP{-2}%
    \figptell-2::#1;#2,#3(#5,#6)\b@undb@xP{-2}%
    \c@ssin{\C@}{\S@}{#6}\v@lmin=\C@ pt\v@lmax=\S@ pt%
    \mili@u=#3\v@lmin\delt@=#2\v@lmax\arct@n\v@leur(\delt@,\mili@u)%
    \mili@u=-#3\v@lmax\delt@=#2\v@lmin\arct@n\c@rre(\delt@,\mili@u)%
    \v@leur=\rdT@deg\v@leur\advance\v@leur-\DePI@deg%
    \c@rre=\rdT@deg\c@rre\advance\c@rre-\DePI@deg%
    \v@lmin=#4pt\v@lmax=#5pt%
    \loop\ifdim\v@leur<\v@lmax\ifdim\v@leur>\v@lmin%
    \edef\@ngle{\repdecn@mb\v@leur}\figptell-2::#1;#2,#3(\@ngle,#6)%
    \b@undb@xP{-2}\fi\advance\v@leur\PI@deg\repeat%
    \loop\ifdim\c@rre<\v@lmax\ifdim\c@rre>\v@lmin%
    \edef\@ngle{\repdecn@mb\c@rre}\figptell-2::#1;#2,#3(\@ngle,#6)%
    \b@undb@xP{-2}\fi\advance\c@rre\PI@deg\repeat%
    \resetc@ntr@l\et@tellBB@x}\ignorespaces}
\ctr@ld@f\def\initb@undb@x{\c@@rdXmin=\maxdimen\c@@rdYmin=\maxdimen%
    \c@@rdXmax=-\maxdimen\c@@rdYmax=-\maxdimen}
\ctr@ld@f\def\c@ntr@lnum#1{%
    \relax\ifnum\typec@ntr@l=\@ne%
    \ifnum#1<\z@%
    \immediate\write16{*** Forbidden point number (#1). Abort.}\end\fi\fi%
    \set@bjc@de{#1}}
\ctr@ln@m\objc@de
\ctr@ld@f\def\set@bjc@de#1{\edef\objc@de{@BJ\ifnum#1<\z@ M\romannumeral-#1\else\romannumeral#1\fi}}
\s@mme=\m@ne\loop\ifnum\s@mme>-19
  \set@bjc@de{\s@mme}\ctr@lcsn@m\objc@de\ctr@lcsn@m{\objc@de T}
\advance\s@mme\m@ne\repeat
\s@mme=\@ne\loop\ifnum\s@mme<6
  \set@bjc@de{\s@mme}\ctr@lcsn@m\objc@de\ctr@lcsn@m{\objc@de T}
\advance\s@mme\@ne\repeat
\ctr@ld@f\def\setc@ntr@l#1{\ifnum\superc@ntr@l>#1\typec@ntr@l=\superc@ntr@l%
    \else\typec@ntr@l=#1\fi}
\ctr@ld@f\def\resetc@ntr@l#1{\global\superc@ntr@l=#1\setc@ntr@l{#1}}
\ctr@ld@f\def\s@uvc@ntr@l#1{\edef#1{\the\superc@ntr@l}}
\ctr@ln@m\c@lproscal
\ctr@ld@f\def\c@lproscalDD#1[#2,#3]{{\Figg@tXY{#2}%
    \edef\Xu@{\repdecn@mb{\v@lX}}\edef\Yu@{\repdecn@mb{\v@lY}}\Figg@tXY{#3}%
    \global\result@t=\Xu@\v@lX\global\advance\result@t\Yu@\v@lY}#1=\result@t}
\ctr@ld@f\def\c@lproscalTD#1[#2,#3]{{\Figg@tXY{#2}\edef\Xu@{\repdecn@mb{\v@lX}}%
    \edef\Yu@{\repdecn@mb{\v@lY}}\edef\Zu@{\repdecn@mb{\v@lZ}}%
    \Figg@tXY{#3}\global\result@t=\Xu@\v@lX\global\advance\result@t\Yu@\v@lY%
    \global\advance\result@t\Zu@\v@lZ}#1=\result@t}
\ctr@ld@f\def\c@lprovec#1{%
    \det@rmC\v@lZa(\v@lX,\v@lY,\v@lmin,\v@lmax)%
    \det@rmC\v@lXa(\v@lY,\v@lZ,\v@lmax,\v@leur)%
    \det@rmC\v@lYa(\v@lZ,\v@lX,\v@leur,\v@lmin)%
    \Figv@ctCreg#1(\v@lXa,\v@lYa,\v@lZa)}
\ctr@ld@f\def\det@rm#1[#2,#3]{{\Figg@tXY{#2}\Figg@tXYa{#3}%
    \delt@=\repdecn@mb{\v@lX}\v@lYa\advance\delt@-\repdecn@mb{\v@lY}\v@lXa%
    \global\result@t=\delt@}#1=\result@t}
\ctr@ld@f\def\det@rmC#1(#2,#3,#4,#5){{\global\result@t=\repdecn@mb{#2}#5%
    \global\advance\result@t-\repdecn@mb{#3}#4}#1=\result@t}
\ctr@ld@f\def\getredf@ctDD#1(#2,#3){{\maxim@m{\v@lXa}{-#2}{#2}\maxim@m{\v@lYa}{-#3}{#3}%
    \maxim@m{\v@lXa}{\v@lXa}{\v@lYa}% \v@lXa = ||X||inf
    \ifdim\v@lXa>\@xci pt\divide\v@lXa\@xci%
    \p@rtentiere{\p@rtent}{\v@lXa}\advance\p@rtent\@ne\else\p@rtent=\@ne\fi%
    \global\result@tent=\p@rtent}#1=\result@tent\ignorespaces}
\ctr@ld@f\def\getredf@ctTD#1(#2,#3,#4){{\maxim@m{\v@lXa}{-#2}{#2}\maxim@m{\v@lYa}{-#3}{#3}%
    \maxim@m{\v@lZa}{-#4}{#4}\maxim@m{\v@lXa}{\v@lXa}{\v@lYa}%
    \maxim@m{\v@lXa}{\v@lXa}{\v@lZa}% \v@lXa = ||X||inf
    \ifdim\v@lXa>\@lxxiv pt\divide\v@lXa\@lxxiv%
    \p@rtentiere{\p@rtent}{\v@lXa}\advance\p@rtent\@ne\else\p@rtent=\@ne\fi%
    \global\result@tent=\p@rtent}#1=\result@tent\ignorespaces}
\ctr@ln@m\getredf@ctB
\ctr@ld@f\def\getredf@ctBDD#1{\getredf@ctDD#1(\v@lX,\v@lY)}
\ctr@ld@f\def\getredf@ctBTD#1{\getredf@ctTD#1(\v@lX,\v@lY,\v@lZ)}
\ctr@ld@f\def\FigptintercircB@zDD#1:#2:#3,#4[#5,#6,#7,#8]{{\s@uvc@ntr@l\et@tfigptintercircB@zDD%
    \setc@ntr@l{2}\figvectPDD-1[#5,#8]\Figg@tXY{-1}\getredf@ctDD\f@ctech(\v@lX,\v@lY)%
    \mili@u=#4\unit@\divide\mili@u\f@ctech\c@rre=\repdecn@mb{\mili@u}\mili@u%
    \figptBezierDD-5::#3[#5,#6,#7,#8]%
    \v@lmin=#3\p@\v@lmax=\v@lmin\advance\v@lmax0.1\p@%
    \loop\edef\T@{\repdecn@mb{\v@lmax}}\figptBezierDD-2::\T@[#5,#6,#7,#8]%
    \figvectPDD-1[-5,-2]\n@rmeucCDD{\delt@}{-1}\ifdim\delt@<\c@rre\v@lmin=\v@lmax%
    \advance\v@lmax0.1\p@\repeat%
    \loop\mili@u=\v@lmin\advance\mili@u\v@lmax%
    \divide\mili@u\tw@\edef\T@{\repdecn@mb{\mili@u}}\figptBezierDD-2::\T@[#5,#6,#7,#8]%
    \figvectPDD-1[-5,-2]\n@rmeucCDD{\delt@}{-1}\ifdim\delt@>\c@rre\v@lmax=\mili@u%
    \else\v@lmin=\mili@u\fi\v@leur=\v@lmax\advance\v@leur-\v@lmin%
    \ifdim\v@leur>\epsil@n\repeat\figptcopyDD#1:#2/-2/%
    \resetc@ntr@l\et@tfigptintercircB@zDD}\ignorespaces}
\ctr@ln@m\figptinterlines
\ctr@ld@f\def\inters@cDD#1:#2[#3,#4;#5,#6]{{\s@uvc@ntr@l\et@tinters@cDD%
    \setc@ntr@l{2}\vecunit@{-1}{#4}\vecunit@{-2}{#6}%
    \Figg@tXY{-1}\setc@ntr@l{1}\Figg@tXYa{#3}%
    \edef\A@{\repdecn@mb{\v@lX}}\edef\B@{\repdecn@mb{\v@lY}}%
    \v@lmin=\B@\v@lXa\advance\v@lmin-\A@\v@lYa%
    \Figg@tXYa{#5}\setc@ntr@l{2}\Figg@tXY{-2}%
    \edef\C@{\repdecn@mb{\v@lX}}\edef\D@{\repdecn@mb{\v@lY}}%
    \v@lmax=\D@\v@lXa\advance\v@lmax-\C@\v@lYa%
    \delt@=\A@\v@lY\advance\delt@-\B@\v@lX%
    \invers@{\v@leur}{\delt@}\edef\v@ldelta{\repdecn@mb{\v@leur}}%
    \v@lXa=\A@\v@lmax\advance\v@lXa-\C@\v@lmin%
    \v@lYa=\B@\v@lmax\advance\v@lYa-\D@\v@lmin%
    \v@lXa=\v@ldelta\v@lXa\v@lYa=\v@ldelta\v@lYa%
    \setc@ntr@l{1}\Figp@intregDD#1:{#2}(\v@lXa,\v@lYa)%
    \resetc@ntr@l\et@tinters@cDD}\ignorespaces}
\ctr@ld@f\def\inters@cTD#1:#2[#3,#4;#5,#6]{{\s@uvc@ntr@l\et@tinters@cTD%
    \setc@ntr@l{2}\figvectNVTD-1[#4,#6]\figvectNVTD-2[#6,-1]\figvectPTD-1[#3,#5]%
    \r@pPSTD\v@leur[-2,-1,#4]\edef\v@lcoef{\repdecn@mb{\v@leur}}%
    \figpttraTD#1:{#2}=#3/\v@lcoef,#4/\resetc@ntr@l\et@tinters@cTD}\ignorespaces}
\ctr@ld@f\def\r@pPSTD#1[#2,#3,#4]{{\Figg@tXY{#2}\edef\Xu@{\repdecn@mb{\v@lX}}%
    \edef\Yu@{\repdecn@mb{\v@lY}}\edef\Zu@{\repdecn@mb{\v@lZ}}%
    \Figg@tXY{#3}\v@lmin=\Xu@\v@lX\advance\v@lmin\Yu@\v@lY\advance\v@lmin\Zu@\v@lZ%
    \Figg@tXY{#4}\v@lmax=\Xu@\v@lX\advance\v@lmax\Yu@\v@lY\advance\v@lmax\Zu@\v@lZ%
    \invers@{\v@leur}{\v@lmax}\global\result@t=\repdecn@mb{\v@leur}\v@lmin}%
    #1=\result@t}
\ctr@ln@m\n@rminf
\ctr@ld@f\def\n@rminfDD#1#2{{\Figg@tXY{#2}\maxim@m{\v@lX}{\v@lX}{-\v@lX}%
    \maxim@m{\v@lY}{\v@lY}{-\v@lY}\maxim@m{\global\result@t}{\v@lX}{\v@lY}}%
    #1=\result@t}
\ctr@ld@f\def\n@rminfTD#1#2{{\Figg@tXY{#2}\maxim@m{\v@lX}{\v@lX}{-\v@lX}%
    \maxim@m{\v@lY}{\v@lY}{-\v@lY}\maxim@m{\v@lZ}{\v@lZ}{-\v@lZ}%
    \maxim@m{\v@lX}{\v@lX}{\v@lY}\maxim@m{\global\result@t}{\v@lX}{\v@lZ}}%
    #1=\result@t}
\ctr@ln@m\n@rmeucC
\ctr@ld@f\def\n@rmeucCDD#1#2{\Figg@tXY{#2}\divide\v@lX\f@ctech\divide\v@lY\f@ctech%
    #1=\repdecn@mb{\v@lX}\v@lX\v@lX=\repdecn@mb{\v@lY}\v@lY\advance#1\v@lX}
\ctr@ld@f\def\n@rmeucCTD#1#2{\Figg@tXY{#2}%
    \divide\v@lX\f@ctech\divide\v@lY\f@ctech\divide\v@lZ\f@ctech%
    #1=\repdecn@mb{\v@lX}\v@lX\v@lX=\repdecn@mb{\v@lY}\v@lY\advance#1\v@lX%
    \v@lX=\repdecn@mb{\v@lZ}\v@lZ\advance#1\v@lX}
\ctr@ln@m\n@rmeucSV
\ctr@ld@f\def\n@rmeucSVDD#1#2{{\Figg@tXY{#2}%
    \v@lXa=\repdecn@mb{\v@lX}\v@lX\v@lYa=\repdecn@mb{\v@lY}\v@lY%
    \advance\v@lXa\v@lYa\sqrt@{\global\result@t}{\v@lXa}}#1=\result@t}
\ctr@ld@f\def\n@rmeucSVTD#1#2{{\Figg@tXY{#2}\v@lXa=\repdecn@mb{\v@lX}\v@lX%
    \v@lYa=\repdecn@mb{\v@lY}\v@lY\v@lZa=\repdecn@mb{\v@lZ}\v@lZ%
    \advance\v@lXa\v@lYa\advance\v@lXa\v@lZa\sqrt@{\global\result@t}{\v@lXa}}#1=\result@t}
\ctr@ln@m\n@rmeuc
\ctr@ld@f\def\n@rmeucDD#1#2{{\Figg@tXY{#2}\getredf@ctDD\f@ctech(\v@lX,\v@lY)%
    \divide\v@lX\f@ctech\divide\v@lY\f@ctech%
    \v@lXa=\repdecn@mb{\v@lX}\v@lX\v@lYa=\repdecn@mb{\v@lY}\v@lY%
    \advance\v@lXa\v@lYa\sqrt@{\global\result@t}{\v@lXa}%
    \global\multiply\result@t\f@ctech}#1=\result@t}
\ctr@ld@f\def\n@rmeucTD#1#2{{\Figg@tXY{#2}\getredf@ctTD\f@ctech(\v@lX,\v@lY,\v@lZ)%
    \divide\v@lX\f@ctech\divide\v@lY\f@ctech\divide\v@lZ\f@ctech%
    \v@lXa=\repdecn@mb{\v@lX}\v@lX%
    \v@lYa=\repdecn@mb{\v@lY}\v@lY\v@lZa=\repdecn@mb{\v@lZ}\v@lZ%
    \advance\v@lXa\v@lYa\advance\v@lXa\v@lZa\sqrt@{\global\result@t}{\v@lXa}%
    \global\multiply\result@t\f@ctech}#1=\result@t}
\ctr@ln@m\vecunit@
\ctr@ld@f\def\vecunit@DD#1#2{{\Figg@tXY{#2}\getredf@ctDD\f@ctech(\v@lX,\v@lY)%
    \divide\v@lX\f@ctech\divide\v@lY\f@ctech%
    \Figv@ctCreg#1(\v@lX,\v@lY)\n@rmeucSV{\v@lYa}{#1}%
    \invers@{\v@lXa}{\v@lYa}\edef\v@lv@lXa{\repdecn@mb{\v@lXa}}%
    \v@lX=\v@lv@lXa\v@lX\v@lY=\v@lv@lXa\v@lY%
    \Figv@ctCreg#1(\v@lX,\v@lY)\multiply\v@lYa\f@ctech\global\result@t=\v@lYa}}
\ctr@ld@f\def\vecunit@TD#1#2{{\Figg@tXY{#2}\getredf@ctTD\f@ctech(\v@lX,\v@lY,\v@lZ)%
    \divide\v@lX\f@ctech\divide\v@lY\f@ctech\divide\v@lZ\f@ctech%
    \Figv@ctCreg#1(\v@lX,\v@lY,\v@lZ)\n@rmeucSV{\v@lYa}{#1}%
    \invers@{\v@lXa}{\v@lYa}\edef\v@lv@lXa{\repdecn@mb{\v@lXa}}%
    \v@lX=\v@lv@lXa\v@lX\v@lY=\v@lv@lXa\v@lY\v@lZ=\v@lv@lXa\v@lZ%
    \Figv@ctCreg#1(\v@lX,\v@lY,\v@lZ)\multiply\v@lYa\f@ctech\global\result@t=\v@lYa}}
\ctr@ld@f\def\vecunitC@TD[#1,#2]{\Figg@tXYa{#1}\Figg@tXY{#2}%
    \advance\v@lX-\v@lXa\advance\v@lY-\v@lYa\advance\v@lZ-\v@lZa\c@lvecunitTD}
\ctr@ld@f\def\vecunitCV@TD#1{\Figg@tXY{#1}\c@lvecunitTD}
\ctr@ld@f\def\c@lvecunitTD{\getredf@ctTD\f@ctech(\v@lX,\v@lY,\v@lZ)%
    \divide\v@lX\f@ctech\divide\v@lY\f@ctech\divide\v@lZ\f@ctech%
    \v@lXa=\repdecn@mb{\v@lX}\v@lX%
    \v@lYa=\repdecn@mb{\v@lY}\v@lY\v@lZa=\repdecn@mb{\v@lZ}\v@lZ%
    \advance\v@lXa\v@lYa\advance\v@lXa\v@lZa\sqrt@{\v@lYa}{\v@lXa}%
    \invers@{\v@lXa}{\v@lYa}\edef\v@lv@lXa{\repdecn@mb{\v@lXa}}%
    \v@lX=\v@lv@lXa\v@lX\v@lY=\v@lv@lXa\v@lY\v@lZ=\v@lv@lXa\v@lZ}
\ctr@ln@m\figgetangle
\ctr@ld@f\def\figgetangleDD#1[#2,#3,#4]{\ifGR@cri{\s@uvc@ntr@l\et@tfiggetangleDD\setc@ntr@l{2}%
    \figvectPDD-1[#2,#3]\figvectPDD-2[#2,#4]\vecunit@{-1}{-1}%
    \c@lproscalDD\delt@[-2,-1]\figvectNVDD-1[-1]\c@lproscalDD\v@leur[-2,-1]%
    \arct@n\v@lmax(\delt@,\v@leur)\v@lmax=\rdT@deg\v@lmax%
    \ifdim\v@lmax<\z@\advance\v@lmax\DePI@deg\fi\xdef#1{\repdecn@mb{\v@lmax}}%
    \resetc@ntr@l\et@tfiggetangleDD}\ignorespaces\fi}
\ctr@ld@f\def\figgetangleTD#1[#2,#3,#4,#5]{\ifGR@cri{\s@uvc@ntr@l\et@tfiggetangleTD\setc@ntr@l{2}%
    \figvectPTD-1[#2,#3]\figvectPTD-2[#2,#5]\figvectNVTD-3[-1,-2]%
    \figvectPTD-2[#2,#4]\figvectNVTD-4[-3,-1]%
    \vecunit@{-1}{-1}\c@lproscalTD\delt@[-2,-1]\c@lproscalTD\v@leur[-2,-4]%
    \arct@n\v@lmax(\delt@,\v@leur)\v@lmax=\rdT@deg\v@lmax%
    \ifdim\v@lmax<\z@\advance\v@lmax\DePI@deg\fi\xdef#1{\repdecn@mb{\v@lmax}}%
    \resetc@ntr@l\et@tfiggetangleTD}\ignorespaces\fi}    
\ctr@ld@f\def\figgetdist#1[#2,#3]{\ifGR@cri{\s@uvc@ntr@l\et@tfiggetdist\setc@ntr@l{2}%
    \figvectP-1[#2,#3]\n@rmeuc{\v@lX}{-1}\v@lX=\ptT@unit@\v@lX\xdef#1{\repdecn@mb{\v@lX}}%
    \resetc@ntr@l\et@tfiggetdist}\ignorespaces\fi}
\ctr@ld@f\def\figget#1=#2[#3]{\keln@mun#1|%
    \def\n@mref{a}\ifx\l@debut\n@mref\figgetangle#2[#3]\else% angle
    \def\n@mref{d}\ifx\l@debut\n@mref\figgetdist#2[#3]\else% distance
    \W@rnmeskwd{figget}{#1}\fi\fi\ignorespaces}
\ctr@ld@f\def\Figg@tT#1{\c@ntr@lnum{#1}%
    {\expandafter\expandafter\expandafter\extr@ctT\csname\objc@de\endcsname:%
     \ifnum\B@@ltxt=\z@\ptn@me{#1}\else\csname\objc@de T\endcsname\fi}}
\ctr@ld@f\def\extr@ctT#1,#2,#3/#4:{\def\B@@ltxt{#3}}
\ctr@ld@f\def\Figg@tXY#1{\c@ntr@lnum{#1}%
    \expandafter\expandafter\expandafter\extr@ctC\csname\objc@de\endcsname:}
\ctr@ln@m\extr@ctC
\ctr@ld@f\def\extr@ctCDD#1/#2,#3,#4:{\v@lX=#2\v@lY=#3}
\ctr@ld@f\def\extr@ctCTD#1/#2,#3,#4:{\v@lX=#2\v@lY=#3\v@lZ=#4}
\ctr@ld@f\def\Figg@tXYa#1{\c@ntr@lnum{#1}%
    \expandafter\expandafter\expandafter\extr@ctCa\csname\objc@de\endcsname:}
\ctr@ln@m\extr@ctCa
\ctr@ld@f\def\extr@ctCaDD#1/#2,#3,#4:{\v@lXa=#2\v@lYa=#3}
\ctr@ld@f\def\extr@ctCaTD#1/#2,#3,#4:{\v@lXa=#2\v@lYa=#3\v@lZa=#4}
\ctr@ln@m\t@xt@
\ctr@ld@f\def\figinit#1{\t@stc@tcodech@nge\initpr@lim\Figinit@#1,:\initpss@ttings\ignorespaces}
\ctr@ld@f\def\Figinit@#1,#2:{\setunit@{#1}\def\t@xt@{#2}\ifx\t@xt@\empty\else\Figinit@@#2:\fi}
\ctr@ld@f\def\Figinit@@#1#2:{\if#12 \else\Figs@tproj{#1}\initTD@\fi}
\ctr@ln@w{newif}\ifTr@isDim
\ctr@ld@f\def\UnD@fined{UNDEFINED}
\ctr@ln@m\@utoFN
\ctr@ln@m\@utoFInDone
\ctr@ln@m\disob@unit
\ctr@ld@f\def\initpr@lim{\initb@undb@x\figsetmark{}\figsetptname{$A_{##1}$}\def\Sc@leFact{1}%
    \initDD@\figsetroundcoord{yes}\GR@critrue\expandafter\setupd@te\D@FTupdate:%
    \edef\disob@unit{\UnD@fined}\edef\t@rgetpt{\UnD@fined}\gdef\@utoFInDone{1}\gdef\@utoFN{0}}
\ctr@ld@f\def\initDD@{\Tr@isDimfalse%
    \ifPDFm@ke%
     \let\Ps@rcerc=\Ps@rcercBz%
     \let\Ps@rell=\Ps@rellBz%
    \fi
    \let\c@lDCUn=\c@lDCUnDD%
    \let\c@lDCDeux=\c@lDCDeuxDD%
    \let\c@ldefproj=\relax%
    \let\c@lproscal=\c@lproscalDD%
    \let\c@lprojSP=\relax%
    \let\extr@ctC=\extr@ctCDD%
    \let\extr@ctCa=\extr@ctCaDD%
    \let\extr@ctCF=\extr@ctCFDD%
    \let\Figp@intreg=\Figp@intregDD%
    \let\Figpts@xes=\Figpts@xesDD%
    \let\getredf@ctB=\getredf@ctBDD%
    \let\n@rmeucSV=\n@rmeucSVDD\let\n@rmeuc=\n@rmeucDD\let\n@rmeucC\n@rmeucCDD\let\n@rminf=\n@rminfDD%
    \let\pr@dMatV=\pr@dMatVDD%
    \let\Q@@xes=\Q@@xesDD%
    \let\vecunit@=\vecunit@DD%
    \let\figcoord=\figcoordDD%
    \let\figgetangle=\figgetangleDD%
    \let\figpt=\figptDD%
    \let\figptBezier=\figptBezierDD%
    \let\figptbary=\figptbaryDD%
    \let\figptcirc=\figptcircDD%
    \let\figptcircumcenter=\figptcircumcenterDD%
    \let\figptcopy=\figptcopyDD%
    \let\figptcurvcenter=\figptcurvcenterDD%
    \let\figptell=\figptellDD%
    \let\figptendnormal=\figptendnormalDD%
    \let\figptinterlineplane=\figptinterlineplaneDD%
    \let\figptinterlines=\inters@cDD%
    \let\figptorthocenter=\figptorthocenterDD%
    \let\figptorthoprojline=\figptorthoprojlineDD%
    \let\figptorthoprojplane=\figptorthoprojplaneDD%
    \let\figptrot=\figptrotDD%
    \let\figptscontrol=\figptscontrolDD%
    \let\figptsintercirc=\figptsintercircDD%
    \let\figptsinterlinell=\figptsinterlinellDD%
    \let\figptsorthoprojline=\figptsorthoprojlineDD%
    \let\figptorthoprojplane=\figptorthoprojplaneDD%
    \let\figptsrot=\figptsrotDD%
    \let\figptssym=\figptssymDD%
    \let\figptstra=\figptstraDD%
    \let\figptsym=\figptsymDD%
    \let\figpttraC=\figpttraCDD%
    \let\figpttra=\figpttraDD%
    \let\figptvisilimSL=\figptvisilimSLDD%
    \let\figsetobdist=\figsetobdistDD%
    \let\figsettarget=\figsettargetDD%
    \let\figsetview=\figsetviewDD%
    \let\figvectDBezier=\figvectDBezierDD%
    \let\figvectN=\figvectNDD%
    \let\figvectNV=\figvectNVDD%
    \let\figvectP=\figvectPDD%
    \let\figvectU=\figvectUDD%
    \let\figdrawarccircP=\Q@arccircPDD%
    \let\figdrawarccirc=\Q@arccircDD%
    \let\figdrawarcell=\Q@arcellDD%
    \let\figdrawarcellPA=\Q@arcellPADD%
    \let\figdrawarrowBezier=\Q@arrowBezierDD%
    \let\figdrawarrowcircP=\Q@arrowcircPDD%
    \let\figdrawarrowcirc=\Q@arrowcircDD%
    \let\figdrawarrowhead=\Q@arrowheadDD%
    \let\figdrawarrow=\Q@arrowDD%
    \let\figdrawBezier=\Q@BezierDD%
    \let\figdrawcirc=\Q@circDD%
    \let\figdrawcurve=\Q@curveDD%
    \let\figdrawnormal=\Q@normalDD%
    }
\ctr@ld@f\def\initTD@{\Tr@isDimtrue\initb@undb@xTD\newt@rgetptfalse\newdis@bfalse%
    \let\c@lDCUn=\c@lDCUnTD%
    \let\c@lDCDeux=\c@lDCDeuxTD%
    \let\c@ldefproj=\c@ldefprojTD%
    \let\c@lproscal=\c@lproscalTD%
    \let\extr@ctC=\extr@ctCTD%
    \let\extr@ctCa=\extr@ctCaTD%
    \let\extr@ctCF=\extr@ctCFTD%
    \let\Figp@intreg=\Figp@intregTD%
    \let\Figpts@xes=\Figpts@xesTD%
    \let\getredf@ctB=\getredf@ctBTD%
    \let\n@rmeucSV=\n@rmeucSVTD\let\n@rmeuc=\n@rmeucTD\let\n@rmeucC\n@rmeucCTD\let\n@rminf=\n@rminfTD%
    \let\pr@dMatV=\pr@dMatVTD%
    \let\Q@@xes=\Q@@xesTD%
    \let\vecunit@=\vecunit@TD%
    \let\figcoord=\figcoordTD%
    \let\figgetangle=\figgetangleTD%
    \let\figpt=\figptTD%
    \let\figptBezier=\figptBezierTD%
    \let\figptbary=\figptbaryTD%
    \let\figptcirc=\figptcircTD%
    \let\figptcircumcenter=\figptcircumcenterTD%
    \let\figptcopy=\figptcopyTD%
    \let\figptcurvcenter=\figptcurvcenterTD%
    \let\figptinterlineplane=\figptinterlineplaneTD%
    \let\figptinterlines=\inters@cTD%
    \let\figptorthocenter=\figptorthocenterTD%
    \let\figptorthoprojline=\figptorthoprojlineTD%
    \let\figptorthoprojplane=\figptorthoprojplaneTD%
    \let\figptrot=\figptrotTD%
    \let\figptscontrol=\figptscontrolTD%
    \let\figptsintercirc=\figptsintercircTD%
    \let\figptsorthoprojline=\figptsorthoprojlineTD%
    \let\figptsorthoprojplane=\figptsorthoprojplaneTD%
    \let\figptsrot=\figptsrotTD%
    \let\figptssym=\figptssymTD%
    \let\figptstra=\figptstraTD%
    \let\figptsym=\figptsymTD%
    \let\figpttraC=\figpttraCTD%
    \let\figpttra=\figpttraTD%
    \let\figptvisilimSL=\figptvisilimSLTD%
    \let\figsetobdist=\figsetobdistTD%
    \let\figsettarget=\figsettargetTD%
    \let\figsetview=\figsetviewTD%
    \let\figvectDBezier=\figvectDBezierTD%
    \let\figvectN=\figvectNTD%
    \let\figvectNV=\figvectNVTD%
    \let\figvectP=\figvectPTD%
    \let\figvectU=\figvectUTD%
    \let\figdrawarccircP=\Q@arccircPTD%
    \let\figdrawarccirc=\Q@arccircTD%
    \let\figdrawarcell=\Q@arcellTD%
    \let\figdrawarcellPA=\Q@arcellPATD%
    \let\figdrawarrowBezier=\Q@arrowBezierTD%
    \let\figdrawarrowcircP=\Q@arrowcircPTD%
    \let\figdrawarrowcirc=\Q@arrowcircTD%
    \let\figdrawarrowhead=\Q@arrowheadTD%
    \let\figdrawarrow=\Q@arrowTD%
    \let\figdrawBezier=\Q@BezierTD%
    \let\figdrawcirc=\Q@circTD%
    \let\figdrawcurve=\Q@curveTD%
    }
\ctr@ld@f\def\un@v@ilable#1{\immediate\write16{*** The macro #1 is not available in the current context.}}
\ctr@ld@f\def\figinsert#1{{\def\t@xt@{#1}\relax%
    \ifx\t@xt@\empty\ifnum\@utoFInDone>\z@\Figinsert@\DefGIfilen@me,:\fi%
    \else\expandafter\FiginsertNu@#1 :\fi}\ignorespaces}
\ctr@ld@f\def\FiginsertNu@#1 #2:{\def\t@xt@{#1}\relax\ifx\t@xt@\empty\def\t@xt@{#2}%
    \ifx\t@xt@\empty\ifnum\@utoFInDone>\z@\Figinsert@\DefGIfilen@me,:\fi%
    \else\FiginsertNu@#2:\fi\else\expandafter\FiginsertNd@#1 #2:\fi}
\ctr@ld@f\def\FiginsertNd@#1#2:{\ifcat#1a\Figinsert@#1#2,:\else%
    \ifnum\@utoFInDone>\z@\Figinsert@\DefGIfilen@me,#1#2,:\fi\fi}
\ctr@ln@m\Sc@leFact
\ctr@ld@f\def\Figinsert@#1,#2:{\def\t@xt@{#2}\ifx\t@xt@\empty\xdef\Sc@leFact{1}\else%
    \X@rgdeux@#2\xdef\Sc@leFact{\@rgdeux}\fi%
    \Figdisc@rdLTS{#1}{\t@xt@}\@psfgetbb{\t@xt@}%
    \v@lX=\@psfllx\p@\v@lX=\ptpsT@pt\v@lX\v@lX=\Sc@leFact\v@lX%
    \v@lY=\@psflly\p@\v@lY=\ptpsT@pt\v@lY\v@lY=\Sc@leFact\v@lY%
    \b@undb@x{\v@lX}{\v@lY}%
    \v@lX=\@psfurx\p@\v@lX=\ptpsT@pt\v@lX\v@lX=\Sc@leFact\v@lX%
    \v@lY=\@psfury\p@\v@lY=\ptpsT@pt\v@lY\v@lY=\Sc@leFact\v@lY%
    \b@undb@x{\v@lX}{\v@lY}%
    \ifPDFm@ke\Figinclud@PDF{\t@xt@}{\Sc@leFact}\else%
    \v@lX=\c@nt pt\v@lX=\Sc@leFact\v@lX\edef\F@ct{\repdecn@mb{\v@lX}}%
    \ifx\TeXturesonMacOSltX\special{postscriptfile #1 vscale=\F@ct\space hscale=\F@ct}%
    \else\includegraphics{#1}\fi\fi%
    \message{[\t@xt@]}\ignorespaces}
\ctr@ld@f\def\Figdisc@rdLTS#1#2{\expandafter\Figdisc@rdLTS@#1 :#2}
\ctr@ld@f\def\Figdisc@rdLTS@#1 #2:#3{\def#3{#1}\relax\ifx#3\empty\expandafter\Figdisc@rdLTS@#2:#3\fi}
\ctr@ld@f\def\figinsertE#1{\FiginsertE@#1,:\ignorespaces}
\ctr@ld@f\def\FiginsertE@#1,#2:{{\def\t@xt@{#2}\ifx\t@xt@\empty\xdef\Sc@leFact{1}\else%
    \X@rgdeux@#2\xdef\Sc@leFact{\@rgdeux}\fi%
    \Figdisc@rdLTS{#1}{\t@xt@}\pdfximage{\t@xt@}%
    \setbox\Gb@x=\hbox{\pdfrefximage\pdflastximage}%
    \v@lX=\z@\v@lY=-\Sc@leFact\dp\Gb@x\b@undb@x{\v@lX}{\v@lY}%
    \advance\v@lX\Sc@leFact\wd\Gb@x\advance\v@lY\Sc@leFact\dp\Gb@x%
    \advance\v@lY\Sc@leFact\ht\Gb@x\b@undb@x{\v@lX}{\v@lY}%
    \v@lX=\Sc@leFact\wd\Gb@x\pdfximage width \v@lX {\t@xt@}%
    \rlap{\pdfrefximage\pdflastximage}\message{[\t@xt@]}}\ignorespaces}
\ctr@ld@f\def\X@rgdeux@#1,{\edef\@rgdeux{#1}}
\ctr@ln@m\figpt
\ctr@ld@f\def\figptDD#1:#2(#3,#4){\ifGR@cri\c@ntr@lnum{#1}%
    {\v@lX=#3\unit@\v@lY=#4\unit@\Fig@dmpt{#2}{\z@}}\ignorespaces\fi}
\ctr@ld@f\def\Fig@dmpt#1#2{\def\t@xt@{#1}\ifx\t@xt@\empty\def\B@@ltxt{\z@}%
    \else\expandafter\gdef\csname\objc@de T\endcsname{#1}\def\B@@ltxt{\@ne}\fi%
    \expandafter\xdef\csname\objc@de\endcsname{\ifitis@vect@r\C@dCl@svect%
    \else\C@dCl@spt\fi,\z@,\B@@ltxt/\the\v@lX,\the\v@lY,#2}}
\ctr@ld@f\def\C@dCl@spt{P}
\ctr@ld@f\def\C@dCl@svect{V}
\ctr@ln@m\c@@rdYZ
\ctr@ln@m\c@@rdY
\ctr@ld@f\def\figptTD#1:#2(#3,#4){\ifGR@cri\c@ntr@lnum{#1}%
    \def\c@@rdYZ{#4,0,0}\extrairelepremi@r\c@@rdY\de\c@@rdYZ%
    \extrairelepremi@r\c@@rdZ\de\c@@rdYZ%
    {\v@lX=#3\unit@\v@lY=\c@@rdY\unit@\v@lZ=\c@@rdZ\unit@\Fig@dmpt{#2}{\the\v@lZ}%
    \b@undb@xTD{\v@lX}{\v@lY}{\v@lZ}}\ignorespaces\fi}
\ctr@ln@m\Figp@intreg
\ctr@ld@f\def\Figp@intregDD#1:#2(#3,#4){\c@ntr@lnum{#1}%
    {\result@t=#4\v@lX=#3\v@lY=\result@t\Fig@dmpt{#2}{\z@}}\ignorespaces}
\ctr@ld@f\def\Figp@intregTD#1:#2(#3,#4){\c@ntr@lnum{#1}%
    \def\c@@rdYZ{#4,\z@,\z@}\extrairelepremi@r\c@@rdY\de\c@@rdYZ%
    \extrairelepremi@r\c@@rdZ\de\c@@rdYZ%
    {\v@lX=#3\v@lY=\c@@rdY\v@lZ=\c@@rdZ\Fig@dmpt{#2}{\the\v@lZ}%
    \b@undb@xTD{\v@lX}{\v@lY}{\v@lZ}}\ignorespaces}
\ctr@ln@m\figptBezier
\ctr@ld@f\def\figptBezierDD#1:#2:#3[#4,#5,#6,#7]{\ifGR@cri{\s@uvc@ntr@l\et@tfigptBezierDD%
    \FigptBezier@#3[#4,#5,#6,#7]\Figp@intregDD#1:{#2}(\v@lX,\v@lY)%
    \resetc@ntr@l\et@tfigptBezierDD}\ignorespaces\fi}
\ctr@ld@f\def\figptBezierTD#1:#2:#3[#4,#5,#6,#7]{\ifGR@cri{\s@uvc@ntr@l\et@tfigptBezierTD%
    \FigptBezier@#3[#4,#5,#6,#7]\Figp@intregTD#1:{#2}(\v@lX,\v@lY,\v@lZ)%
    \resetc@ntr@l\et@tfigptBezierTD}\ignorespaces\fi}
\ctr@ld@f\def\FigptBezier@#1[#2,#3,#4,#5]{\setc@ntr@l{2}%
    \edef\T@{#1}\v@leur=\p@\advance\v@leur-#1pt\edef\UNmT@{\repdecn@mb{\v@leur}}%
    \figptcopy-4:/#2/\figptcopy-3:/#3/\figptcopy-2:/#4/\figptcopy-1:/#5/%
    \l@mbd@un=-4 \l@mbd@de=-\thr@@\p@rtent=\m@ne\c@lDecast%
    \l@mbd@un=-4 \l@mbd@de=-\thr@@\p@rtent=-\tw@\c@lDecast%
    \l@mbd@un=-4 \l@mbd@de=-\thr@@\p@rtent=-\thr@@\c@lDecast\Figg@tXY{-4}}
\ctr@ln@m\c@lDCUn
\ctr@ld@f\def\c@lDCUnDD#1#2{\Figg@tXY{#1}\v@lX=\UNmT@\v@lX\v@lY=\UNmT@\v@lY%
    \Figg@tXYa{#2}\advance\v@lX\T@\v@lXa\advance\v@lY\T@\v@lYa%
    \Figp@intregDD#1:(\v@lX,\v@lY)}
\ctr@ld@f\def\c@lDCUnTD#1#2{\Figg@tXY{#1}\v@lX=\UNmT@\v@lX\v@lY=\UNmT@\v@lY\v@lZ=\UNmT@\v@lZ%
    \Figg@tXYa{#2}\advance\v@lX\T@\v@lXa\advance\v@lY\T@\v@lYa\advance\v@lZ\T@\v@lZa%
    \Figp@intregTD#1:(\v@lX,\v@lY,\v@lZ)}
\ctr@ld@f\def\c@lDecast{\relax\ifnum\l@mbd@un<\p@rtent\c@lDCUn{\l@mbd@un}{\l@mbd@de}%
    \advance\l@mbd@un\@ne\advance\l@mbd@de\@ne\c@lDecast\fi}
\ctr@ld@f\def\figptmap#1:#2=#3/#4/#5/{\ifGR@cri{\s@uvc@ntr@l\et@tfigptmap%
    \setc@ntr@l{2}\figvectP-1[#4,#3]\Figg@tXY{-1}%
    \pr@dMatV/#5/\figpttra#1:{#2}=#4/1,-1/%
    \resetc@ntr@l\et@tfigptmap}\ignorespaces\fi}
\ctr@ln@m\pr@dMatV
\ctr@ld@f\def\pr@dMatVDD/#1,#2;#3,#4/{\v@lXa=#1\v@lX\advance\v@lXa#2\v@lY%
    \v@lYa=#3\v@lX\advance\v@lYa#4\v@lY\Figv@ctCreg-1(\v@lXa,\v@lYa)}
\ctr@ld@f\def\pr@dMatVTD/#1,#2,#3;#4,#5,#6;#7,#8,#9/{%
    \v@lXa=#1\v@lX\advance\v@lXa#2\v@lY\advance\v@lXa#3\v@lZ%
    \v@lYa=#4\v@lX\advance\v@lYa#5\v@lY\advance\v@lYa#6\v@lZ%
    \v@lZa=#7\v@lX\advance\v@lZa#8\v@lY\advance\v@lZa#9\v@lZ%
    \Figv@ctCreg-1(\v@lXa,\v@lYa,\v@lZa)}
\ctr@ln@m\figptbary
\ctr@ld@f\def\figptbaryDD#1:#2[#3;#4]{\ifGR@cri{\edef\list@num{#3}\extrairelepremi@r\p@int\de\list@num%
    \s@mme=\z@\@ecfor\c@ef:=#4\do{\advance\s@mme\c@ef}%
    \edef\listec@ef{#4,0}\extrairelepremi@r\c@ef\de\listec@ef%
    \Figg@tXY{\p@int}\divide\v@lX\s@mme\divide\v@lY\s@mme%
    \multiply\v@lX\c@ef\multiply\v@lY\c@ef%
    \@ecfor\p@int:=\list@num\do{\extrairelepremi@r\c@ef\de\listec@ef%
           \Figg@tXYa{\p@int}\divide\v@lXa\s@mme\divide\v@lYa\s@mme%
           \multiply\v@lXa\c@ef\multiply\v@lYa\c@ef%
           \advance\v@lX\v@lXa\advance\v@lY\v@lYa}%
    \Figp@intregDD#1:{#2}(\v@lX,\v@lY)}\ignorespaces\fi}
\ctr@ld@f\def\figptbaryTD#1:#2[#3;#4]{\ifGR@cri{\edef\list@num{#3}\extrairelepremi@r\p@int\de\list@num%
    \s@mme=\z@\@ecfor\c@ef:=#4\do{\advance\s@mme\c@ef}%
    \edef\listec@ef{#4,0}\extrairelepremi@r\c@ef\de\listec@ef%
    \Figg@tXY{\p@int}\divide\v@lX\s@mme\divide\v@lY\s@mme\divide\v@lZ\s@mme%
    \multiply\v@lX\c@ef\multiply\v@lY\c@ef\multiply\v@lZ\c@ef%
    \@ecfor\p@int:=\list@num\do{\extrairelepremi@r\c@ef\de\listec@ef%
           \Figg@tXYa{\p@int}\divide\v@lXa\s@mme\divide\v@lYa\s@mme\divide\v@lZa\s@mme%
           \multiply\v@lXa\c@ef\multiply\v@lYa\c@ef\multiply\v@lZa\c@ef%
           \advance\v@lX\v@lXa\advance\v@lY\v@lYa\advance\v@lZ\v@lZa}%
    \Figp@intregTD#1:{#2}(\v@lX,\v@lY,\v@lZ)}\ignorespaces\fi}
\ctr@ld@f\def\figptbaryR#1:#2[#3;#4]{\ifGR@cri{%
    \v@leur=\z@\@ecfor\c@ef:=#4\do{\maxim@m{\v@lmax}{\c@ef pt}{-\c@ef pt}%
    \ifdim\v@lmax>\v@leur\v@leur=\v@lmax\fi}%
    \ifdim\v@leur<\p@\f@ctech=\@M\else\ifdim\v@leur<\t@n\p@\f@ctech=\@m\else%
    \ifdim\v@leur<\c@nt\p@\f@ctech=\c@nt\else\ifdim\v@leur<\@m\p@\f@ctech=\t@n\else%
    \f@ctech=\@ne\fi\fi\fi\fi%
    \def\listec@ef{0}%
    \@ecfor\c@ef:=#4\do{\sc@lec@nvRI{\c@ef pt}\edef\listec@ef{\listec@ef,\the\s@mme}}%
    \extrairelepremi@r\c@ef\de\listec@ef\figptbary#1:#2[#3;\listec@ef]}\ignorespaces\fi}
\ctr@ld@f\def\sc@lec@nvRI#1{\v@leur=#1\p@rtentiere{\s@mme}{\v@leur}\advance\v@leur-\s@mme\p@%
    \multiply\v@leur\f@ctech\p@rtentiere{\p@rtent}{\v@leur}%
    \multiply\s@mme\f@ctech\advance\s@mme\p@rtent}
\ctr@ln@m\figptcirc
\ctr@ld@f\def\figptcircDD#1:#2:#3;#4(#5){\ifGR@cri{\s@uvc@ntr@l\et@tfigptcircDD%
    \c@lptellDD#1:{#2}:#3;#4,#4(#5)\resetc@ntr@l\et@tfigptcircDD}\ignorespaces\fi}
\ctr@ld@f\def\figptcircTD#1:#2:#3,#4,#5;#6(#7){\ifGR@cri{\s@uvc@ntr@l\et@tfigptcircTD%
    \setc@ntr@l{2}\c@lExtAxes#3,#4,#5(#6)\figptellP#1:{#2}:#3,-4,-5(#7)%
    \resetc@ntr@l\et@tfigptcircTD}\ignorespaces\fi}
\ctr@ln@m\figptcircumcenter
\ctr@ld@f\def\figptcircumcenterDD#1:#2[#3,#4,#5]{\ifGR@cri{\s@uvc@ntr@l\et@tfigptcircumcenterDD%
    \setc@ntr@l{2}\figvectNDD-5[#3,#4]\figptbaryDD-3:[#3,#4;1,1]%
                  \figvectNDD-6[#4,#5]\figptbaryDD-4:[#4,#5;1,1]%
    \resetc@ntr@l{2}\inters@cDD#1:{#2}[-3,-5;-4,-6]%
    \resetc@ntr@l\et@tfigptcircumcenterDD}\ignorespaces\fi}
\ctr@ld@f\def\figptcircumcenterTD#1:#2[#3,#4,#5]{\ifGR@cri{\s@uvc@ntr@l\et@tfigptcircumcenterTD%
    \setc@ntr@l{2}\figvectNTD-1[#3,#4,#5]%
    \figvectPTD-3[#3,#4]\figvectNVTD-5[-1,-3]\figptbaryTD-3:[#3,#4;1,1]%
    \figvectPTD-4[#4,#5]\figvectNVTD-6[-1,-4]\figptbaryTD-4:[#4,#5;1,1]%
    \resetc@ntr@l{2}\inters@cTD#1:{#2}[-3,-5;-4,-6]%
    \resetc@ntr@l\et@tfigptcircumcenterTD}\ignorespaces\fi}
\ctr@ln@m\figptcopy
\ctr@ld@f\def\figptcopyDD#1:#2/#3/{\ifGR@cri{\Figg@tXY{#3}%
    \Figp@intregDD#1:{#2}(\v@lX,\v@lY)}\ignorespaces\fi}
\ctr@ld@f\def\figptcopyTD#1:#2/#3/{\ifGR@cri{\Figg@tXY{#3}%
    \Figp@intregTD#1:{#2}(\v@lX,\v@lY,\v@lZ)}\ignorespaces\fi}
\ctr@ln@m\figptcurvcenter
\ctr@ld@f\def\figptcurvcenterDD#1:#2:#3[#4,#5,#6,#7]{\ifGR@cri{\s@uvc@ntr@l\et@tfigptcurvcenterDD%
    \setc@ntr@l{2}\c@lcurvradDD#3[#4,#5,#6,#7]\edef\Sprim@{\repdecn@mb{\result@t}}%
    \figptBezierDD-1::#3[#4,#5,#6,#7]\figpttraDD#1:{#2}=-1/\Sprim@,-5/%
    \resetc@ntr@l\et@tfigptcurvcenterDD}\ignorespaces\fi}
\ctr@ld@f\def\figptcurvcenterTD#1:#2:#3[#4,#5,#6,#7]{\ifGR@cri{\s@uvc@ntr@l\et@tfigptcurvcenterTD%
    \setc@ntr@l{2}\figvectDBezierTD -5:1,#3[#4,#5,#6,#7]%
    \figvectDBezierTD -6:2,#3[#4,#5,#6,#7]\vecunit@TD{-5}{-5}%
    \edef\Sprim@{\repdecn@mb{\result@t}}\figvectNVTD-1[-6,-5]%
    \figvectNVTD-5[-5,-1]\c@lproscalTD\v@leur[-6,-5]%
    \invers@{\v@leur}{\v@leur}\v@leur=\Sprim@\v@leur\v@leur=\Sprim@\v@leur%
    \figptBezierTD-1::#3[#4,#5,#6,#7]\edef\Sprim@{\repdecn@mb{\v@leur}}%
    \figpttraTD#1:{#2}=-1/\Sprim@,-5/\resetc@ntr@l\et@tfigptcurvcenterTD}\ignorespaces\fi}
\ctr@ld@f\def\c@lcurvradDD#1[#2,#3,#4,#5]{{\figvectDBezierDD -5:1,#1[#2,#3,#4,#5]%
    \figvectDBezierDD -6:2,#1[#2,#3,#4,#5]\vecunit@DD{-5}{-5}%
    \edef\Sprim@{\repdecn@mb{\result@t}}\figvectNVDD-5[-5]\c@lproscalDD\v@leur[-6,-5]%
    \invers@{\v@leur}{\v@leur}\v@leur=\Sprim@\v@leur\v@leur=\Sprim@\v@leur%
    \global\result@t=\v@leur}}
\ctr@ln@m\figptell
\ctr@ld@f\def\figptellDD#1:#2:#3;#4,#5(#6,#7){\ifGR@cri{\s@uvc@ntr@l\et@tfigptell%
    \c@lptellDD#1::#3;#4,#5(#6)\figptrotDD#1:{#2}=#1/#3,#7/%
    \resetc@ntr@l\et@tfigptell}\ignorespaces\fi}
\ctr@ld@f\def\c@lptellDD#1:#2:#3;#4,#5(#6){\c@ssin{\C@}{\S@}{#6}\v@lmin=\C@ pt\v@lmax=\S@ pt%
    \v@lmin=#4\v@lmin\v@lmax=#5\v@lmax%
    \edef\Xc@mp{\repdecn@mb{\v@lmin}}\edef\Yc@mp{\repdecn@mb{\v@lmax}}%
    \setc@ntr@l{2}\figvectC-1(\Xc@mp,\Yc@mp)\figpttraDD#1:{#2}=#3/1,-1/}
\ctr@ld@f\def\figptellP#1:#2:#3,#4,#5(#6){\ifGR@cri{\s@uvc@ntr@l\et@tfigptellP%
    \setc@ntr@l{2}\figvectP-1[#3,#4]\figvectP-2[#3,#5]%
    \v@leur=#6pt\c@lptellP{#3}{-1}{-2}\figptcopy#1:{#2}/-3/%
    \resetc@ntr@l\et@tfigptellP}\ignorespaces\fi}
\ctr@ln@m\@ngle
\ctr@ld@f\def\c@lptellP#1#2#3{\edef\@ngle{\repdecn@mb\v@leur}\c@ssin{\C@}{\S@}{\@ngle}%
    \figpttra-3:=#1/\C@,#2/\figpttra-3:=-3/\S@,#3/}
\ctr@ln@m\figptendnormal
\ctr@ld@f\def\figptendnormalDD#1:#2:#3,#4[#5,#6]{\ifGR@cri{\s@uvc@ntr@l\et@tfigptendnormal%
    \Figg@tXYa{#5}\Figg@tXY{#6}%
    \advance\v@lX-\v@lXa\advance\v@lY-\v@lYa%
    \setc@ntr@l{2}\Figv@ctCreg-1(\v@lX,\v@lY)\vecunit@{-1}{-1}\Figg@tXY{-1}%
    \delt@=#3\unit@\maxim@m{\delt@}{\delt@}{-\delt@}\edef\l@ngueur{\repdecn@mb{\delt@}}%
    \v@lX=\l@ngueur\v@lX\v@lY=\l@ngueur\v@lY%
    \delt@=\p@\advance\delt@-#4pt\edef\l@ngueur{\repdecn@mb{\delt@}}%
    \figptbaryR-1:[#5,#6;#4,\l@ngueur]\Figg@tXYa{-1}%
    \advance\v@lXa\v@lY\advance\v@lYa-\v@lX%
    \setc@ntr@l{1}\Figp@intregDD#1:{#2}(\v@lXa,\v@lYa)\resetc@ntr@l\et@tfigptendnormal}%
    \ignorespaces\fi}
\ctr@ld@f\def\figptexcenter#1:#2[#3,#4,#5]{\ifGR@cri{\let@xte={-}% for computational macro
    \Figptexinsc@nter#1:#2[#3,#4,#5]}\ignorespaces\fi}
\ctr@ld@f\def\figptincenter#1:#2[#3,#4,#5]{\ifGR@cri{\let@xte={}% for computational macro
    \Figptexinsc@nter#1:#2[#3,#4,#5]}\ignorespaces\fi}
\ctr@ld@f% for compatibility with older versions
\ctr@ld@f\def\Figptexinsc@nter#1:#2[#3,#4,#5]{%
    \figgetdist\LA@[#4,#5]\figgetdist\LB@[#3,#5]\figgetdist\LC@[#3,#4]%
    \figptbaryR#1:{#2}[#3,#4,#5;\the\let@xte\LA@,\LB@,\LC@]}
\ctr@ln@m\figptinterlineplane
\ctr@ld@f\def\figptinterlineplaneDD{\un@v@ilable{figptinterlineplane}}
\ctr@ld@f\def\figptinterlineplaneTD#1:#2[#3,#4;#5,#6]{\ifGR@cri{\s@uvc@ntr@l\et@tfigptinterlineplane%
    \setc@ntr@l{2}\figvectPTD-1[#3,#5]\vecunit@TD{-2}{#6}%
    \r@pPSTD\v@leur[-2,-1,#4]\edef\v@lcoef{\repdecn@mb{\v@leur}}%
    \figpttraTD#1:{#2}=#3/\v@lcoef,#4/\resetc@ntr@l\et@tfigptinterlineplane}\ignorespaces\fi}
\ctr@ln@m\figptorthocenter
\ctr@ld@f\def\figptorthocenterDD#1:#2[#3,#4,#5]{\ifGR@cri{\s@uvc@ntr@l\et@tfigptorthocenterDD%
    \setc@ntr@l{2}\figvectNDD-3[#3,#4]\figvectNDD-4[#4,#5]%
    \resetc@ntr@l{2}\inters@cDD#1:{#2}[#5,-3;#3,-4]%
    \resetc@ntr@l\et@tfigptorthocenterDD}\ignorespaces\fi}
\ctr@ld@f\def\figptorthocenterTD#1:#2[#3,#4,#5]{\ifGR@cri{\s@uvc@ntr@l\et@tfigptorthocenterTD%
    \setc@ntr@l{2}\figvectNTD-1[#3,#4,#5]%
    \figvectPTD-2[#3,#4]\figvectNVTD-3[-1,-2]%
    \figvectPTD-2[#4,#5]\figvectNVTD-4[-1,-2]%
    \resetc@ntr@l{2}\inters@cTD#1:{#2}[#5,-3;#3,-4]%
    \resetc@ntr@l\et@tfigptorthocenterTD}\ignorespaces\fi}
\ctr@ln@m\figptorthoprojline
\ctr@ld@f\def\figptorthoprojlineDD#1:#2=#3/#4,#5/{\ifGR@cri{\s@uvc@ntr@l\et@tfigptorthoprojlineDD%
    \setc@ntr@l{2}\figvectPDD-3[#4,#5]\figvectNVDD-4[-3]\resetc@ntr@l{2}%
    \inters@cDD#1:{#2}[#3,-4;#4,-3]\resetc@ntr@l\et@tfigptorthoprojlineDD}\ignorespaces\fi}
\ctr@ld@f\def\figptorthoprojlineTD#1:#2=#3/#4,#5/{\ifGR@cri{\s@uvc@ntr@l\et@tfigptorthoprojlineTD%
    \setc@ntr@l{2}\figvectPTD-1[#4,#3]\figvectPTD-2[#4,#5]\vecunit@TD{-2}{-2}%
    \c@lproscalTD\v@leur[-1,-2]\edef\v@lcoef{\repdecn@mb{\v@leur}}%
    \figpttraTD#1:{#2}=#4/\v@lcoef,-2/\resetc@ntr@l\et@tfigptorthoprojlineTD}\ignorespaces\fi}
\ctr@ln@m\figptorthoprojplane
\ctr@ld@f\def\figptorthoprojplaneDD{\un@v@ilable{figptorthoprojplane}}
\ctr@ld@f\def\figptorthoprojplaneTD#1:#2=#3/#4,#5/{\ifGR@cri{\s@uvc@ntr@l\et@tfigptorthoprojplane%
    \setc@ntr@l{2}\figvectPTD-1[#3,#4]\vecunit@TD{-2}{#5}%
    \c@lproscalTD\v@leur[-1,-2]\edef\v@lcoef{\repdecn@mb{\v@leur}}%
    \figpttraTD#1:{#2}=#3/\v@lcoef,-2/\resetc@ntr@l\et@tfigptorthoprojplane}\ignorespaces\fi}
\ctr@ld@f\def\figpthom#1:#2=#3/#4,#5/{\ifGR@cri{\s@uvc@ntr@l\et@tfigpthom%
    \setc@ntr@l{2}\figvectP-1[#4,#3]\figpttra#1:{#2}=#4/#5,-1/%
    \resetc@ntr@l\et@tfigpthom}\ignorespaces\fi}
\ctr@ld@f\def\figptinv#1:#2=#3/#4,#5/{\ifGR@cri{\s@uvc@ntr@l\et@tfigptinv%
    \setc@ntr@l{2}\figvectP-1[#4,#3]\Figg@tXY{-1}%
    \getredf@ctB\f@ctech\n@rmeucC{\delt@}{-1}%
    \delt@=\ptT@unit@\delt@\delt@=\ptT@unit@\delt@%
    \invers@{\delt@}{\delt@}\multiply\f@ctech\f@ctech\divide\delt@\f@ctech%
    \delt@=#5\delt@\edef\v@lcoef{\repdecn@mb{\delt@}}\figpttra#1:{#2}=#4/\v@lcoef,-1/%
    \resetc@ntr@l\et@tfigptinv}\ignorespaces\fi}
\ctr@ln@m\figptrot
\ctr@ld@f\def\figptrotDD#1:#2=#3/#4,#5/{\ifGR@cri{\s@uvc@ntr@l\et@tfigptrotDD%
    \c@ssin{\C@}{\S@}{#5}\setc@ntr@l{2}\figvectPDD-1[#4,#3]\Figg@tXY{-1}%
    \v@lXa=\C@\v@lX\advance\v@lXa-\S@\v@lY%
    \v@lYa=\S@\v@lX\advance\v@lYa\C@\v@lY%
    \Figv@ctCreg-1(\v@lXa,\v@lYa)\figpttraDD#1:{#2}=#4/1,-1/%
    \resetc@ntr@l\et@tfigptrotDD}\ignorespaces\fi}
\ctr@ld@f\def\figptrotTD#1:#2=#3/#4,#5,#6/{\ifGR@cri{\s@uvc@ntr@l\et@tfigptrotTD%
    \c@ssin{\C@}{\S@}{#5}%
    \setc@ntr@l{2}\figptorthoprojplaneTD-3:=#4/#3,#6/\figvectPTD-2[-3,#3]%
    \n@rmeucTD\v@leur{-2}\ifdim\v@leur<\Cepsil@n\Figg@tXYa{#3}\else%
    \edef\v@lcoef{\repdecn@mb{\v@leur}}\figvectNVTD-1[#6,-2]%
    \Figg@tXYa{-1}\v@lXa=\v@lcoef\v@lXa\v@lYa=\v@lcoef\v@lYa\v@lZa=\v@lcoef\v@lZa%
    \v@lXa=\S@\v@lXa\v@lYa=\S@\v@lYa\v@lZa=\S@\v@lZa\Figg@tXY{-2}%
    \advance\v@lXa\C@\v@lX\advance\v@lYa\C@\v@lY\advance\v@lZa\C@\v@lZ%
    \Figg@tXY{-3}\advance\v@lXa\v@lX\advance\v@lYa\v@lY\advance\v@lZa\v@lZ\fi%
    \Figp@intregTD#1:{#2}(\v@lXa,\v@lYa,\v@lZa)\resetc@ntr@l\et@tfigptrotTD}\ignorespaces\fi}
\ctr@ln@m\figptsym
\ctr@ld@f\def\figptsymDD#1:#2=#3/#4,#5/{\ifGR@cri{\s@uvc@ntr@l\et@tfigptsymDD%
    \resetc@ntr@l{2}\figptorthoprojlineDD-5:=#3/#4,#5/\figvectPDD-2[#3,-5]%
    \figpttraDD#1:{#2}=#3/2,-2/\resetc@ntr@l\et@tfigptsymDD}\ignorespaces\fi}
\ctr@ld@f\def\figptsymTD#1:#2=#3/#4,#5/{\ifGR@cri{\s@uvc@ntr@l\et@tfigptsymTD%
    \resetc@ntr@l{2}\figptorthoprojplaneTD-3:=#3/#4,#5/\figvectPTD-2[#3,-3]%
    \figpttraTD#1:{#2}=#3/2,-2/\resetc@ntr@l\et@tfigptsymTD}\ignorespaces\fi}
\ctr@ln@m\figpttra
\ctr@ld@f\def\figpttraDD#1:#2=#3/#4,#5/{\ifGR@cri{\Figg@tXYa{#5}\v@lXa=#4\v@lXa\v@lYa=#4\v@lYa%
    \Figg@tXY{#3}\advance\v@lX\v@lXa\advance\v@lY\v@lYa%
    \Figp@intregDD#1:{#2}(\v@lX,\v@lY)}\ignorespaces\fi}
\ctr@ld@f\def\figpttraTD#1:#2=#3/#4,#5/{\ifGR@cri{\Figg@tXYa{#5}\v@lXa=#4\v@lXa\v@lYa=#4\v@lYa%
    \v@lZa=#4\v@lZa\Figg@tXY{#3}\advance\v@lX\v@lXa\advance\v@lY\v@lYa%
    \advance\v@lZ\v@lZa\Figp@intregTD#1:{#2}(\v@lX,\v@lY,\v@lZ)}\ignorespaces\fi}
\ctr@ln@m\figpttraC
\ctr@ld@f\def\figpttraCDD#1:#2=#3/#4,#5/{\ifGR@cri{\v@lXa=#4\unit@\v@lYa=#5\unit@%
    \Figg@tXY{#3}\advance\v@lX\v@lXa\advance\v@lY\v@lYa%
    \Figp@intregDD#1:{#2}(\v@lX,\v@lY)}\ignorespaces\fi}
\ctr@ld@f\def\figpttraCTD#1:#2=#3/#4,#5,#6/{\ifGR@cri{\v@lXa=#4\unit@\v@lYa=#5\unit@\v@lZa=#6\unit@%
    \Figg@tXY{#3}\advance\v@lX\v@lXa\advance\v@lY\v@lYa\advance\v@lZ\v@lZa%
    \Figp@intregTD#1:{#2}(\v@lX,\v@lY,\v@lZ)}\ignorespaces\fi}
\ctr@ld@f\def\figptsaxes#1:#2(#3){\ifGR@cri{\an@lys@xes#3,:\ifx\t@xt@\empty%
    \ifTr@isDim\Figpts@xes#1:#2(0,#3,0,#3,0,#3)\else\Figpts@xes#1:#2(0,#3,0,#3)\fi%
    \else\Figpts@xes#1:#2(#3)\fi}\ignorespaces\fi}
\ctr@ln@m\Figpts@xes
\ctr@ld@f\def\Figpts@xesDD#1:#2(#3,#4,#5,#6){%
    \s@mme=#1\figpttraC\the\s@mme:$x$=#2/#4,0/%
    \advance\s@mme\@ne\figpttraC\the\s@mme:$y$=#2/0,#6/}
\ctr@ld@f\def\Figpts@xesTD#1:#2(#3,#4,#5,#6,#7,#8){%
    \s@mme=#1\figpttraC\the\s@mme:$x$=#2/#4,0,0/%
    \advance\s@mme\@ne\figpttraC\the\s@mme:$y$=#2/0,#6,0/%
    \advance\s@mme\@ne\figpttraC\the\s@mme:$z$=#2/0,0,#8/}
\ctr@ld@f\def\figptsmap#1=#2/#3/#4/{\ifGR@cri{\s@uvc@ntr@l\et@tfigptsmap%
    \setc@ntr@l{2}\def\list@num{#2}\s@mme=#1%
    \@ecfor\p@int:=\list@num\do{\figvectP-1[#3,\p@int]\Figg@tXY{-1}%
    \pr@dMatV/#4/\figpttra\the\s@mme:=#3/1,-1/\advance\s@mme\@ne}%
    \resetc@ntr@l\et@tfigptsmap}\ignorespaces\fi}
\ctr@ln@m\figptscontrol
\ctr@ld@f\def\figptscontrolDD#1[#2,#3,#4,#5]{\ifGR@cri{\s@uvc@ntr@l\et@tfigptscontrolDD\setc@ntr@l{2}%
    \v@lX=\z@\v@lY=\z@\Figtr@nptDD{-5}{#2}\Figtr@nptDD{2}{#5}%
    \divide\v@lX\@vi\divide\v@lY\@vi%
    \Figtr@nptDD{3}{#3}\Figtr@nptDD{-1.5}{#4}\Figp@intregDD-1:(\v@lX,\v@lY)%
    \v@lX=\z@\v@lY=\z@\Figtr@nptDD{2}{#2}\Figtr@nptDD{-5}{#5}%
    \divide\v@lX\@vi\divide\v@lY\@vi\Figtr@nptDD{-1.5}{#3}\Figtr@nptDD{3}{#4}%
    \s@mme=#1\advance\s@mme\@ne\Figp@intregDD\the\s@mme:(\v@lX,\v@lY)%
    \figptcopyDD#1:/-1/\resetc@ntr@l\et@tfigptscontrolDD}\ignorespaces\fi}
\ctr@ld@f\def\figptscontrolTD#1[#2,#3,#4,#5]{\ifGR@cri{\s@uvc@ntr@l\et@tfigptscontrolTD\setc@ntr@l{2}%
    \v@lX=\z@\v@lY=\z@\v@lZ=\z@\Figtr@nptTD{-5}{#2}\Figtr@nptTD{2}{#5}%
    \divide\v@lX\@vi\divide\v@lY\@vi\divide\v@lZ\@vi%
    \Figtr@nptTD{3}{#3}\Figtr@nptTD{-1.5}{#4}\Figp@intregTD-1:(\v@lX,\v@lY,\v@lZ)%
    \v@lX=\z@\v@lY=\z@\v@lZ=\z@\Figtr@nptTD{2}{#2}\Figtr@nptTD{-5}{#5}%
    \divide\v@lX\@vi\divide\v@lY\@vi\divide\v@lZ\@vi\Figtr@nptTD{-1.5}{#3}\Figtr@nptTD{3}{#4}%
    \s@mme=#1\advance\s@mme\@ne\Figp@intregTD\the\s@mme:(\v@lX,\v@lY,\v@lZ)%
    \figptcopyTD#1:/-1/\resetc@ntr@l\et@tfigptscontrolTD}\ignorespaces\fi}
\ctr@ld@f\def\Figtr@nptDD#1#2{\Figg@tXYa{#2}\v@lXa=#1\v@lXa\v@lYa=#1\v@lYa%
    \advance\v@lX\v@lXa\advance\v@lY\v@lYa}
\ctr@ld@f\def\Figtr@nptTD#1#2{\Figg@tXYa{#2}\v@lXa=#1\v@lXa\v@lYa=#1\v@lYa\v@lZa=#1\v@lZa%
    \advance\v@lX\v@lXa\advance\v@lY\v@lYa\advance\v@lZ\v@lZa}
\ctr@ld@f\def\figptscontrolcurve#1,#2[#3]{\ifGR@cri{\s@uvc@ntr@l\et@tfigptscontrolcurve%
    \def\list@num{#3}\extrairelepremi@r\Ak@\de\list@num%
    \extrairelepremi@r\Ai@\de\list@num\extrairelepremi@r\Aj@\de\list@num%
    \s@mme=#1\figptcopy\the\s@mme:/\Ai@/%
    \setc@ntr@l{2}\figvectP -1[\Ak@,\Aj@]%
    \@ecfor\Ak@:=\list@num\do{\advance\s@mme\@ne\figpttra\the\s@mme:=\Ai@/\curv@roundness,-1/%
       \figvectP -1[\Ai@,\Ak@]\advance\s@mme\@ne\figpttra\the\s@mme:=\Aj@/-\curv@roundness,-1/%
       \advance\s@mme\@ne\figptcopy\the\s@mme:/\Aj@/%
       \edef\Ai@{\Aj@}\edef\Aj@{\Ak@}}\advance\s@mme-#1\divide\s@mme\thr@@%
       \xdef#2{\the\s@mme}%
    \resetc@ntr@l\et@tfigptscontrolcurve}\ignorespaces\fi}
\ctr@ln@m\figptsintercirc
\ctr@ld@f\def\figptsintercircDD#1[#2,#3;#4,#5]{\ifGR@cri{\s@uvc@ntr@l\et@tfigptsintercircDD%
    \setc@ntr@l{2}\let\c@lNVintc=\c@lNVintcDD\Figptsintercirc@#1[#2,#3;#4,#5]%    
    \resetc@ntr@l\et@tfigptsintercircDD}\ignorespaces\fi}
\ctr@ld@f\def\figptsintercircTD#1[#2,#3;#4,#5;#6]{\ifGR@cri{\s@uvc@ntr@l\et@tfigptsintercircTD%
    \setc@ntr@l{2}\let\c@lNVintc=\c@lNVintcTD\vecunitC@TD[#2,#6]%
    \Figv@ctCreg-3(\v@lX,\v@lY,\v@lZ)\Figptsintercirc@#1[#2,#3;#4,#5]%
    \resetc@ntr@l\et@tfigptsintercircTD}\ignorespaces\fi}
\ctr@ld@f\def\Figptsintercirc@#1[#2,#3;#4,#5]{\figvectP-1[#2,#4]%
    \vecunit@{-1}{-1}\delt@=\result@t\f@ctech=\result@tent%
    \s@mme=#1\advance\s@mme\@ne\figptcopy#1:/#2/\figptcopy\the\s@mme:/#4/%
    \ifdim\delt@=\z@\else%
    \v@lmin=#3\unit@\v@lmax=#5\unit@\v@leur=\v@lmin\advance\v@leur\v@lmax%
    \ifdim\v@leur>\delt@%
    \v@leur=\v@lmin\advance\v@leur-\v@lmax\maxim@m{\v@leur}{\v@leur}{-\v@leur}%
    \ifdim\v@leur<\delt@%
    \divide\v@lmin\f@ctech\divide\v@lmax\f@ctech\divide\delt@\f@ctech%
    \v@lmin=\repdecn@mb{\v@lmin}\v@lmin\v@lmax=\repdecn@mb{\v@lmax}\v@lmax%
    \invers@{\v@leur}{\delt@}\advance\v@lmax-\v@lmin%
    \v@lmax=-\repdecn@mb{\v@leur}\v@lmax\advance\delt@\v@lmax\delt@=.5\delt@%
    \v@lmax=\delt@\multiply\v@lmax\f@ctech%
    \edef\t@ille{\repdecn@mb{\v@lmax}}\figpttra-2:=#2/\t@ille,-1/%
    \delt@=\repdecn@mb{\delt@}\delt@\advance\v@lmin-\delt@%
    \sqrt@{\v@leur}{\v@lmin}\multiply\v@leur\f@ctech\edef\t@ille{\repdecn@mb{\v@leur}}%
    \c@lNVintc\figpttra#1:=-2/-\t@ille,-1/\figpttra\the\s@mme:=-2/\t@ille,-1/\fi\fi\fi}
\ctr@ld@f\def\c@lNVintcDD{\Figg@tXY{-1}\Figv@ctCreg-1(-\v@lY,\v@lX)} % <=> \figvectNVDD-1[-1]
\ctr@ld@f\def\c@lNVintcTD{{\Figg@tXY{-3}\v@lmin=\v@lX\v@lmax=\v@lY\v@leur=\v@lZ%
    \Figg@tXY{-1}\c@lprovec{-3}\vecunit@{-3}{-3}% <=> \figvectNVTD-3[-1,-3]\vecunit@{-3}{-3}
    \Figg@tXY{-1}\v@lmin=\v@lX\v@lmax=\v@lY%
    \v@leur=\v@lZ\Figg@tXY{-3}\c@lprovec{-1}}} % <=> \figvectNVTD-1[-3,-1]
\ctr@ln@m\figptsinterlinell
\ctr@ld@f\def\figptsinterlinellDD#1[#2,#3,#4,#5;#6,#7]{\ifGR@cri{\s@uvc@ntr@l\et@tfigptsinterlinellDD%
    \figptcopy#1:/#6/\s@mme=#1\advance\s@mme\@ne\figptcopy\the\s@mme:/#7/%
    \v@lmin=#3\unit@\v@lmax=#4\unit@% a, b
    \setc@ntr@l{2}\figptbaryDD-4:[#6,#7;1,1]\figptsrotDD-3=-4,#7/#2,-#5/% D et rotation
    \Figg@tXY{-3}\Figg@tXYa{#2}\advance\v@lX-\v@lXa\advance\v@lY-\v@lYa% alpha, beta
    \figvectP-1[-3,-2]\Figg@tXYa{-1}\figvectP-3[-4,#7]\Figptsint@rLE{#1}% u1, u2
    \resetc@ntr@l\et@tfigptsinterlinellDD}\ignorespaces\fi}
\ctr@ld@f\def\figptsinterlinellP#1[#2,#3,#4;#5,#6]{\ifGR@cri{\s@uvc@ntr@l\et@tfigptsinterlinellP%
    \figptcopy#1:/#5/\s@mme=#1\advance\s@mme\@ne\figptcopy\the\s@mme:/#6/\setc@ntr@l{2}%
    \figvectP-1[#2,#3]\vecunit@{-1}{-1}\v@lmin=\result@t% a
    \figvectP-2[#2,#4]\vecunit@{-2}{-2}\v@lmax=\result@t% b
    \figptbary-4:[#5,#6;1,1]% D
    \figvectP-3[#2,-4]\c@lproscal\v@lX[-3,-1]\c@lproscal\v@lY[-3,-2]% alpha, beta
    \figvectP-3[-4,#6]\c@lproscal\v@lXa[-3,-1]\c@lproscal\v@lYa[-3,-2]% u1, u2
    \Figptsint@rLE{#1}\resetc@ntr@l\et@tfigptsinterlinellP}\ignorespaces\fi}
\ctr@ld@f\def\Figptsint@rLE#1{%
    \getredf@ctDD\f@ctech(\v@lmin,\v@lmax)%
    \getredf@ctDD\p@rtent(\v@lX,\v@lY)\ifnum\p@rtent>\f@ctech\f@ctech=\p@rtent\fi%
    \getredf@ctDD\p@rtent(\v@lXa,\v@lYa)\ifnum\p@rtent>\f@ctech\f@ctech=\p@rtent\fi%
    \divide\v@lmin\f@ctech\divide\v@lmax\f@ctech\divide\v@lX\f@ctech\divide\v@lY\f@ctech%
    \divide\v@lXa\f@ctech\divide\v@lYa\f@ctech%
    \c@rre=\repdecn@mb\v@lXa\v@lmax\mili@u=\repdecn@mb\v@lYa\v@lmin%
    \getredf@ctDD\f@ctech(\c@rre,\mili@u)%
    \c@rre=\repdecn@mb\v@lX\v@lmax\mili@u=\repdecn@mb\v@lY\v@lmin%
    \getredf@ctDD\p@rtent(\c@rre,\mili@u)\ifnum\p@rtent>\f@ctech\f@ctech=\p@rtent\fi%
    \divide\v@lmin\f@ctech\divide\v@lmax\f@ctech\divide\v@lX\f@ctech\divide\v@lY\f@ctech%
    \divide\v@lXa\f@ctech\divide\v@lYa\f@ctech%
    \v@lmin=\repdecn@mb{\v@lmin}\v@lmin\v@lmax=\repdecn@mb{\v@lmax}\v@lmax%
    \edef\G@xde{\repdecn@mb\v@lmin}\edef\P@xde{\repdecn@mb\v@lmax}%
    \c@rre=-\v@lmax\v@leur=\repdecn@mb\v@lY\v@lY\advance\c@rre\v@leur\c@rre=\G@xde\c@rre%
    \v@leur=\repdecn@mb\v@lX\v@lX\v@leur=\P@xde\v@leur\advance\c@rre\v@leur% C
    \v@lmin=\repdecn@mb\v@lYa\v@lmin\v@lmax=\repdecn@mb\v@lXa\v@lmax%
    \mili@u=\repdecn@mb\v@lX\v@lmax\advance\mili@u\repdecn@mb\v@lY\v@lmin% B
    \v@lmax=\repdecn@mb\v@lXa\v@lmax\advance\v@lmax\repdecn@mb\v@lYa\v@lmin% A
    \ifdim\v@lmax>\epsil@n%
    \maxim@m{\v@leur}{\c@rre}{-\c@rre}\maxim@m{\v@lmin}{\mili@u}{-\mili@u}%
    \maxim@m{\v@leur}{\v@leur}{\v@lmin}\maxim@m{\v@lmin}{\v@lmax}{-\v@lmax}%
    \maxim@m{\v@leur}{\v@leur}{\v@lmin}\p@rtentiere{\p@rtent}{\v@leur}\advance\p@rtent\@ne%
    \divide\c@rre\p@rtent\divide\mili@u\p@rtent\divide\v@lmax\p@rtent%
    \delt@=\repdecn@mb{\mili@u}\mili@u\v@leur=\repdecn@mb{\v@lmax}\c@rre%
    \advance\delt@-\v@leur\ifdim\delt@<\z@\else\sqrt@\delt@\delt@%
    \invers@\v@lmax\v@lmax\edef\Uns@rAp{\repdecn@mb\v@lmax}%
    \v@leur=-\mili@u\advance\v@leur-\delt@\v@leur=\Uns@rAp\v@leur%
    \edef\t@ille{\repdecn@mb\v@leur}\figpttra#1:=-4/\t@ille,-3/\s@mme=#1\advance\s@mme\@ne%
    \v@leur=-\mili@u\advance\v@leur\delt@\v@leur=\Uns@rAp\v@leur%
    \edef\t@ille{\repdecn@mb\v@leur}\figpttra\the\s@mme:=-4/\t@ille,-3/\fi\fi}
\ctr@ln@m\figptsorthoprojline
\ctr@ld@f\def\figptsorthoprojlineDD#1=#2/#3,#4/{\ifGR@cri{\s@uvc@ntr@l\et@tfigptsorthoprojlineDD%
    \setc@ntr@l{2}\figvectPDD-3[#3,#4]\figvectNVDD-4[-3]\resetc@ntr@l{2}%
    \def\list@num{#2}\s@mme=#1\@ecfor\p@int:=\list@num\do{%
    \inters@cDD\the\s@mme:[\p@int,-4;#3,-3]\advance\s@mme\@ne}%
    \resetc@ntr@l\et@tfigptsorthoprojlineDD}\ignorespaces\fi}
\ctr@ld@f\def\figptsorthoprojlineTD#1=#2/#3,#4/{\ifGR@cri{\s@uvc@ntr@l\et@tfigptsorthoprojlineTD%
    \setc@ntr@l{2}\figvectPTD-2[#3,#4]\vecunit@TD{-2}{-2}%
    \def\list@num{#2}\s@mme=#1\@ecfor\p@int:=\list@num\do{%
    \figvectPTD-1[#3,\p@int]\c@lproscalTD\v@leur[-1,-2]%
    \edef\v@lcoef{\repdecn@mb{\v@leur}}\figpttraTD\the\s@mme:=#3/\v@lcoef,-2/%
    \advance\s@mme\@ne}\resetc@ntr@l\et@tfigptsorthoprojlineTD}\ignorespaces\fi}
\ctr@ln@m\figptsorthoprojplane
\ctr@ld@f\def\figptsorthoprojplaneDD{\un@v@ilable{figptsorthoprojplane}}
\ctr@ld@f\def\figptsorthoprojplaneTD#1=#2/#3,#4/{\ifGR@cri{\s@uvc@ntr@l\et@tfigptsorthoprojplane%
    \setc@ntr@l{2}\vecunit@TD{-2}{#4}%
    \def\list@num{#2}\s@mme=#1\@ecfor\p@int:=\list@num\do{\figvectPTD-1[\p@int,#3]%
    \c@lproscalTD\v@leur[-1,-2]\edef\v@lcoef{\repdecn@mb{\v@leur}}%
    \figpttraTD\the\s@mme:=\p@int/\v@lcoef,-2/\advance\s@mme\@ne}%
    \resetc@ntr@l\et@tfigptsorthoprojplane}\ignorespaces\fi}
\ctr@ld@f\def\figptshom#1=#2/#3,#4/{\ifGR@cri{\s@uvc@ntr@l\et@tfigptshom%
    \setc@ntr@l{2}\def\list@num{#2}\s@mme=#1%
    \@ecfor\p@int:=\list@num\do{\figvectP-1[#3,\p@int]%
    \figpttra\the\s@mme:=#3/#4,-1/\advance\s@mme\@ne}%
    \resetc@ntr@l\et@tfigptshom}\ignorespaces\fi}
\ctr@ld@f\def\figptsinv#1=#2/#3,#4/{\ifGR@cri{\s@uvc@ntr@l\et@tfigptsinv%
    \setc@ntr@l{2}\def\list@num{#2}\s@mme=#1%
    \@ecfor\p@int:=\list@num\do{\figvectP-1[#3,\p@int]\Figg@tXY{-1}%
    \getredf@ctB\f@ctech\n@rmeucC{\delt@}{-1}%
    \delt@=\ptT@unit@\delt@\delt@=\ptT@unit@\delt@%
    \invers@{\delt@}{\delt@}\multiply\f@ctech\f@ctech\divide\delt@\f@ctech%
    \delt@=#4\delt@\edef\v@lcoef{\repdecn@mb{\delt@}}\figpttra\the\s@mme:=#3/\v@lcoef,-1/%
    \advance\s@mme\@ne}\resetc@ntr@l\et@tfigptsinv}\ignorespaces\fi}
\ctr@ln@m\figptsrot
\ctr@ld@f\def\figptsrotDD#1=#2/#3,#4/{\ifGR@cri{\s@uvc@ntr@l\et@tfigptsrotDD%
    \c@ssin{\C@}{\S@}{#4}\setc@ntr@l{2}\def\list@num{#2}\s@mme=#1%
    \@ecfor\p@int:=\list@num\do{\figvectPDD-1[#3,\p@int]\Figg@tXY{-1}%
    \v@lXa=\C@\v@lX\advance\v@lXa-\S@\v@lY%
    \v@lYa=\S@\v@lX\advance\v@lYa\C@\v@lY%
    \Figv@ctCreg-1(\v@lXa,\v@lYa)\figpttraDD\the\s@mme:=#3/1,-1/\advance\s@mme\@ne}%
    \resetc@ntr@l\et@tfigptsrotDD}\ignorespaces\fi}
\ctr@ld@f\def\figptsrotTD#1=#2/#3,#4,#5/{\ifGR@cri{\s@uvc@ntr@l\et@tfigptsrotTD%
    \c@ssin{\C@}{\S@}{#4}%
    \setc@ntr@l{2}\def\list@num{#2}\s@mme=#1%
    \@ecfor\p@int:=\list@num\do{\figptorthoprojplaneTD-3:=#3/\p@int,#5/%
    \figvectPTD-2[-3,\p@int]%
    \figvectNVTD-1[#5,-2]\n@rmeucTD\v@leur{-2}\edef\v@lcoef{\repdecn@mb{\v@leur}}%
    \Figg@tXYa{-1}\v@lXa=\v@lcoef\v@lXa\v@lYa=\v@lcoef\v@lYa\v@lZa=\v@lcoef\v@lZa%
    \v@lXa=\S@\v@lXa\v@lYa=\S@\v@lYa\v@lZa=\S@\v@lZa\Figg@tXY{-2}%
    \advance\v@lXa\C@\v@lX\advance\v@lYa\C@\v@lY\advance\v@lZa\C@\v@lZ%
    \Figg@tXY{-3}\advance\v@lXa\v@lX\advance\v@lYa\v@lY\advance\v@lZa\v@lZ%
    \Figp@intregTD\the\s@mme:(\v@lXa,\v@lYa,\v@lZa)\advance\s@mme\@ne}%
    \resetc@ntr@l\et@tfigptsrotTD}\ignorespaces\fi}
\ctr@ln@m\figptssym
\ctr@ld@f\def\figptssymDD#1=#2/#3,#4/{\ifGR@cri{\s@uvc@ntr@l\et@tfigptssymDD%
    \setc@ntr@l{2}\figvectPDD-3[#3,#4]\Figg@tXY{-3}\Figv@ctCreg-4(-\v@lY,\v@lX)%
    \resetc@ntr@l{2}\def\list@num{#2}\s@mme=#1%
    \@ecfor\p@int:=\list@num\do{\inters@cDD-5:[#3,-3;\p@int,-4]\figvectPDD-2[\p@int,-5]%
    \figpttraDD\the\s@mme:=\p@int/2,-2/\advance\s@mme\@ne}%
    \resetc@ntr@l\et@tfigptssymDD}\ignorespaces\fi}
\ctr@ld@f\def\figptssymTD#1=#2/#3,#4/{\ifGR@cri{\s@uvc@ntr@l\et@tfigptssymTD%
    \setc@ntr@l{2}\vecunit@TD{-2}{#4}\def\list@num{#2}\s@mme=#1%
    \@ecfor\p@int:=\list@num\do{\figvectPTD-1[\p@int,#3]%
    \c@lproscalTD\v@leur[-1,-2]\v@leur=2\v@leur\edef\v@lcoef{\repdecn@mb{\v@leur}}%
    \figpttraTD\the\s@mme:=\p@int/\v@lcoef,-2/\advance\s@mme\@ne}%
    \resetc@ntr@l\et@tfigptssymTD}\ignorespaces\fi}
\ctr@ln@m\figptstra
\ctr@ld@f\def\figptstraDD#1=#2/#3,#4/{\ifGR@cri{\Figg@tXYa{#4}\v@lXa=#3\v@lXa\v@lYa=#3\v@lYa%
    \def\list@num{#2}\s@mme=#1\@ecfor\p@int:=\list@num\do{\Figg@tXY{\p@int}%
    \advance\v@lX\v@lXa\advance\v@lY\v@lYa%
    \Figp@intregDD\the\s@mme:(\v@lX,\v@lY)\advance\s@mme\@ne}}\ignorespaces\fi}
\ctr@ld@f\def\figptstraTD#1=#2/#3,#4/{\ifGR@cri{\Figg@tXYa{#4}\v@lXa=#3\v@lXa\v@lYa=#3\v@lYa%
    \v@lZa=#3\v@lZa\def\list@num{#2}\s@mme=#1\@ecfor\p@int:=\list@num\do{\Figg@tXY{\p@int}%
    \advance\v@lX\v@lXa\advance\v@lY\v@lYa\advance\v@lZ\v@lZa%
    \Figp@intregTD\the\s@mme:(\v@lX,\v@lY,\v@lZ)\advance\s@mme\@ne}}\ignorespaces\fi}
\ctr@ln@m\figptvisilimSL
\ctr@ld@f\def\figptvisilimSLDD{\un@v@ilable{figptvisilimSL}}
\ctr@ld@f\def\figptvisilimSLTD#1:#2[#3,#4;#5,#6]{\ifGR@cri{\s@uvc@ntr@l\et@tfigptvisilimSLTD%
    \setc@ntr@l{2}\figvectP-1[#3,#4]\n@rminf{\delt@}{-1}%
    \ifcase\CUR@proj\v@lX=\cxa@\p@\v@lY=-\p@\v@lZ=\cxb@\p@% Proj cav
    \Figv@ctCreg-2(\v@lX,\v@lY,\v@lZ)\figvectP-3[#5,#6]\figvectNV-1[-2,-3]%
    \or\figvectP-1[#5,#6]\vecunitCV@TD{-1}\v@lmin=\v@lX\v@lmax=\v@lY% Proj ortho
    \v@leur=\v@lZ\v@lX=\cza@\p@\v@lY=\czb@\p@\v@lZ=\czc@\p@\c@lprovec{-1}%
    \or\c@ley@pt{-2}\figvectN-1[#5,#6,-2]\fi% Proj rea
    \edef\Ai@{#3}\edef\Aj@{#4}\figvectP-2[#5,\Ai@]\c@lproscal\v@leur[-1,-2]%
    \ifdim\v@leur>\z@\p@rtent=\@ne\else\p@rtent=\m@ne\fi%
    \figvectP-2[#5,\Aj@]\c@lproscal\v@leur[-1,-2]%
    \ifdim\p@rtent\v@leur>\z@\figptcopy#1:#2/#3/%
    \message{*** \BS@ figptvisilimSL: points are on the same side.}\else%
    \figptcopy-3:/#3/\figptcopy-4:/#4/%
    \loop\figptbary-5:[-3,-4;1,1]\figvectP-2[#5,-5]\c@lproscal\v@leur[-1,-2]%
    \ifdim\p@rtent\v@leur>\z@\figptcopy-3:/-5/\else\figptcopy-4:/-5/\fi%
    \divide\delt@\tw@\ifdim\delt@>\epsil@n\repeat%
    \figptbary#1:#2[-3,-4;1,1]\fi\resetc@ntr@l\et@tfigptvisilimSLTD}\ignorespaces\fi}
\ctr@ld@f\def\c@ley@pt#1{\t@stp@r\ifitis@K\v@lX=\cza@\p@\v@lY=\czb@\p@\v@lZ=\czc@\p@%
    \Figv@ctCreg-1(\v@lX,\v@lY,\v@lZ)\Figp@intreg-2:(\wd\Bt@rget,\ht\Bt@rget,\dp\Bt@rget)%
    \figpttra#1:=-2/-\disob@intern,-1/\else\end\fi}
\ctr@ld@f\def\t@stp@r{\itis@Ktrue\ifnewt@rgetpt\else\itis@Kfalse%
    \message{*** \BS@ figptvisilimXX: target point undefined.}\fi\ifnewdis@b\else%
    \itis@Kfalse\message{*** \BS@ figptvisilimXX: observation distance undefined.}\fi%
    \ifitis@K\else\message{*** This macro must be called after \BS@ figdrawbegin or after
    having set the missing parameter(s) with \BS@ figset proj()}\fi}
\ctr@ld@f\def\figscan#1(#2,#3){{\s@uvc@ntr@l\et@tfigscan\@psfgetbb{#1}\if@psfbbfound\else%
    \def\@psfllx{0}\def\@psflly{20}\def\@psfurx{540}\def\@psfury{640}\fi\figscan@{#2}{#3}%
    \resetc@ntr@l\et@tfigscan}\ignorespaces}
\ctr@ld@f\def\figscan@#1#2{%
    \unit@=\@ne bp\setc@ntr@l{2}\figsetmark{}%
    \def\minst@p{20pt}%
    \v@lX=\@psfllx\p@\v@lX=\Sc@leFact\v@lX\r@undint\v@lX\v@lX%
    \v@lY=\@psflly\p@\v@lY=\Sc@leFact\v@lY\ifdim\v@lY>\z@\r@undint\v@lY\v@lY\fi%
    \delt@=\@psfury\p@\delt@=\Sc@leFact\delt@%
    \advance\delt@-\v@lY\v@lXa=\@psfurx\p@\v@lXa=\Sc@leFact\v@lXa\v@leur=\minst@p%
    \edef\valv@lY{\repdecn@mb{\v@lY}}\edef\LgTr@it{\the\delt@}%
    \loop\ifdim\v@lX<\v@lXa\edef\valv@lX{\repdecn@mb{\v@lX}}%
    \figptDD -1:(\valv@lX,\valv@lY)\figwriten -1:\hbox{\vrule height\LgTr@it}(0)%
    \ifdim\v@leur<\minst@p\else\figsetmark{\raise-8bp\hbox{$\scriptscriptstyle\triangle$}}%
    \figwrites -1:\@ffichnb{0}{\valv@lX}(6)\v@leur=\z@\figsetmark{}\fi%
    \advance\v@leur#1pt\advance\v@lX#1pt\repeat%
    \def\minst@p{10pt}%
    \v@lX=\@psfllx\p@\v@lX=\Sc@leFact\v@lX\ifdim\v@lX>\z@\r@undint\v@lX\v@lX\fi%
    \v@lY=\@psflly\p@\v@lY=\Sc@leFact\v@lY\r@undint\v@lY\v@lY%
    \delt@=\@psfurx\p@\delt@=\Sc@leFact\delt@%
    \advance\delt@-\v@lX\v@lYa=\@psfury\p@\v@lYa=\Sc@leFact\v@lYa\v@leur=\minst@p%
    \edef\valv@lX{\repdecn@mb{\v@lX}}\edef\LgTr@it{\the\delt@}%
    \loop\ifdim\v@lY<\v@lYa\edef\valv@lY{\repdecn@mb{\v@lY}}%
    \figptDD -1:(\valv@lX,\valv@lY)\figwritee -1:\vbox{\hrule width\LgTr@it}(0)%
    \ifdim\v@leur<\minst@p\else\figsetmark{$\triangleright$\kern4bp}%
    \figwritew -1:\@ffichnb{0}{\valv@lY}(6)\v@leur=\z@\figsetmark{}\fi%
    \advance\v@leur#2pt\advance\v@lY#2pt\repeat}
\ctr@ld@f
\ctr@ld@f\def\figscan@E#1(#2,#3){{\s@uvc@ntr@l\et@tfigscan@E%
    \Figdisc@rdLTS{#1}{\t@xt@}\pdfximage{\t@xt@}%
    \setbox\Gb@x=\hbox{\pdfrefximage\pdflastximage}%
    \edef\@psfllx{0}\v@lY=-\dp\Gb@x\edef\@psflly{\repdecn@mb{\v@lY}}%
    \edef\@psfurx{\repdecn@mb{\wd\Gb@x}}%
    \v@lY=\dp\Gb@x\advance\v@lY\ht\Gb@x\edef\@psfury{\repdecn@mb{\v@lY}}%
    \figscan@{#2}{#3}\resetc@ntr@l\et@tfigscan@E}\ignorespaces}
\ctr@ld@f\def\figshowpts[#1,#2]{{\figsetmark{$\bullet$}\figsetptname{\bf ##1}%
    \p@rtent=#2\relax\ifnum\p@rtent<\z@\p@rtent=\z@\fi%
    \s@mme=#1\relax\ifnum\s@mme<\z@\s@mme=\z@\fi%
    \loop\ifnum\s@mme<\p@rtent\pt@rvect{\s@mme}%
    \ifitis@K\figwriten{\the\s@mme}:(4pt)\fi\advance\s@mme\@ne\repeat%
    \pt@rvect{\s@mme}\ifitis@K\figwriten{\the\s@mme}:(4pt)\fi}\ignorespaces}
\ctr@ld@f\def\pt@rvect#1{\set@bjc@de{#1}%
    \expandafter\expandafter\expandafter\inqpt@rvec\csname\objc@de\endcsname:}
\ctr@ld@f\def\inqpt@rvec#1#2:{\if#1\C@dCl@spt\itis@Ktrue\else\itis@Kfalse\fi}
\ctr@ld@f\def\figshowsettings{{%
    \immediate\write16{====================================================================}%
    \immediate\write16{ Current settings are (DDV means "with dynamic default value"):}%
    \immediate\write16{ --- GENERAL ---}%
    \immediate\write16{Scale factor and Unit = \unit@util\space (\the\unit@)
     \space -> \BS@ figinit{ScaleFactorUnit}}%
    \immediate\write16{Update mode = \ifGRupdatem@de yes\else no\fi
     \space-> \BS@ figset(update=yes/no) or \BS@ figsetdefault(update=yes/no)}%
    \immediate\write16{ --- WRITING ---}%
    \immediate\write16{Implicit point name = \ptn@me{i} \space-> \BS@ figset write(ptname={Name})}%
    \immediate\write16{Point marker = \the\c@nsymb \space -> \BS@ figset write(mark=Mark)}%
    \immediate\write16{Print rounded coordinates = \ifr@undcoord yes\else no\fi
     \space-> \BS@ figset write(roundcoord=yes/no)}%
    \immediate\write16{ --- GRAPHICAL (general) ---}%
    \immediate\write16{Color = \CUR@color \space-> \BS@ figset(color=ColorDefinition)}%
    \immediate\write16{Filling mode = \iffillm@de yes\else no\fi
     \space-> \BS@ figset(fillmode=yes/no)}%
    \immediate\write16{Line join = \CUR@join \space-> \BS@ figset(join=miter/round/bevel)}%
    \immediate\write16{Line style = \CUR@dash \space-> \BS@ figset(dash=Index/Pattern)}%
    \immediate\write16{Line width = \CUR@width
     \space-> \BS@ figset(width=real in PostScript units)}%
    \immediate\write16{ --- GRAPHICAL (specific) ---}%
    \immediate\write16{Altitude (all the following attributes are DDV):}%
    \immediate\write16{ Base line color =
     \ifx\DDV@blcolor\D@FTref general color\else\DDV@blcolor\fi
     \space-> \BS@ figset altitude(blcolor=ColorDefinition)}%
    \immediate\write16{ Base line style =
     \ifx\DDV@bldash\D@FTref general style\else\DDV@bldash\fi
     \space-> \BS@ figset altitude(bldash=Index/Pattern)}%
    \immediate\write16{ Base line width =
     \ifx\DDV@blwidth\D@FTref general width\else\DDV@blwidth\fi
     \space-> \BS@ figset altitude(blwidth=real in PostScript units)}%
    \immediate\write16{ Square line color =
     \ifx\DDV@sqcolor\D@FTref general color\else\DDV@sqcolor\fi
     \space-> \BS@ figset altitude(sqcolor=ColorDefinition)}%
    \immediate\write16{ Square line style =
     \ifx\DDV@sqdash\D@FTref general style\else\DDV@sqdash\fi
     \space-> \BS@ figset altitude(sqdash=Index/Pattern)}%
    \immediate\write16{ Square line width =
     \ifx\DDV@sqwidth\D@FTref general width\else\DDV@sqwidth\fi
     \space-> \BS@ figset altitude(sqwidth=real in PostScript units)}%
    \immediate\write16{Arrowhead:}%
    \immediate\write16{ (half-)Angle = \@rrowheadangle
     \space-> \BS@ figset arrowhead(angle=real in degrees)}%
    \immediate\write16{ Filling mode = \if@rrowhfill yes\else no\fi
     \space-> \BS@ figset arrowhead(fillmode=yes/no)}%
    \immediate\write16{ "Outside" = \if@rrowhout yes\else no\fi
     \space-> \BS@ figset arrowhead(out=yes/no)}%
    \immediate\write16{ Length = \@rrowheadlength
     \if@rrowratio\space(not active)\else\space(active)\fi
     \space-> \BS@ figset arrowhead(length=real in user coord.)}%
    \immediate\write16{ Ratio = \@rrowheadratio
     \if@rrowratio\space(active)\else\space(not active)\fi
     \space-> \BS@ figset arrowhead(ratio=real in [0,1])}%
    \immediate\write16{Curve:}%
    \immediate\write16{ Roundness = \curv@roundness
     \space-> \BS@ figset curve(roundness=real in [0,0.5])}%
    \immediate\write16{Flow chart:}%
    \immediate\write16{ Arrow position = \@rrowp@s
     \space-> \BS@ figset flowchart(arrowposition=real in [0,1])}%
    \immediate\write16{ Arrow reference point = \ifcase\@rrowr@fpt start\else end\fi
     \space-> \BS@ figset flowchart(arrowrefpt = start/end)}%     
    \immediate\write16{ Background color = \fcbgc@lor
     \space-> \BS@ figset flowchart(bgcolor=ColorDefinition)}%
    \immediate\write16{ Line type = \ifcase\fclin@typ@ curve\else polygon\fi
     \space-> \BS@ figset flowchart(line=polygon/curve)}%
    \immediate\write16{ Padding = (\Xp@dd, \Yp@dd)
     \space-> \BS@ figset flowchart(padding = real in user coord.)}%
    \immediate\write16{\space\space\space\space(or
     \BS@ figset flowchart(xpadding=real, ypadding=real) )}%
    \immediate\write16{ Radius = \fclin@r@d
     \space-> \BS@ figset flowchart(radius=positive real in user coord.)}%
    \immediate\write16{ Shape = \fcsh@pe
     \space-> \BS@ figset flowchart(shape = rectangle, ellipse or lozenge)}%
    \immediate\write16{ Thickness color (DDV) = 
     \ifx\DDV@thickcolor\D@FTref general color\else\DDV@thickcolor\fi
     \space-> \BS@ figset flowchart(thickcolor=ColorDefinition)}%
    \immediate\write16{ Thickness = \thickn@ss
     \space-> \BS@ figset flowchart(thickness = real in user coord.)}%
    \immediate\write16{Mesh:}%
    \immediate\write16{ Diagonal = \c@ntrolmesh
     \space-> \BS@ figset mesh(diag=integer in {-1,0,1})}%
    \immediate\write16{ Lines color (DDV) =
     \ifx\DDV@meshcolor\D@FTref general color\else\DDV@meshcolor\fi
     \space-> \BS@ figset mesh(color=ColorDefinition)}%
    \immediate\write16{ Lines style (DDV) =
     \ifx\DDV@meshdash\D@FTref general style\else\DDV@meshdash\fi
     \space-> \BS@ figset mesh(dash=Index/Pattern)}%
    \immediate\write16{ Lines width (DDV) =
     \ifx\DDV@meshwidth\D@FTref general width\else\DDV@meshwidth\fi
     \space-> \BS@ figset mesh(width=real in PostScript units)}%
    \immediate\write16{Trimesh:}%
    \immediate\write16{ Lines color (DDV) =
     \ifx\DDV@tmeshcolor\D@FTref general color\else\DDV@tmeshcolor\fi
     \space-> \BS@ figset trimesh(color=ColorDefinition)}%
    \immediate\write16{ Lines style (DDV) =
     \ifx\DDV@tmeshdash\D@FTref general style\else\DDV@tmeshdash\fi
     \space-> \BS@ figset trimesh(dash=Index/Pattern)}%
    \immediate\write16{ Lines width (DDV) =
     \ifx\DDV@tmeshwidth\D@FTref general width\else\DDV@tmeshwidth\fi
     \space-> \BS@ figset trimesh(width=real in PostScript units)}%
    \ifTr@isDim%
    \immediate\write16{ --- 3D to 2D PROJECTION ---}%
    \immediate\write16{Projection : \typ@proj \space-> \BS@ figinit{ScaleFactorUnit, ProjType}}%
    \immediate\write16{Longitude (psi) = \v@lPsi \space-> \BS@ figset proj(psi=real in degrees)}%
    \ifcase\CUR@proj\immediate\write16{Depth coeff. (Lambda)
     \space = \v@lTheta \space-> \BS@ figset proj(lambda=real in [0,1])}%
    \else\immediate\write16{Latitude (theta)
     \space = \v@lTheta \space-> \BS@ figset proj(theta=real in degrees)}%
    \fi%
    \ifnum\CUR@proj=\tw@%
    \immediate\write16{Observation distance = \disob@unit
     \space-> \BS@ figset proj(dist=real in user coord.)}%
    \immediate\write16{Target point = \t@rgetpt \space-> \BS@ figset proj(targetpt=pt number)}%
     \v@lX=\ptT@unit@\wd\Bt@rget\v@lY=\ptT@unit@\ht\Bt@rget\v@lZ=\ptT@unit@\dp\Bt@rget%
    \immediate\write16{ Its coordinates are
     (\repdecn@mb{\v@lX}, \repdecn@mb{\v@lY}, \repdecn@mb{\v@lZ})}%
    \fi%
    \fi%
    \immediate\write16{====================================================================}%
    \ignorespaces}}
\ctr@ln@w{newif}\ifitis@vect@r
\ctr@ld@f\def\figvectC#1(#2,#3){{\itis@vect@rtrue\figpt#1:(#2,#3)}\ignorespaces}
\ctr@ld@f\def\Figv@ctCreg#1(#2,#3){{\itis@vect@rtrue\Figp@intreg#1:(#2,#3)}\ignorespaces}
\ctr@ln@m\figvectDBezier
\ctr@ld@f\def\figvectDBezierDD#1:#2,#3[#4,#5,#6,#7]{\ifGR@cri{\s@uvc@ntr@l\et@tfigvectDBezierDD%
    \FigvectDBezier@#2,#3[#4,#5,#6,#7]\v@lX=\c@ef\v@lX\v@lY=\c@ef\v@lY%
    \Figv@ctCreg#1(\v@lX,\v@lY)\resetc@ntr@l\et@tfigvectDBezierDD}\ignorespaces\fi}
\ctr@ld@f\def\figvectDBezierTD#1:#2,#3[#4,#5,#6,#7]{\ifGR@cri{\s@uvc@ntr@l\et@tfigvectDBezierTD%
    \FigvectDBezier@#2,#3[#4,#5,#6,#7]\v@lX=\c@ef\v@lX\v@lY=\c@ef\v@lY\v@lZ=\c@ef\v@lZ%
    \Figv@ctCreg#1(\v@lX,\v@lY,\v@lZ)\resetc@ntr@l\et@tfigvectDBezierTD}\ignorespaces\fi}
\ctr@ld@f\def\FigvectDBezier@#1,#2[#3,#4,#5,#6]{\setc@ntr@l{2}%
    \edef\T@{#2}\v@leur=\p@\advance\v@leur-#2pt\edef\UNmT@{\repdecn@mb{\v@leur}}%
    \ifnum#1=\tw@\def\c@ef{6}\else\def\c@ef{3}\fi%
    \figptcopy-4:/#3/\figptcopy-3:/#4/\figptcopy-2:/#5/\figptcopy-1:/#6/%
    \l@mbd@un=-4 \l@mbd@de=-\thr@@\p@rtent=\m@ne\c@lDecast%
    \ifnum#1=\tw@\c@lDCDeux{-4}{-3}\c@lDCDeux{-3}{-2}\c@lDCDeux{-4}{-3}\else%
    \l@mbd@un=-4 \l@mbd@de=-\thr@@\p@rtent=-\tw@\c@lDecast%
    \c@lDCDeux{-4}{-3}\fi\Figg@tXY{-4}}
\ctr@ln@m\c@lDCDeux
\ctr@ld@f\def\c@lDCDeuxDD#1#2{\Figg@tXY{#2}\Figg@tXYa{#1}%
    \advance\v@lX-\v@lXa\advance\v@lY-\v@lYa\Figp@intregDD#1:(\v@lX,\v@lY)}
\ctr@ld@f\def\c@lDCDeuxTD#1#2{\Figg@tXY{#2}\Figg@tXYa{#1}\advance\v@lX-\v@lXa%
    \advance\v@lY-\v@lYa\advance\v@lZ-\v@lZa\Figp@intregTD#1:(\v@lX,\v@lY,\v@lZ)}
\ctr@ln@m\figvectN
\ctr@ld@f\def\figvectNDD#1[#2,#3]{\ifGR@cri{\Figg@tXYa{#2}\Figg@tXY{#3}%
    \advance\v@lX-\v@lXa\advance\v@lY-\v@lYa%
    \Figv@ctCreg#1(-\v@lY,\v@lX)}\ignorespaces\fi}
\ctr@ld@f\def\figvectNTD#1[#2,#3,#4]{\ifGR@cri{\vecunitC@TD[#2,#4]\v@lmin=\v@lX\v@lmax=\v@lY%
    \v@leur=\v@lZ\vecunitC@TD[#2,#3]\c@lprovec{#1}}\ignorespaces\fi}
\ctr@ln@m\figvectNV
\ctr@ld@f\def\figvectNVDD#1[#2]{\ifGR@cri{\Figg@tXY{#2}\Figv@ctCreg#1(-\v@lY,\v@lX)}\ignorespaces\fi}
\ctr@ld@f\def\figvectNVTD#1[#2,#3]{\ifGR@cri{\vecunitCV@TD{#3}\v@lmin=\v@lX\v@lmax=\v@lY%
    \v@leur=\v@lZ\vecunitCV@TD{#2}\c@lprovec{#1}}\ignorespaces\fi}
\ctr@ln@m\figvectP
\ctr@ld@f\def\figvectPDD#1[#2,#3]{\ifGR@cri{\Figg@tXYa{#2}\Figg@tXY{#3}%
    \advance\v@lX-\v@lXa\advance\v@lY-\v@lYa%
    \Figv@ctCreg#1(\v@lX,\v@lY)}\ignorespaces\fi}
\ctr@ld@f\def\figvectPTD#1[#2,#3]{\ifGR@cri{\Figg@tXYa{#2}\Figg@tXY{#3}%
    \advance\v@lX-\v@lXa\advance\v@lY-\v@lYa\advance\v@lZ-\v@lZa%
    \Figv@ctCreg#1(\v@lX,\v@lY,\v@lZ)}\ignorespaces\fi}
\ctr@ln@m\figvectU
\ctr@ld@f\def\figvectUDD#1[#2]{\ifGR@cri{\n@rmeuc\v@leur{#2}\invers@\v@leur\v@leur%
    \delt@=\repdecn@mb{\v@leur}\unit@\edef\v@ldelt@{\repdecn@mb{\delt@}}%
    \Figg@tXY{#2}\v@lX=\v@ldelt@\v@lX\v@lY=\v@ldelt@\v@lY%
    \Figv@ctCreg#1(\v@lX,\v@lY)}\ignorespaces\fi}
\ctr@ld@f\def\figvectUTD#1[#2]{\ifGR@cri{\n@rmeuc\v@leur{#2}\invers@\v@leur\v@leur%
    \delt@=\repdecn@mb{\v@leur}\unit@\edef\v@ldelt@{\repdecn@mb{\delt@}}%
    \Figg@tXY{#2}\v@lX=\v@ldelt@\v@lX\v@lY=\v@ldelt@\v@lY\v@lZ=\v@ldelt@\v@lZ%
    \Figv@ctCreg#1(\v@lX,\v@lY,\v@lZ)}\ignorespaces\fi}
\ctr@ld@f\def\figvisu#1#2#3{\c@ldefproj\initb@undb@x\xdef\figforTeXFigno{\figforTeXnextFigno}%
    \s@mme=\figforTeXnextFigno\advance\s@mme\@ne\xdef\figforTeXnextFigno{\number\s@mme}%
    \setbox\b@xvisu=\hbox{\ifnum\@utoFN>\z@\figinsert{}\gdef\@utoFInDone{0}\fi\ignorespaces#3}%
    \gdef\@utoFInDone{1}\gdef\@utoFN{0}%
    \v@lXa=-\c@@rdYmin\v@lYa=\c@@rdYmax\advance\v@lYa-\c@@rdYmin%
    \v@lX=\c@@rdXmax\advance\v@lX-\c@@rdXmin%
    \setbox#1=\hbox{#2}\v@lY=-\v@lX\maxim@m{\v@lX}{\v@lX}{\wd#1}%
    \advance\v@lY\v@lX\divide\v@lY\tw@\advance\v@lY-\c@@rdXmin%
    \setbox#1=\vbox{\parindent\z@\hsize=\v@lX\vskip\v@lYa%
    \rlap{\hskip\v@lY\smash{\raise\v@lXa\box\b@xvisu}}%
    \def\t@xt@{#2}\ifx\t@xt@\empty\else\medskip\centerline{#2}\fi}\wd#1=\v@lX}
\ctr@ld@f\def\figDecrementFigno{{\xdef\figforTeXnextFigno{\figforTeXFigno}%
    \s@mme=\figforTeXFigno\advance\s@mme\m@ne\xdef\figforTeXFigno{\number\s@mme}}}
\ctr@ln@w{newbox}\Bt@rget\setbox\Bt@rget=\null
\ctr@ln@w{newbox}\BminTD@\setbox\BminTD@=\null
\ctr@ln@w{newbox}\BmaxTD@\setbox\BmaxTD@=\null
\ctr@ln@w{newif}\ifnewt@rgetpt\ctr@ln@w{newif}\ifnewdis@b
\ctr@ld@f\def\b@undb@xTD#1#2#3{%
    \relax\ifdim#1<\wd\BminTD@\global\wd\BminTD@=#1\fi%
    \relax\ifdim#2<\ht\BminTD@\global\ht\BminTD@=#2\fi%
    \relax\ifdim#3<\dp\BminTD@\global\dp\BminTD@=#3\fi%
    \relax\ifdim#1>\wd\BmaxTD@\global\wd\BmaxTD@=#1\fi%
    \relax\ifdim#2>\ht\BmaxTD@\global\ht\BmaxTD@=#2\fi%
    \relax\ifdim#3>\dp\BmaxTD@\global\dp\BmaxTD@=#3\fi}
\ctr@ld@f\def\c@ldefdisob{{\ifdim\wd\BminTD@<\maxdimen\v@leur=\wd\BmaxTD@\advance\v@leur-\wd\BminTD@%
    \delt@=\ht\BmaxTD@\advance\delt@-\ht\BminTD@\maxim@m{\v@leur}{\v@leur}{\delt@}%
    \delt@=\dp\BmaxTD@\advance\delt@-\dp\BminTD@\maxim@m{\v@leur}{\v@leur}{\delt@}%
    \v@leur=5\v@leur\else\v@leur=800pt\fi\c@ldefdisob@{\v@leur}}}
\ctr@ln@m\disob@intern
\ctr@ln@m\disob@
\ctr@ln@m\divf@ctproj
\ctr@ld@f\def\c@ldefdisob@#1{{\v@leur=#1\ifdim\v@leur<\p@\v@leur=800pt\fi%
    \xdef\disob@intern{\repdecn@mb{\v@leur}}%
    \delt@=\ptT@unit@\v@leur\xdef\disob@unit{\repdecn@mb{\delt@}}%
    \f@ctech=\@ne\loop\ifdim\v@leur>\t@n pt\divide\v@leur\t@n\multiply\f@ctech\t@n\repeat%
    \xdef\disob@{\repdecn@mb{\v@leur}}\xdef\divf@ctproj{\the\f@ctech}}%
    \global\newdis@btrue}
\ctr@ln@m\t@rgetpt
\ctr@ld@f\def\c@ldeft@rgetpt{\newt@rgetpttrue\def\t@rgetpt{CenterBoundBox}{%
    \delt@=\wd\BmaxTD@\advance\delt@-\wd\BminTD@\divide\delt@\tw@%
    \v@leur=\wd\BminTD@\advance\v@leur\delt@\global\wd\Bt@rget=\v@leur%
    \delt@=\ht\BmaxTD@\advance\delt@-\ht\BminTD@\divide\delt@\tw@%
    \v@leur=\ht\BminTD@\advance\v@leur\delt@\global\ht\Bt@rget=\v@leur%
    \delt@=\dp\BmaxTD@\advance\delt@-\dp\BminTD@\divide\delt@\tw@%
    \v@leur=\dp\BminTD@\advance\v@leur\delt@\global\dp\Bt@rget=\v@leur}}
\ctr@ln@m\c@ldefproj
\ctr@ld@f\def\c@ldefprojTD{\ifnewt@rgetpt\else\c@ldeft@rgetpt\fi\ifnewdis@b\else\c@ldefdisob\fi}
\ctr@ld@f\def\c@lprojcav{% Projection cavaliere : X = x + y L cos t, Y = z + y L sin t
    \v@lZa=\cxa@\v@lY\advance\v@lX\v@lZa%
    \v@lZa=\cxb@\v@lY\v@lY=\v@lZ\advance\v@lY\v@lZa\ignorespaces}
\ctr@ln@m\v@lcoef
\ctr@ld@f\def\c@lprojrea{% Projection realiste
    \advance\v@lX-\wd\Bt@rget\advance\v@lY-\ht\Bt@rget\advance\v@lZ-\dp\Bt@rget%
    \v@lZa=\cza@\v@lX\advance\v@lZa\czb@\v@lY\advance\v@lZa\czc@\v@lZ%
    \divide\v@lZa\divf@ctproj\advance\v@lZa\disob@ pt\invers@{\v@lZa}{\v@lZa}%
    \v@lZa=\disob@\v@lZa\edef\v@lcoef{\repdecn@mb{\v@lZa}}%
    \v@lXa=\cxa@\v@lX\advance\v@lXa\cxb@\v@lY\v@lXa=\v@lcoef\v@lXa%
    \v@lY=\cyb@\v@lY\advance\v@lY\cya@\v@lX\advance\v@lY\cyc@\v@lZ%
    \v@lY=\v@lcoef\v@lY\v@lX=\v@lXa\ignorespaces}
\ctr@ld@f\def\c@lprojort{% Projection orthogonale
    \v@lXa=\cxa@\v@lX\advance\v@lXa\cxb@\v@lY%
    \v@lY=\cyb@\v@lY\advance\v@lY\cya@\v@lX\advance\v@lY\cyc@\v@lZ%
    \v@lX=\v@lXa\ignorespaces}
\ctr@ld@f\def\Figptpr@j#1:#2/#3/{{\Figg@tXY{#3}\superc@lprojSP%
    \Figp@intregDD#1:{#2}(\v@lX,\v@lY)}\ignorespaces}
\ctr@ln@m\figsetobdist
\ctr@ld@f\def\figsetobdistDD{\un@v@ilable{figsetobdist}}
\ctr@ld@f\def\figsetobdistTD(#1){{\ifCUR@PS\W@rnmesIgn{figset proj(dist=...)}%
    \else\v@leur=#1\unit@\c@ldefdisob@{\v@leur}\fi}\ignorespaces}
\ctr@ln@m\c@lprojSP
\ctr@ln@m\CUR@proj
\ctr@ln@m\typ@proj
\ctr@ln@m\superc@lprojSP
\ctr@ld@f\def\Figs@tproj#1{%
    \if#13 \def@ultproj\else\if#1c\def@ultproj%
    \else\if#1o\xdef\CUR@proj{1}\xdef\typ@proj{orthogonal}%
         \figsetviewTD(\def@ultpsi,\def@ulttheta)%
         \global\let\c@lprojSP=\c@lprojort\global\let\superc@lprojSP=\c@lprojort%
    \else\if#1r\xdef\CUR@proj{2}\xdef\typ@proj{realistic}%
         \figsetviewTD(\def@ultpsi,\def@ulttheta)%
         \global\let\c@lprojSP=\c@lprojrea\global\let\superc@lprojSP=\c@lprojrea%
    \else\def@ultproj\message{*** Unknown projection. Cavalier projection assumed.}%
    \fi\fi\fi\fi}
\ctr@ld@f\def\def@ultproj{\xdef\CUR@proj{0}\xdef\typ@proj{cavalier}\figsetviewTD(\def@ultpsi,0.5)%
         \global\let\c@lprojSP=\c@lprojcav\global\let\superc@lprojSP=\c@lprojcav}
\ctr@ln@m\figsettarget
\ctr@ld@f\def\figsettargetDD{\un@v@ilable{figsettarget}}
\ctr@ld@f\def\figsettargetTD[#1]{{\ifCUR@PS\W@rnmesIgn{figset proj(targetpt=...)}%
    \else\global\newt@rgetpttrue\xdef\t@rgetpt{#1}\Figg@tXY{#1}\global\wd\Bt@rget=\v@lX%
    \global\ht\Bt@rget=\v@lY\global\dp\Bt@rget=\v@lZ\fi}\ignorespaces}
\ctr@ln@m\figsetview
\ctr@ld@f\def\figsetviewDD{\un@v@ilable{figsetview}}
\ctr@ld@f\def\figsetviewTD(#1){\ifCUR@PS\W@rnmesIgn{figset proj(Psi|Theta|Lambda=...)}%
     \else\Figsetview@#1,:\fi\ignorespaces}
\ctr@ld@f\def\Figsetview@#1,#2:{{\xdef\v@lPsi{#1}\def\t@xt@{#2}%
    \ifx\t@xt@\empty\def\@rgdeux{\v@lTheta}\else\X@rgdeux@#2\fi%
    \c@ssin{\costhet@}{\sinthet@}{#1}\v@lmin=\costhet@ pt\v@lmax=\sinthet@ pt%
    \ifcase\CUR@proj%
    \v@leur=\@rgdeux\v@lmin\xdef\cxa@{\repdecn@mb{\v@leur}}%
    \v@leur=\@rgdeux\v@lmax\xdef\cxb@{\repdecn@mb{\v@leur}}\v@leur=\@rgdeux pt%
    \relax\ifdim\v@leur>\p@\message{*** Lambda too large ! See \BS@ figset proj() !}\fi%
    \else%
    \v@lmax=-\v@lmax\xdef\cxa@{\repdecn@mb{\v@lmax}}\xdef\cxb@{\costhet@}%
    \ifx\t@xt@\empty\edef\@rgdeux{\def@ulttheta}\fi\c@ssin{\C@}{\S@}{\@rgdeux}%
    \v@lmax=-\S@ pt%
    \v@leur=\v@lmax\v@leur=\costhet@\v@leur\xdef\cya@{\repdecn@mb{\v@leur}}%
    \v@leur=\v@lmax\v@leur=\sinthet@\v@leur\xdef\cyb@{\repdecn@mb{\v@leur}}%
    \xdef\cyc@{\C@}\v@lmin=-\C@ pt%
    \v@leur=\v@lmin\v@leur=\costhet@\v@leur\xdef\cza@{\repdecn@mb{\v@leur}}%
    \v@leur=\v@lmin\v@leur=\sinthet@\v@leur\xdef\czb@{\repdecn@mb{\v@leur}}%
    \xdef\czc@{\repdecn@mb{\v@lmax}}\fi%
    \xdef\v@lTheta{\@rgdeux}}}
\ctr@ld@f\def\def@ultpsi{40}
\ctr@ld@f\def\def@ulttheta{25}
\ctr@ln@m\l@debut
\ctr@ln@m\n@mref
\ctr@ld@f\def\Figsetpr@j#1=#2|{\keln@mtr#1|%
    \def\n@mref{dep}\ifx\l@debut\n@mref\Figsetd@p{#2}\else% depth (lambda)
    \def\n@mref{dis}\ifx\l@debut\n@mref%
     \ifnum\CUR@proj=\tw@\figsetobdist(#2)\else\Figset@rr\fi\else% dist
    \def\n@mref{lam}\ifx\l@debut\n@mref\Figsetd@p{#2}\else% depth (lambda)
    \def\n@mref{lat}\ifx\l@debut\n@mref\Figsetth@{#2}\else% latitude (theta)
    \def\n@mref{lon}\ifx\l@debut\n@mref\figsetview(#2)\else% longitude (psi)
    \def\n@mref{psi}\ifx\l@debut\n@mref\figsetview(#2)\else% longitude (psi)
    \def\n@mref{tar}\ifx\l@debut\n@mref%
     \ifnum\CUR@proj=\tw@\figsettarget[#2]\else\Figset@rr\fi\else% target point
    \def\n@mref{the}\ifx\l@debut\n@mref\Figsetth@{#2}\else% latitude (theta)
    \W@rnmesAttr{figset proj}{#1}\fi\fi\fi\fi\fi\fi\fi\fi}
\ctr@ld@f\def\Figsetd@p#1{\ifnum\CUR@proj=\z@\figsetview(\v@lPsi,#1)\else\Figset@rr\fi}
\ctr@ld@f\def\Figsetth@#1{\ifnum\CUR@proj=\z@\Figset@rr\else\figsetview(\v@lPsi,#1)\fi}
\ctr@ld@f\def\Figset@rr{\message{*** \BS@ figset proj(): Attribute "\n@mref" ignored, incompatible
    with current projection}}
\ctr@ld@f\def\initb@undb@xTD{\wd\BminTD@=\maxdimen\ht\BminTD@=\maxdimen\dp\BminTD@=\maxdimen%
    \wd\BmaxTD@=-\maxdimen\ht\BmaxTD@=-\maxdimen\dp\BmaxTD@=-\maxdimen}
\ctr@ln@w{newbox}\Gb@x      % boite a tout faire
\ctr@ln@w{newbox}\Gb@xSC    % boite qui contient le point marker
\ctr@ln@w{newtoks}\c@nsymb  % the point marker
\ctr@ln@w{newif}\ifr@undcoord\ctr@ln@w{newif}\ifunitpr@sent
\ctr@ld@f\def\unssqrttw@{0.707106 }
\ctr@ld@f\def\figAst{\raise-1.15ex\hbox{$\ast$}}
\ctr@ld@f\def\figBullet{\raise-1.15ex\hbox{$\bullet$}}
\ctr@ld@f\def\figCirc{\raise-1.15ex\hbox{$\circ$}}
\ctr@ld@f\def\figDiamond{\raise-1.15ex\hbox{$\diamond$}}%
\ctr@ld@f\def\boxit#1#2{\leavevmode\hbox{\vrule\vbox{\hrule\vglue#1%
    \vtop{\hbox{\kern#1{#2}\kern#1}\vglue#1\hrule}}\vrule}}
\ctr@ld@f
\ctr@ld@f
\ctr@ld@f\def\c@nterpt{\ignorespaces%
    \kern-.5\wd\Gb@xSC%
    \raise-.5\ht\Gb@xSC\rlap{\hbox{\raise.5\dp\Gb@xSC\hbox{\copy\Gb@xSC}}}%
    \kern .5\wd\Gb@xSC\ignorespaces}
\ctr@ld@f\def\b@undb@xSC#1#2{{\v@lXa=#1\v@lYa=#2%
    \v@leur=\ht\Gb@xSC\advance\v@leur\dp\Gb@xSC%
    \advance\v@lXa-.5\wd\Gb@xSC\advance\v@lYa-.5\v@leur\b@undb@x{\v@lXa}{\v@lYa}%
    \advance\v@lXa\wd\Gb@xSC\advance\v@lYa\v@leur\b@undb@x{\v@lXa}{\v@lYa}}}
\ctr@ln@m\Dist@n
\ctr@ln@m\l@suite
\ctr@ld@f\def\@keldist#1#2{\edef\Dist@n{#2}\y@tiunit{\Dist@n}%
    \ifunitpr@sent#1=\Dist@n\else#1=\Dist@n\unit@\fi}
\ctr@ld@f\def\y@tiunit#1{\unitpr@sentfalse\expandafter\y@tiunit@#1:}
\ctr@ld@f\def\y@tiunit@#1#2:{\ifcat#1a\unitpr@senttrue\else\def\l@suite{#2}%
    \ifx\l@suite\empty\else\y@tiunit@#2:\fi\fi}
\ctr@ln@m\figcoord
\ctr@ld@f\def\figcoordDD#1{{\v@lX=\ptT@unit@\v@lX\v@lY=\ptT@unit@\v@lY%
    \ifr@undcoord\ifcase#1\v@leur=0.5pt\or\v@leur=0.05pt\or\v@leur=0.005pt%
    \or\v@leur=0.0005pt\else\v@leur=\z@\fi%
    \ifdim\v@lX<\z@\advance\v@lX-\v@leur\else\advance\v@lX\v@leur\fi%
    \ifdim\v@lY<\z@\advance\v@lY-\v@leur\else\advance\v@lY\v@leur\fi\fi%
    (\@ffichnb{#1}{\repdecn@mb{\v@lX}},\ifmmode\else\thinspace\fi%
    \@ffichnb{#1}{\repdecn@mb{\v@lY}})}}
\ctr@ld@f\def\@ffichnb#1#2{{\def\@@ffich{\@ffich#1(}\edef\n@mbre{#2}%
    \expandafter\@@ffich\n@mbre)}}
\ctr@ld@f\def\@ffich#1(#2.#3){{#2\ifnum#1>\z@.\fi\def\dig@ts{#3}\s@mme=\z@%
    \loop\ifnum\s@mme<#1\expandafter\@ffichdec\dig@ts:\advance\s@mme\@ne\repeat}}
\ctr@ld@f\def\@ffichdec#1#2:{\relax#1\def\dig@ts{#20}}
\ctr@ld@f\def\figcoordTD#1{{\v@lX=\ptT@unit@\v@lX\v@lY=\ptT@unit@\v@lY\v@lZ=\ptT@unit@\v@lZ%
    \ifr@undcoord\ifcase#1\v@leur=0.5pt\or\v@leur=0.05pt\or\v@leur=0.005pt%
    \or\v@leur=0.0005pt\else\v@leur=\z@\fi%
    \ifdim\v@lX<\z@\advance\v@lX-\v@leur\else\advance\v@lX\v@leur\fi%
    \ifdim\v@lY<\z@\advance\v@lY-\v@leur\else\advance\v@lY\v@leur\fi%
    \ifdim\v@lZ<\z@\advance\v@lZ-\v@leur\else\advance\v@lZ\v@leur\fi\fi%
    (\@ffichnb{#1}{\repdecn@mb{\v@lX}},\ifmmode\else\thinspace\fi%
     \@ffichnb{#1}{\repdecn@mb{\v@lY}},\ifmmode\else\thinspace\fi%
     \@ffichnb{#1}{\repdecn@mb{\v@lZ}})}}
\ctr@ld@f\def\figsetroundcoord#1{\expandafter\Figsetr@undcoord#1:\ignorespaces}
\ctr@ld@f\def\Figsetr@undcoord#1#2:{\if#1n\r@undcoordfalse\else\r@undcoordtrue\fi}
\ctr@ld@f\def\Figsetwr@te#1=#2|{\keln@mun#1|%
    \def\n@mref{m}\ifx\l@debut\n@mref\figsetmark{#2}\else% mark
    \def\n@mref{p}\ifx\l@debut\n@mref\figsetptname{#2}\else% ptname
    \def\n@mref{r}\ifx\l@debut\n@mref\figsetroundcoord{#2}\else% roundcoord
    \W@rnmesAttr{figset write}{#1}\fi\fi\fi}
\ctr@ld@f\def\figsetmark#1{\c@nsymb={#1}\setbox\Gb@xSC=\hbox{\the\c@nsymb}\ignorespaces}
\ctr@ln@m\ptn@me
\ctr@ld@f\def\figsetptname#1{\def\ptn@me##1{#1}\ignorespaces}
\ctr@ld@f\def\FigWrit@L#1:#2(#3,#4){\ignorespaces\@keldist\v@leur{#3}\@keldist\delt@{#4}%
    \C@rp@r@m\def\list@num{#1}\@ecfor\p@int:=\list@num\do{\FigWrit@pt{\p@int}{#2}}}
\ctr@ld@f\def\FigWrit@pt#1#2{\FigWp@r@m{#1}{#2}\Vc@rrect\figWp@si%
    \ifdim\wd\Gb@xSC>\z@\b@undb@xSC{\v@lX}{\v@lY}\fi\figWBB@x}
\ctr@ld@f\def\FigWp@r@m#1#2{\Figg@tXY{#1}%
    \setbox\Gb@x=\hbox{\def\t@xt@{#2}\ifx\t@xt@\empty\Figg@tT{#1}\else#2\fi}\c@lprojSP}
\ctr@ld@f\let\Vc@rrect=\relax
\ctr@ld@f\let\C@rp@r@m=\relax
\ctr@ld@f\def\figwrite[#1]#2{{\ignorespaces\def\list@num{#1}\@ecfor\p@int:=\list@num\do{%
    \setbox\Gb@x=\hbox{\def\t@xt@{#2}\ifx\t@xt@\empty\Figg@tT{\p@int}\else#2\fi}%
    \Figwrit@{\p@int}}}\ignorespaces}
\ctr@ld@f\def\Figwrit@#1{\Figg@tXY{#1}\c@lprojSP%
    \rlap{\kern\v@lX\raise\v@lY\hbox{\unhcopy\Gb@x}}\v@leur=\v@lY%
    \advance\v@lY\ht\Gb@x\b@undb@x{\v@lX}{\v@lY}\advance\v@lX\wd\Gb@x%
    \v@lY=\v@leur\advance\v@lY-\dp\Gb@x\b@undb@x{\v@lX}{\v@lY}}
\ctr@ld@f\def\figwritec[#1]#2{{\ignorespaces\def\list@num{#1}%
    \@ecfor\p@int:=\list@num\do{\Figwrit@c{\p@int}{#2}}}\ignorespaces}
\ctr@ld@f\def\Figwrit@c#1#2{\FigWp@r@m{#1}{#2}%
    \rlap{\kern\v@lX\raise\v@lY\hbox{\rlap{\kern-.5\wd\Gb@x%
    \raise-.5\ht\Gb@x\hbox{\raise.5\dp\Gb@x\hbox{\unhcopy\Gb@x}}}}}%
    \v@leur=\ht\Gb@x\advance\v@leur\dp\Gb@x%
    \advance\v@lX-.5\wd\Gb@x\advance\v@lY-.5\v@leur\b@undb@x{\v@lX}{\v@lY}%
    \advance\v@lX\wd\Gb@x\advance\v@lY\v@leur\b@undb@x{\v@lX}{\v@lY}}
\ctr@ld@f\def\figwritep[#1]{{\ignorespaces\def\list@num{#1}\setbox\Gb@x=\hbox{\c@nterpt}%
    \@ecfor\p@int:=\list@num\do{\Figwrit@{\p@int}}}\ignorespaces}
\ctr@ld@f\def\figwritew#1:#2(#3){\figwritegcw#1:{#2}(#3,0pt)}
\ctr@ld@f\def\figwritee#1:#2(#3){\figwritegce#1:{#2}(#3,0pt)}
\ctr@ld@f\def\figwriten#1:#2(#3){{\def\Vc@rrect{\v@lZ=\v@leur\advance\v@lZ\dp\Gb@x}%
    \Figwrit@NS#1:{#2}(#3)}\ignorespaces}
\ctr@ld@f\def\figwrites#1:#2(#3){{\def\Vc@rrect{\v@lZ=-\v@leur\advance\v@lZ-\ht\Gb@x}%
    \Figwrit@NS#1:{#2}(#3)}\ignorespaces}
\ctr@ld@f\def\Figwrit@NS#1:#2(#3){\let\figWp@si=\FigWp@siNS\let\figWBB@x=\FigWBB@xNS%
    \FigWrit@L#1:{#2}(#3,0pt)}
\ctr@ld@f\def\FigWp@siNS{\rlap{\kern\v@lX\raise\v@lY\hbox{\rlap{\kern-.5\wd\Gb@x%
    \raise\v@lZ\hbox{\unhcopy\Gb@x}}\c@nterpt}}}
\ctr@ld@f\def\FigWBB@xNS{\advance\v@lY\v@lZ%
    \advance\v@lY-\dp\Gb@x\advance\v@lX-.5\wd\Gb@x\b@undb@x{\v@lX}{\v@lY}%
    \advance\v@lY\ht\Gb@x\advance\v@lY\dp\Gb@x%
    \advance\v@lX\wd\Gb@x\b@undb@x{\v@lX}{\v@lY}}
\ctr@ld@f\def\figwritenw#1:#2(#3){{\let\figWp@si=\FigWp@sigW\let\figWBB@x=\FigWBB@xgWE%
    \def\C@rp@r@m{\v@leur=\unssqrttw@\v@leur\delt@=\v@leur%
    \ifdim\delt@=\z@\delt@=\epsil@n\fi}\let@xte={-}\FigWrit@L#1:{#2}(#3,0pt)}\ignorespaces}
\ctr@ld@f\def\figwritesw#1:#2(#3){{\let\figWp@si=\FigWp@sigW\let\figWBB@x=\FigWBB@xgWE%
    \def\C@rp@r@m{\v@leur=\unssqrttw@\v@leur\delt@=-\v@leur%
    \ifdim\delt@=\z@\delt@=-\epsil@n\fi}\let@xte={-}\FigWrit@L#1:{#2}(#3,0pt)}\ignorespaces}
\ctr@ld@f\def\figwritene#1:#2(#3){{\let\figWp@si=\FigWp@sigE\let\figWBB@x=\FigWBB@xgWE%
    \def\C@rp@r@m{\v@leur=\unssqrttw@\v@leur\delt@=\v@leur%
    \ifdim\delt@=\z@\delt@=\epsil@n\fi}\let@xte={}\FigWrit@L#1:{#2}(#3,0pt)}\ignorespaces}
\ctr@ld@f\def\figwritese#1:#2(#3){{\let\figWp@si=\FigWp@sigE\let\figWBB@x=\FigWBB@xgWE%
    \def\C@rp@r@m{\v@leur=\unssqrttw@\v@leur\delt@=-\v@leur%
    \ifdim\delt@=\z@\delt@=-\epsil@n\fi}\let@xte={}\FigWrit@L#1:{#2}(#3,0pt)}\ignorespaces}
\ctr@ld@f\def\figwritegw#1:#2(#3,#4){{\let\figWp@si=\FigWp@sigW\let\figWBB@x=\FigWBB@xgWE%
    \let@xte={-}\FigWrit@L#1:{#2}(#3,#4)}\ignorespaces}
\ctr@ld@f\def\figwritege#1:#2(#3,#4){{\let\figWp@si=\FigWp@sigE\let\figWBB@x=\FigWBB@xgWE%
    \let@xte={}\FigWrit@L#1:{#2}(#3,#4)}\ignorespaces}
\ctr@ld@f\def\FigWp@sigW{\v@lXa=\z@\v@lYa=\ht\Gb@x\advance\v@lYa\dp\Gb@x%
    \ifdim\delt@>\z@\relax%
    \rlap{\kern\v@lX\raise\v@lY\hbox{\rlap{\kern-\wd\Gb@x\kern-\v@leur%
          \raise\delt@\hbox{\raise\dp\Gb@x\hbox{\unhcopy\Gb@x}}}\c@nterpt}}%
    \else\ifdim\delt@<\z@\relax\v@lYa=-\v@lYa%
    \rlap{\kern\v@lX\raise\v@lY\hbox{\rlap{\kern-\wd\Gb@x\kern-\v@leur%
          \raise\delt@\hbox{\raise-\ht\Gb@x\hbox{\unhcopy\Gb@x}}}\c@nterpt}}%
    \else\v@lXa=-.5\v@lYa%
    \rlap{\kern\v@lX\raise\v@lY\hbox{\rlap{\kern-\wd\Gb@x\kern-\v@leur%
          \raise-.5\ht\Gb@x\hbox{\raise.5\dp\Gb@x\hbox{\unhcopy\Gb@x}}}\c@nterpt}}%
    \fi\fi}
\ctr@ld@f\def\FigWp@sigE{\v@lXa=\z@\v@lYa=\ht\Gb@x\advance\v@lYa\dp\Gb@x%
    \ifdim\delt@>\z@\relax%
    \rlap{\kern\v@lX\raise\v@lY\hbox{\c@nterpt\kern\v@leur%
          \raise\delt@\hbox{\raise\dp\Gb@x\hbox{\unhcopy\Gb@x}}}}%
    \else\ifdim\delt@<\z@\relax\v@lYa=-\v@lYa%
    \rlap{\kern\v@lX\raise\v@lY\hbox{\c@nterpt\kern\v@leur%
          \raise\delt@\hbox{\raise-\ht\Gb@x\hbox{\unhcopy\Gb@x}}}}%
    \else\v@lXa=-.5\v@lYa%
    \rlap{\kern\v@lX\raise\v@lY\hbox{\c@nterpt\kern\v@leur%
          \raise-.5\ht\Gb@x\hbox{\raise.5\dp\Gb@x\hbox{\unhcopy\Gb@x}}}}%
    \fi\fi}
\ctr@ld@f\def\FigWBB@xgWE{\advance\v@lY\delt@%
    \advance\v@lX\the\let@xte\v@leur\advance\v@lY\v@lXa\b@undb@x{\v@lX}{\v@lY}%
    \advance\v@lX\the\let@xte\wd\Gb@x\advance\v@lY\v@lYa\b@undb@x{\v@lX}{\v@lY}}
\ctr@ld@f\def\figwritegcw#1:#2(#3,#4){{\let\figWp@si=\FigWp@sigcW\let\figWBB@x=\FigWBB@xgcWE%
    \let@xte={-}\FigWrit@L#1:{#2}(#3,#4)}\ignorespaces}
\ctr@ld@f\def\figwritegce#1:#2(#3,#4){{\let\figWp@si=\FigWp@sigcE\let\figWBB@x=\FigWBB@xgcWE%
    \let@xte={}\FigWrit@L#1:{#2}(#3,#4)}\ignorespaces}
\ctr@ld@f\def\FigWp@sigcW{\rlap{\kern\v@lX\raise\v@lY\hbox{\rlap{\kern-\wd\Gb@x\kern-\v@leur%
     \raise-.5\ht\Gb@x\hbox{\raise\delt@\hbox{\raise.5\dp\Gb@x\hbox{\unhcopy\Gb@x}}}}%
     \c@nterpt}}}
\ctr@ld@f\def\FigWp@sigcE{\rlap{\kern\v@lX\raise\v@lY\hbox{\c@nterpt\kern\v@leur%
    \raise-.5\ht\Gb@x\hbox{\raise\delt@\hbox{\raise.5\dp\Gb@x\hbox{\unhcopy\Gb@x}}}}}}
\ctr@ld@f\def\FigWBB@xgcWE{\v@lZ=\ht\Gb@x\advance\v@lZ\dp\Gb@x%
    \advance\v@lX\the\let@xte\v@leur\advance\v@lY\delt@\advance\v@lY.5\v@lZ%
    \b@undb@x{\v@lX}{\v@lY}%
    \advance\v@lX\the\let@xte\wd\Gb@x\advance\v@lY-\v@lZ\b@undb@x{\v@lX}{\v@lY}}
\ctr@ld@f\def\figwritebn#1:#2(#3){{\def\Vc@rrect{\v@lZ=\v@leur}\Figwrit@NS#1:{#2}(#3)}\ignorespaces}
\ctr@ld@f\def\figwritebs#1:#2(#3){{\def\Vc@rrect{\v@lZ=-\v@leur}\Figwrit@NS#1:{#2}(#3)}\ignorespaces}
\ctr@ld@f\def\figwritebw#1:#2(#3){{\let\figWp@si=\FigWp@sibW\let\figWBB@x=\FigWBB@xbWE%
    \let@xte={-}\FigWrit@L#1:{#2}(#3,0pt)}\ignorespaces}
\ctr@ld@f\def\figwritebe#1:#2(#3){{\let\figWp@si=\FigWp@sibE\let\figWBB@x=\FigWBB@xbWE%
    \let@xte={}\FigWrit@L#1:{#2}(#3,0pt)}\ignorespaces}
\ctr@ld@f\def\FigWp@sibW{\rlap{\kern\v@lX\raise\v@lY\hbox{\rlap{\kern-\wd\Gb@x\kern-\v@leur%
          \hbox{\unhcopy\Gb@x}}\c@nterpt}}}
\ctr@ld@f\def\FigWp@sibE{\rlap{\kern\v@lX\raise\v@lY\hbox{\c@nterpt\kern\v@leur%
          \hbox{\unhcopy\Gb@x}}}}
\ctr@ld@f\def\FigWBB@xbWE{\v@lZ=\ht\Gb@x\advance\v@lZ\dp\Gb@x%
    \advance\v@lX\the\let@xte\v@leur\advance\v@lY\ht\Gb@x\b@undb@x{\v@lX}{\v@lY}%
    \advance\v@lX\the\let@xte\wd\Gb@x\advance\v@lY-\v@lZ\b@undb@x{\v@lX}{\v@lY}}
\ctr@ln@w{newread}\frf@g  \ctr@ln@w{newwrite}\fwf@g
\ctr@ln@w{newif}\ifCUR@PS
\ctr@ln@w{newif}\ifGR@cri
\ctr@ln@w{newif}\ifUse@llipse
\ctr@ln@w{newif}\ifGRdebugm@de \GRdebugm@defalse 
\ctr@ln@w{newif}\ifPDFm@ke
\ifx\pdfliteral\undefined\else\ifnum\pdfoutput>\z@\PDFm@ketrue\fi\fi
\ctr@ld@f\def\initPDF@rDVI{%
\ifPDFm@ke
 \let\figscan=\figscan@E
 \let\newGr@FN=\newGr@FNPDF
 \ctr@ld@f\def\c@mcurveto{c}
 \ctr@ld@f\def\c@mfill{f}
 \ctr@ld@f\def\c@mgsave{q}
 \ctr@ld@f\def\c@mgrestore{Q}
 \ctr@ld@f\def\c@mlineto{l}
 \ctr@ld@f\def\c@mmoveto{m}
 \ctr@ld@f\def\c@msetgray{g}     \ctr@ld@f\def\c@msetgrayStroke{G}
 \ctr@ld@f\def\c@msetcmykcolor{k}\ctr@ld@f\def\c@msetcmykcolorStroke{K}
 \ctr@ld@f\def\c@msetrgbcolor{rg}\ctr@ld@f\def\c@msetrgbcolorStroke{RG}
 \ctr@ld@f\def\d@fprimarC@lor{\CUR@color\space\CUR@colorc@md%
               \space\CUR@color\space\CUR@colorc@mdStroke}
 \ctr@ld@f\def\c@msetdash{d}
 \ctr@ld@f\def\c@msetlinejoin{j}
 \ctr@ld@f\def\c@msetlinewidth{w}
 \ctr@ld@f\def\f@gclosestroke{\immediate\write\fwf@g{s}}
 \ctr@ld@f\def\f@gfill{\immediate\write\fwf@g{\fillc@md}}% Voir la def de \fillc@md ****
 \ctr@ld@f\def\f@gnewpath{}
 \ctr@ld@f\def\f@gstroke{\immediate\write\fwf@g{S}}
\else
 \let\figinsertE=\figinsert
 \let\newGr@FN=\newGr@FNDVI
 \ctr@ld@f\def\c@mcurveto{curveto}
 \ctr@ld@f\def\c@mfill{fill}
 \ctr@ld@f\def\c@mgsave{gsave}
 \ctr@ld@f\def\c@mgrestore{grestore}
 \ctr@ld@f\def\c@mlineto{lineto}
 \ctr@ld@f\def\c@mmoveto{moveto}
 \ctr@ld@f\def\c@msetgray{setgray}          \ctr@ld@f\def\c@msetgrayStroke{}
 \ctr@ld@f\def\c@msetcmykcolor{setcmykcolor}\ctr@ld@f\def\c@msetcmykcolorStroke{}
 \ctr@ld@f\def\c@msetrgbcolor{setrgbcolor}  \ctr@ld@f\def\c@msetrgbcolorStroke{}
 \ctr@ld@f\def\d@fprimarC@lor{\CUR@color\space\CUR@colorc@md}
 \ctr@ld@f\def\c@msetdash{setdash}
 \ctr@ld@f\def\c@msetlinejoin{setlinejoin}
 \ctr@ld@f\def\c@msetlinewidth{setlinewidth}
 \ctr@ld@f\def\f@gclosestroke{\immediate\write\fwf@g{closepath\space stroke}}
 \ctr@ld@f\def\f@gfill{\immediate\write\fwf@g{\fillc@md}}
 \ctr@ld@f\def\f@gnewpath{\immediate\write\fwf@g{newpath}}
 \ctr@ld@f\def\f@gstroke{\immediate\write\fwf@g{stroke}}
\fi}
\ctr@ld@f\def\c@pypsfile#1#2{\c@pyfil@{\immediate\write#1}{#2}}
\ctr@ld@f\def\Figinclud@PDF#1#2{\openin\frf@g=#1\pdfliteral{q #2 0 0 #2 0 0 cm}%
    \c@pyfil@{\pdfliteral}{\frf@g}\pdfliteral{Q}\closein\frf@g}
\ctr@ln@w{newif}\ifmored@ta
\ctr@ln@m\bl@nkline
\ctr@ld@f\def\c@pyfil@#1#2{\def\bl@nkline{\par}{\catcode`\%=12
    \loop\ifeof#2\mored@tafalse\else\mored@tatrue\immediate\read#2 to\tr@c
    \ifx\tr@c\bl@nkline\else#1{\tr@c}\fi\fi\ifmored@ta\repeat}}
\ctr@ld@f\def\keln@mun#1#2|{\def\l@debut{#1}\def\l@suite{#2}}
\ctr@ld@f\def\keln@mde#1#2#3|{\def\l@debut{#1#2}\def\l@suite{#3}}
\ctr@ld@f\def\keln@mtr#1#2#3#4|{\def\l@debut{#1#2#3}\def\l@suite{#4}}
\ctr@ld@f\def\keln@mqu#1#2#3#4#5|{\def\l@debut{#1#2#3#4}\def\l@suite{#5}}
\ctr@ld@f\let\@psffilein=\frf@g % file to \read
\ctr@ln@w{newif}\if@psffileok    % continue looking for the bounding box?
\ctr@ln@w{newif}\if@psfbbfound   % success?
\ctr@ln@w{newif}\if@psfverbose   % report what you're making?
\@psfverbosetrue
\ctr@ln@m\@psfllx \ctr@ln@m\@psflly
\ctr@ln@m\@psfurx \ctr@ln@m\@psfury
\ctr@ln@m\resetcolonc@tcode
\ctr@ld@f\def\@psfgetbb#1{\global\@psfbbfoundfalse%
\global\def\@psfllx{0}\global\def\@psflly{0}%
\global\def\@psfurx{30}\global\def\@psfury{30}%
\openin\@psffilein=#1\relax
\ifeof\@psffilein\errmessage{I couldn't open #1, will ignore it}\else
   \edef\resetcolonc@tcode{\catcode`\noexpand\:\the\catcode`\:\relax}%
   {\@psffileoktrue \chardef\other=12
    \def\do##1{\catcode`##1=\other}\dospecials \catcode`\ =10 \resetcolonc@tcode
    \loop
       \read\@psffilein to \@psffileline
       \ifeof\@psffilein\@psffileokfalse\else
          \expandafter\@psfaux\@psffileline:. \\%
       \fi
   \if@psffileok\repeat
   \if@psfbbfound\else
    \if@psfverbose\message{No bounding box comment in #1; using defaults}\fi\fi
   }\closein\@psffilein\fi}%
\ctr@ln@m\@psfbblit
\ctr@ln@m\@psfpercent
{\catcode`\%=12 \global\let\@psfpercent=%\global\def\@psfbblit{%BoundingBox}}%
\ctr@ln@m\@psfaux
\long\def\@psfaux#1#2:#3\\{\ifx#1\@psfpercent
   \def\testit{#2}\ifx\testit\@psfbblit
      \@psfgrab #3 . . . \\%
      \@psffileokfalse
      \global\@psfbbfoundtrue
   \fi\else\ifx#1\par\else\@psffileokfalse\fi\fi}%
\ctr@ld@f\def\@psfempty{}%
\ctr@ld@f\def\@psfgrab #1 #2 #3 #4 #5\\{%
\global\def\@psfllx{#1}\ifx\@psfllx\@psfempty
      \@psfgrab #2 #3 #4 #5 .\\\else
   \global\def\@psflly{#2}%
   \global\def\@psfurx{#3}\global\def\@psfury{#4}\fi}%
\ctr@ld@f\def\PSwrit@cmd#1#2#3{{\Figg@tXY{#1}\c@lprojSP\b@undb@x{\v@lX}{\v@lY}%
    \v@lX=\ptT@ptps\v@lX\v@lY=\ptT@ptps\v@lY%
    \immediate\write#3{\repdecn@mb{\v@lX}\space\repdecn@mb{\v@lY}\space#2}}}
\ctr@ld@f\def\PSwrit@cmdS#1#2#3#4#5{{\Figg@tXY{#1}\c@lprojSP\b@undb@x{\v@lX}{\v@lY}%
    \global\result@t=\v@lX\global\result@@t=\v@lY%
    \v@lX=\ptT@ptps\v@lX\v@lY=\ptT@ptps\v@lY%
    \immediate\write#3{\repdecn@mb{\v@lX}\space\repdecn@mb{\v@lY}\space#2}}%
    \edef#4{\the\result@t}\edef#5{\the\result@@t}}
\ctr@ld@f\def\update@ttr#1#2#3{\Figdisc@rdLTS{#3}{\n@mref}%
    \ifx\n@mref\D@FTref#2{#1}\else#2{#3}\fi}
\ctr@ld@f\def\D@FTref{default}
\ctr@ld@f\def\W@rnmesAttr#1#2{%
    \immediate\write16{*** Unknown attribute: \BS@ #1(..., #2=...)}}
\ctr@ld@f\def\W@rnmeskwd#1#2{%
    \immediate\write16{*** Unknown keyword #2 in \BS@ #1}}
\ctr@ld@f\def\W@rnmesIgn#1{\immediate\write16{*** \BS@ #1 is ignored inside a
     \BS@ figdrawbegin-\BS@ figdrawend block.}}
\ctr@ld@f\def\Psset@lti#1=#2|{\keln@mtr#1|%
    \def\n@mref{blc}\ifx\l@debut\n@mref\update@ttr\D@FTref\P@setblcolor{#2}\else% base line color
    \def\n@mref{bld}\ifx\l@debut\n@mref\update@ttr\D@FTref\P@setbldash{#2}\else% base line dash
    \def\n@mref{blw}\ifx\l@debut\n@mref\update@ttr\D@FTref\P@setblwidth{#2}\else% base line width
    \def\n@mref{sqc}\ifx\l@debut\n@mref\update@ttr\D@FTref\P@setsqcolor{#2}\else% square color
    \def\n@mref{sqd}\ifx\l@debut\n@mref\update@ttr\D@FTref\P@setsqdash{#2}\else% square dash
    \def\n@mref{sqw}\ifx\l@debut\n@mref\update@ttr\D@FTref\P@setsqwidth{#2}\else% square width
    \W@rnmesAttr{figset altitude}{#1}\fi\fi\fi\fi\fi\fi}
\ctr@ln@m\DDV@blcolor
\ctr@ld@f\def\P@setblcolor#1{\edef\DDV@blcolor{#1}}
\ctr@ln@m\DDV@bldash
\ctr@ld@f\def\P@setbldash#1{\edef\DDV@bldash{#1}}
\ctr@ln@m\DDV@blwidth
\ctr@ld@f\def\P@setblwidth#1{\edef\DDV@blwidth{#1}}
\ctr@ln@m\DDV@sqcolor
\ctr@ld@f\def\P@setsqcolor#1{\edef\DDV@sqcolor{#1}}
\ctr@ln@m\DDV@sqdash
\ctr@ld@f\def\P@setsqdash#1{\edef\DDV@sqdash{#1}}
\ctr@ln@m\DDV@sqwidth
\ctr@ld@f\def\P@setsqwidth#1{\edef\DDV@sqwidth{#1}}
\ctr@ld@f\def\figdrawaltitude#1[#2,#3,#4]{{\ifCUR@PS\ifGR@cri%
    \PSc@mment{altitude Square Dim=#1, Triangle=[#2 / #3,#4]}%
    \s@uvc@ntr@l\et@tpsaltitude\resetc@ntr@l{2}\figptorthoprojline-5:=#2/#3,#4/%
    \figvectP -1[#3,#4]\n@rminf{\v@leur}{-1}\vecunit@{-3}{-1}%
    \figvectP -1[-5,#3]\n@rminf{\v@lmin}{-1}\figvectP -2[-5,#4]\n@rminf{\v@lmax}{-2}%
    \ifdim\v@lmin<\v@lmax\s@mme=#3\else\v@lmax=\v@lmin\s@mme=#4\fi%
    \figvectP -4[-5,#2]\vecunit@{-4}{-4}\delt@=#1\unit@%
    \edef\t@ille{\repdecn@mb{\delt@}}\figpttra-1:=-5/\t@ille,-3/%
    \figptstra-3=-5,-1/\t@ille,-4/\figdrawline[#2,-5]%
    \Pss@tspecifSt{color=\DDV@sqcolor,dash=\DDV@sqdash,width=\DDV@sqwidth}%
    \figdrawline[-1,-2,-3]%
    \Psrest@reSt{color=\DDV@sqcolor,dash=\DDV@sqdash,width=\DDV@sqwidth}%
    \ifdim\v@leur<\v@lmax%
    \Pss@tspecifSt{color=\DDV@blcolor,dash=\DDV@bldash,width=\DDV@blwidth}%
    \figdrawline[-5,\the\s@mme]%
    \Psrest@reSt{color=\DDV@blcolor,dash=\DDV@bldash,width=\DDV@blwidth}%
    \fi\PSc@mment{End altitude}\resetc@ntr@l\et@tpsaltitude\fi\fi}}
\ctr@ld@f\def\Ps@rcerc#1;#2(#3,#4){\ellBB@x#1;#2,#2(#3,#4,0)%
    \f@gnewpath{\delt@=#2\unit@\delt@=\ptT@ptps\delt@%
    \BdingB@xfalse%
    \PSwrit@cmd{#1}{\repdecn@mb{\delt@}\space #3\space #4\space arc}{\fwf@g}}}
\ctr@ln@m\figdrawarccirc
\ctr@ld@f\def\Q@arccircDD#1;#2(#3,#4){\ifCUR@PS\ifGR@cri%
    \PSc@mment{arccircDD Center=#1 ; Radius=#2 (Ang1=#3, Ang2=#4)}%
    \iffillm@de\Ps@rcerc#1;#2(#3,#4)%
    \f@gfill%
    \else\Ps@rcerc#1;#2(#3,#4)\f@gstroke\fi%
    \PSc@mment{End arccircDD}\fi\fi}
\ctr@ld@f\def\Q@arccircTD#1,#2,#3;#4(#5,#6){{\ifCUR@PS\ifGR@cri\s@uvc@ntr@l\et@tpsarccircTD%
    \PSc@mment{arccircTD Center=#1,P1=#2,P2=#3 ; Radius=#4 (Ang1=#5, Ang2=#6)}%
    \setc@ntr@l{2}\c@lExtAxes#1,#2,#3(#4)\Q@arcellPATD#1,-4,-5(#5,#6)%
    \PSc@mment{End arccircTD}\resetc@ntr@l\et@tpsarccircTD\fi\fi}}
\ctr@ld@f\def\c@lExtAxes#1,#2,#3(#4){%
    \figvectPTD-5[#1,#2]\vecunit@{-5}{-5}\figvectNTD-4[#1,#2,#3]\vecunit@{-4}{-4}%
    \figvectNVTD-3[-4,-5]\delt@=#4\unit@\edef\r@yon{\repdecn@mb{\delt@}}%
    \figpttra-4:=#1/\r@yon,-5/\figpttra-5:=#1/\r@yon,-3/}
\ctr@ln@m\figdrawarccircP
\ctr@ld@f\def\Q@arccircPDD#1;#2[#3,#4]{{\ifCUR@PS\ifGR@cri\s@uvc@ntr@l\et@tpsarccircPDD%
    \PSc@mment{arccircPDD Center=#1; Radius=#2, [P1=#3, P2=#4]}%
    \Ps@ngleparam#1;#2[#3,#4]\ifdim\v@lmin>\v@lmax\advance\v@lmax\DePI@deg\fi%
    \edef\@ngdeb{\repdecn@mb{\v@lmin}}\edef\@ngfin{\repdecn@mb{\v@lmax}}%
    \figdrawarccirc#1;\r@dius(\@ngdeb,\@ngfin)%
    \PSc@mment{End arccircPDD}\resetc@ntr@l\et@tpsarccircPDD\fi\fi}}
\ctr@ld@f\def\Q@arccircPTD#1;#2[#3,#4,#5]{{\ifCUR@PS\ifGR@cri\s@uvc@ntr@l\et@tpsarccircPTD%
    \PSc@mment{arccircPTD Center=#1; Radius=#2, [P1=#3, P2=#4, P3=#5]}%
    \setc@ntr@l{2}\c@lExtAxes#1,#3,#5(#2)\figdrawarcellPP#1,-4,-5[#3,#4]%
    \PSc@mment{End arccircPTD}\resetc@ntr@l\et@tpsarccircPTD\fi\fi}}
\ctr@ld@f\def\Ps@ngleparam#1;#2[#3,#4]{\setc@ntr@l{2}%
    \figvectPDD-1[#1,#3]\vecunit@{-1}{-1}\Figg@tXY{-1}\arct@n\v@lmin(\v@lX,\v@lY)%
    \figvectPDD-2[#1,#4]\vecunit@{-2}{-2}\Figg@tXY{-2}\arct@n\v@lmax(\v@lX,\v@lY)%
    \v@lmin=\rdT@deg\v@lmin\v@lmax=\rdT@deg\v@lmax%
    \v@leur=#2pt\maxim@m{\mili@u}{-\v@leur}{\v@leur}%
    \edef\r@dius{\repdecn@mb{\mili@u}}}
\ctr@ld@f\def\Ps@rcercBz#1;#2(#3,#4){\Ps@rellBz#1;#2,#2(#3,#4,0)}
\ctr@ld@f\def\Ps@rellBz#1;#2,#3(#4,#5,#6){%
    \ellBB@x#1;#2,#3(#4,#5,#6)\BdingB@xfalse%
    \c@lNbarcs{#4}{#5}\v@leur=#4pt\setc@ntr@l{2}\figptell-13::#1;#2,#3(#4,#6)%
    \f@gnewpath\PSwrit@cmd{-13}{\c@mmoveto}{\fwf@g}%
    \s@mme=\z@\bcl@rellBz#1;#2,#3(#6)\BdingB@xtrue}
\ctr@ld@f\def\bcl@rellBz#1;#2,#3(#4){\relax%
    \ifnum\s@mme<\p@rtent\advance\s@mme\@ne%
    \advance\v@leur\delt@\edef\@ngle{\repdecn@mb\v@leur}\figptell-14::#1;#2,#3(\@ngle,#4)%
    \advance\v@leur\delt@\edef\@ngle{\repdecn@mb\v@leur}\figptell-15::#1;#2,#3(\@ngle,#4)%
    \advance\v@leur\delt@\edef\@ngle{\repdecn@mb\v@leur}\figptell-16::#1;#2,#3(\@ngle,#4)%
    \figptscontrolDD-18[-13,-14,-15,-16]%
    \PSwrit@cmd{-18}{}{\fwf@g}\PSwrit@cmd{-17}{}{\fwf@g}%
    \PSwrit@cmd{-16}{\c@mcurveto}{\fwf@g}%
    \figptcopyDD-13:/-16/\bcl@rellBz#1;#2,#3(#4)\fi}
\ctr@ld@f\def\Ps@rell#1;#2,#3(#4,#5,#6){\ellBB@x#1;#2,#3(#4,#5,#6)%
    \f@gnewpath{\v@lmin=#2\unit@\v@lmin=\ptT@ptps\v@lmin%
    \v@lmax=#3\unit@\v@lmax=\ptT@ptps\v@lmax\BdingB@xfalse%
    \PSwrit@cmd{#1}%
    {#6\space\repdecn@mb{\v@lmin}\space\repdecn@mb{\v@lmax}\space #4\space #5\space ellipse}{\fwf@g}}%
    \global\Use@llipsetrue}
\ctr@ln@m\figdrawarcell
\ctr@ld@f\def\Q@arcellDD#1;#2,#3(#4,#5,#6){{\ifCUR@PS\ifGR@cri%
    \PSc@mment{arcellDD Center=#1 ; XRad=#2, YRad=#3 (Ang1=#4, Ang2=#5, Inclination=#6)}%
    \iffillm@de\Ps@rell#1;#2,#3(#4,#5,#6)%
    \f@gfill%
    \else\Ps@rell#1;#2,#3(#4,#5,#6)\f@gstroke\fi%
    \PSc@mment{End arcellDD}\fi\fi}}
\ctr@ld@f\def\Q@arcellTD#1;#2,#3(#4,#5,#6){{\ifCUR@PS\ifGR@cri\s@uvc@ntr@l\et@tpsarcellTD%
    \PSc@mment{arcellTD Center=#1 ; XRad=#2, YRad=#3 (Ang1=#4, Ang2=#5, Inclination=#6)}%
    \setc@ntr@l{2}\figpttraC -8:=#1/#2,0,0/\figpttraC -7:=#1/0,#3,0/%
    \figvectC -4(0,0,1)\figptsrot -8=-8,-7/#1,#6,-4/\Q@arcellPATD#1,-8,-7(#4,#5)%
    \PSc@mment{End arcellTD}\resetc@ntr@l\et@tpsarcellTD\fi\fi}}
\ctr@ln@m\figdrawarcellPA
\ctr@ld@f\def\Q@arcellPADD#1,#2,#3(#4,#5){{\ifCUR@PS\ifGR@cri\s@uvc@ntr@l\et@tpsarcellPADD%
    \PSc@mment{arcellPADD Center=#1,PtAxis1=#2,PtAxis2=#3 (Ang1=#4, Ang2=#5)}%
    \setc@ntr@l{2}\figvectPDD-1[#1,#2]\vecunit@DD{-1}{-1}\v@lX=\ptT@unit@\result@t%
    \edef\XR@d{\repdecn@mb{\v@lX}}\Figg@tXY{-1}\arct@n\v@lmin(\v@lX,\v@lY)%
    \v@lmin=\rdT@deg\v@lmin\edef\Inclin@{\repdecn@mb{\v@lmin}}%
    \figgetdist\YR@d[#1,#3]\Q@arcellDD#1;\XR@d,\YR@d(#4,#5,\Inclin@)%
    \PSc@mment{End arcellPADD}\resetc@ntr@l\et@tpsarcellPADD\fi\fi}}
\ctr@ld@f\def\Q@arcellPATD#1,#2,#3(#4,#5){{\ifCUR@PS\ifGR@cri\s@uvc@ntr@l\et@tpsarcellPATD%
    \PSc@mment{arcellPATD Center=#1,PtAxis1=#2,PtAxis2=#3 (Ang1=#4, Ang2=#5)}%
    \iffillm@de\Ps@rellPATD#1,#2,#3(#4,#5)%
    \f@gfill%
    \else\Ps@rellPATD#1,#2,#3(#4,#5)\f@gstroke\fi%
    \PSc@mment{End arcellPATD}\resetc@ntr@l\et@tpsarcellPATD\fi\fi}}
\ctr@ld@f\def\Ps@rellPATD#1,#2,#3(#4,#5){\let\c@lprojSP=\relax%
    \setc@ntr@l{2}\figvectPTD-1[#1,#2]\figvectPTD-2[#1,#3]\c@lNbarcs{#4}{#5}%
    \v@leur=#4pt\c@lptellP{#1}{-1}{-2}\Figptpr@j-5:/-3/%
    \f@gnewpath\PSwrit@cmdS{-5}{\c@mmoveto}{\fwf@g}{\X@un}{\Y@un}%
    \edef\C@nt@r{#1}\s@mme=\z@\bcl@rellPATD}
\ctr@ld@f\def\bcl@rellPATD{\relax%
    \ifnum\s@mme<\p@rtent\advance\s@mme\@ne%
    \advance\v@leur\delt@\c@lptellP{\C@nt@r}{-1}{-2}\Figptpr@j-4:/-3/%
    \advance\v@leur\delt@\c@lptellP{\C@nt@r}{-1}{-2}\Figptpr@j-6:/-3/%
    \advance\v@leur\delt@\c@lptellP{\C@nt@r}{-1}{-2}\Figptpr@j-3:/-3/%
    \v@lX=\z@\v@lY=\z@\Figtr@nptDD{-5}{-5}\Figtr@nptDD{2}{-3}%
    \divide\v@lX\@vi\divide\v@lY\@vi%
    \Figtr@nptDD{3}{-4}\Figtr@nptDD{-1.5}{-6}\v@lmin=\v@lX\v@lmax=\v@lY%
    \v@lX=\z@\v@lY=\z@\Figtr@nptDD{2}{-5}\Figtr@nptDD{-5}{-3}%
    \divide\v@lX\@vi\divide\v@lY\@vi\Figtr@nptDD{-1.5}{-4}\Figtr@nptDD{3}{-6}%
    \BdingB@xfalse%
    \Figp@intregDD-4:(\v@lmin,\v@lmax)\PSwrit@cmdS{-4}{}{\fwf@g}{\X@de}{\Y@de}%
    \Figp@intregDD-4:(\v@lX,\v@lY)\PSwrit@cmdS{-4}{}{\fwf@g}{\X@tr}{\Y@tr}%
    \BdingB@xtrue\PSwrit@cmdS{-3}{\c@mcurveto}{\fwf@g}{\X@qu}{\Y@qu}%
    \B@zierBB@x{1}{\Y@un}(\X@un,\X@de,\X@tr,\X@qu)%
    \B@zierBB@x{2}{\X@un}(\Y@un,\Y@de,\Y@tr,\Y@qu)%
    \edef\X@un{\X@qu}\edef\Y@un{\Y@qu}\figptcopyDD-5:/-3/\bcl@rellPATD\fi}
\ctr@ld@f\def\c@lNbarcs#1#2{%
    \delt@=#2pt\advance\delt@-#1pt\maxim@m{\v@lmax}{\delt@}{-\delt@}%
    \v@leur=\v@lmax\divide\v@leur45 \p@rtentiere{\p@rtent}{\v@leur}\advance\p@rtent\@ne%
    \s@mme=\p@rtent\multiply\s@mme\thr@@\divide\delt@\s@mme}
\ctr@ld@f\def\figdrawarcellPP#1,#2,#3[#4,#5]{{\ifCUR@PS\ifGR@cri\s@uvc@ntr@l\et@tpsarcellPP%
    \PSc@mment{arcellPP Center=#1,PtAxis1=#2,PtAxis2=#3 [Point1=#4, Point2=#5]}%
    \setc@ntr@l{2}\figvectP-2[#1,#3]\vecunit@{-2}{-2}\v@lmin=\result@t%
    \invers@{\v@lmax}{\v@lmin}%
    \figvectP-1[#1,#2]\vecunit@{-1}{-1}\v@leur=\result@t%
    \v@leur=\repdecn@mb{\v@lmax}\v@leur\edef\AsB@{\repdecn@mb{\v@leur}}% a/b
    \c@lAngle{#1}{#4}{\v@lmin}\edef\@ngdeb{\repdecn@mb{\v@lmin}}%
    \c@lAngle{#1}{#5}{\v@lmax}\ifdim\v@lmin>\v@lmax\advance\v@lmax\DePI@deg\fi%
    \edef\@ngfin{\repdecn@mb{\v@lmax}}\figdrawarcellPA#1,#2,#3(\@ngdeb,\@ngfin)%
    \PSc@mment{End arcellPP}\resetc@ntr@l\et@tpsarcellPP\fi\fi}}
\ctr@ld@f\def\c@lAngle#1#2#3{\figvectP-3[#1,#2]%
    \c@lproscal\delt@[-3,-1]\c@lproscal\v@leur[-3,-2]%
    \v@leur=\AsB@\v@leur\arct@n#3(\delt@,\v@leur)#3=\rdT@deg#3}
\ctr@ln@w{newif}\if@rrowratio\@rrowratiotrue
\ctr@ln@w{newif}\if@rrowhfill
\ctr@ln@w{newif}\if@rrowhout
\ctr@ld@f\def\Psset@rrowhe@d#1=#2|{\keln@mun#1|%
    \def\n@mref{a}\ifx\l@debut\n@mref\update@ttr\D@FTarrowheadangle\Q@s@tarrowheadangle{#2}\else% angle
    \def\n@mref{f}\ifx\l@debut\n@mref\update@ttr\D@FTarrowheadfill\Q@s@tarrowheadfill{#2}\else% fillmode
    \def\n@mref{l}\ifx\l@debut\n@mref\update@ttr\D@FTarrowheadlength\Q@s@tarrowheadlength{#2}\else% length
    \def\n@mref{o}\ifx\l@debut\n@mref\update@ttr\D@FTarrowheadout\Q@s@tarrowheadout{#2}\else% out
    \def\n@mref{r}\ifx\l@debut\n@mref\update@ttr\D@FTarrowheadratio\Q@s@tarrowheadratio{#2}\else% ratio
    \W@rnmesAttr{figset arrowhead}{#1}\fi\fi\fi\fi\fi}
\ctr@ln@m\@rrowheadangle
\ctr@ln@m\C@AHANG \ctr@ln@m\S@AHANG \ctr@ln@m\UNSS@N
\ctr@ld@f\def\Q@s@tarrowheadangle#1{\edef\@rrowheadangle{#1}{\c@ssin{\C@}{\S@}{#1}%
    \xdef\C@AHANG{\C@}\xdef\S@AHANG{\S@}\v@lmax=\S@ pt%
    \invers@{\v@leur}{\v@lmax}\maxim@m{\v@leur}{\v@leur}{-\v@leur}%
    \xdef\UNSS@N{\the\v@leur}}}
\ctr@ld@f\def\Q@s@tarrowheadfill#1{\expandafter\set@rrowhfill#1:}
\ctr@ld@f\def\set@rrowhfill#1#2:{\if#1n\@rrowhfillfalse\else\@rrowhfilltrue\fi}
\ctr@ld@f\def\Q@s@tarrowheadout#1{\expandafter\set@rrowhout#1:}
\ctr@ld@f\def\set@rrowhout#1#2:{\if#1n\@rrowhoutfalse\else\@rrowhouttrue\fi}
\ctr@ln@m\@rrowheadlength
\ctr@ld@f\def\Q@s@tarrowheadlength#1{\edef\@rrowheadlength{#1}\@rrowratiofalse}
\ctr@ln@m\@rrowheadratio
\ctr@ld@f\def\Q@s@tarrowheadratio#1{\edef\@rrowheadratio{#1}\@rrowratiotrue}
\ctr@ln@m\D@FTarrowheadlength
\ctr@ld@f\def\figresetarrowhead{%
    \Q@s@tarrowheadangle{\D@FTarrowheadangle}%
    \Q@s@tarrowheadfill{\D@FTarrowheadfill}%
    \Q@s@tarrowheadout{\D@FTarrowheadout}%
    \Q@s@tarrowheadratio{\D@FTarrowheadratio}%
    \d@fm@cdim\D@FTarrowheadlength{\D@FTh@rdahlength}% Valeur par defaut...
    \Q@s@tarrowheadlength{\D@FTarrowheadlength}}
\ctr@ld@f\def\D@FTarrowheadratio{0.1}
\ctr@ld@f\def\D@FTarrowheadangle{20}
\ctr@ld@f\def\D@FTarrowheadfill{no}
\ctr@ld@f\def\D@FTarrowheadout{no}
\ctr@ld@f\def\D@FTh@rdahlength{8pt}
\ctr@ln@m\figdrawarrow
\ctr@ld@f\def\Q@arrowDD[#1,#2]{{\ifCUR@PS\ifGR@cri\s@uvc@ntr@l\et@tpsarrow%
    \PSc@mment{arrowDD [Pt1,Pt2]=[#1,#2]}\Q@s@tfillmode{no}%
    \Q@arrowheadDD[#1,#2]\setc@ntr@l{2}\figdrawline[#1,-3]%
    \PSc@mment{End arrowDD}\resetc@ntr@l\et@tpsarrow\fi\fi}}
\ctr@ld@f\def\Q@arrowTD[#1,#2]{{\ifCUR@PS\ifGR@cri\s@uvc@ntr@l\et@tpsarrowTD%
    \PSc@mment{arrowTD [Pt1,Pt2]=[#1,#2]}\resetc@ntr@l{2}%
    \Figptpr@j-5:/#1/\Figptpr@j-6:/#2/\let\c@lprojSP=\relax\Q@arrowDD[-5,-6]%
    \PSc@mment{End arrowTD}\resetc@ntr@l\et@tpsarrowTD\fi\fi}}
\ctr@ln@m\figdrawarrowhead
\ctr@ld@f\def\Q@arrowheadDD[#1,#2]{{\ifCUR@PS\ifGR@cri\s@uvc@ntr@l\et@tpsarrowheadDD%
    \if@rrowhfill\def\@hangle{-\@rrowheadangle}\else\def\@hangle{\@rrowheadangle}\fi%
    \if@rrowratio%
    \if@rrowhout\def\@hratio{-\@rrowheadratio}\else\def\@hratio{\@rrowheadratio}\fi%
    \PSc@mment{arrowheadDD Ratio=\@hratio, Angle=\@hangle, [Pt1,Pt2]=[#1,#2]}%
    \Ps@rrowhead\@hratio,\@hangle[#1,#2]%
    \else%
    \if@rrowhout\def\@hlength{-\@rrowheadlength}\else\def\@hlength{\@rrowheadlength}\fi%
    \PSc@mment{arrowheadDD Length=\@hlength, Angle=\@hangle, [Pt1,Pt2]=[#1,#2]}%
    \Ps@rrowheadfd\@hlength,\@hangle[#1,#2]%
    \fi%
    \PSc@mment{End arrowheadDD}\resetc@ntr@l\et@tpsarrowheadDD\fi\fi}}
\ctr@ld@f\def\Q@arrowheadTD[#1,#2]{{\ifCUR@PS\ifGR@cri\s@uvc@ntr@l\et@tpsarrowheadTD%
    \PSc@mment{arrowheadTD [Pt1,Pt2]=[#1,#2]}\resetc@ntr@l{2}%
    \Figptpr@j-5:/#1/\Figptpr@j-6:/#2/\let\c@lprojSP=\relax\Q@arrowheadDD[-5,-6]%
    \PSc@mment{End arrowheadTD}\resetc@ntr@l\et@tpsarrowheadTD\fi\fi}}
\ctr@ld@f\def\Ps@rrowhead#1,#2[#3,#4]{\v@leur=#1\p@\maxim@m{\v@leur}{\v@leur}{-\v@leur}%
    \ifdim\v@leur>\Cepsil@n{% Arrow is not degenerated
    \PSc@mment{@rrowhead Ratio=#1, Angle=#2, [Pt1,Pt2]=[#3,#4]}\v@leur=\UNSS@N%
    \v@leur=\CUR@width\v@leur\v@leur=\ptpsT@pt\v@leur\delt@=.5\v@leur% = width / (2 sin(Angle))
    \setc@ntr@l{2}\figvectPDD-3[#4,#3]%
    \Figg@tXY{-3}\v@lX=#1\v@lX\v@lY=#1\v@lY\Figv@ctCreg-3(\v@lX,\v@lY)%
    \vecunit@{-4}{-3}\mili@u=\result@t%
    \ifdim#2pt>\z@\v@lXa=-\C@AHANG\delt@%
     \edef\c@ef{\repdecn@mb{\v@lXa}}\figpttraDD-3:=-3/\c@ef,-4/\fi%
    \edef\c@ef{\repdecn@mb{\delt@}}%
    \v@lXa=\mili@u\v@lXa=\C@AHANG\v@lXa%
    \v@lYa=\ptpsT@pt\p@\v@lYa=\CUR@width\v@lYa\v@lYa=\sDcc@ngle\v@lYa%
    \advance\v@lXa-\v@lYa\gdef\sDcc@ngle{0}%
    \ifdim\v@lXa>\v@leur\edef\c@efendpt{\repdecn@mb{\v@leur}}%
    \else\edef\c@efendpt{\repdecn@mb{\v@lXa}}\fi%
    \Figg@tXY{-3}\v@lmin=\v@lX\v@lmax=\v@lY%
    \v@lXa=\C@AHANG\v@lmin\v@lYa=\S@AHANG\v@lmax\advance\v@lXa\v@lYa%
    \v@lYa=-\S@AHANG\v@lmin\v@lX=\C@AHANG\v@lmax\advance\v@lYa\v@lX%
    \setc@ntr@l{1}\Figg@tXY{#4}\advance\v@lX\v@lXa\advance\v@lY\v@lYa%
    \setc@ntr@l{2}\Figp@intregDD-2:(\v@lX,\v@lY)%
    \v@lXa=\C@AHANG\v@lmin\v@lYa=-\S@AHANG\v@lmax\advance\v@lXa\v@lYa%
    \v@lYa=\S@AHANG\v@lmin\v@lX=\C@AHANG\v@lmax\advance\v@lYa\v@lX%
    \setc@ntr@l{1}\Figg@tXY{#4}\advance\v@lX\v@lXa\advance\v@lY\v@lYa%
    \setc@ntr@l{2}\Figp@intregDD-1:(\v@lX,\v@lY)%
    \ifdim#2pt<\z@\fillm@detrue\figdrawline[-2,#4,-1]% fill
    \else\figptstraDD-3=#4,-2,-1/\c@ef,-4/\s@uvdash{\typ@dash}\Q@s@tdash{\D@FTdash}%
    \figdrawline[-2,-3,-1]\Q@s@tdash{\typ@dash}\fi% no fill
    \ifdim#1pt>\z@\figpttraDD-3:=#4/\c@efendpt,-4/\else\figptcopyDD-3:/#4/\fi%
    \PSc@mment{End @rrowhead}}\fi}
\ctr@ld@f\def\sDcc@ngle{0}% Initialisation
\ctr@ld@f\def\Ps@rrowheadfd#1,#2[#3,#4]{{%
    \PSc@mment{@rrowheadfd Length=#1, Angle=#2, [Pt1,Pt2]=[#3,#4]}%
    \setc@ntr@l{2}\figvectPDD-1[#3,#4]\n@rmeucDD{\v@leur}{-1}\v@leur=\ptT@unit@\v@leur%
    \invers@{\v@leur}{\v@leur}\v@leur=#1\v@leur\edef\R@tio{\repdecn@mb{\v@leur}}%
    \Ps@rrowhead\R@tio,#2[#3,#4]\PSc@mment{End @rrowheadfd}}}
\ctr@ln@m\figdrawarrowBezier
\ctr@ld@f\def\Q@arrowBezierDD[#1,#2,#3,#4]{{\ifCUR@PS\ifGR@cri\s@uvc@ntr@l\et@tpsarrowBezierDD%
    \PSc@mment{arrowBezierDD Control points=#1,#2,#3,#4}\setc@ntr@l{2}%
    \if@rrowratio\c@larclengthDD\v@leur,10[#1,#2,#3,#4]\else\v@leur=\z@\fi%
    \Ps@rrowB@zDD\v@leur[#1,#2,#3,#4]%
    \PSc@mment{End arrowBezierDD}\resetc@ntr@l\et@tpsarrowBezierDD\fi\fi}}
\ctr@ld@f\def\Q@arrowBezierTD[#1,#2,#3,#4]{{\ifCUR@PS\ifGR@cri\s@uvc@ntr@l\et@tpsarrowBezierTD%
    \PSc@mment{arrowBezierTD Control points=#1,#2,#3,#4}\resetc@ntr@l{2}%
    \Figptpr@j-7:/#1/\Figptpr@j-8:/#2/\Figptpr@j-9:/#3/\Figptpr@j-10:/#4/%
    \let\c@lprojSP=\relax\ifnum\CUR@proj<\tw@\Q@arrowBezierDD[-7,-8,-9,-10]%
    \else\f@gnewpath\PSwrit@cmd{-7}{\c@mmoveto}{\fwf@g}%
    \if@rrowratio\c@larclengthDD\mili@u,10[-7,-8,-9,-10]\else\mili@u=\z@\fi%
    \p@rtent=\NBz@rcs\advance\p@rtent\m@ne\subB@zierTD\p@rtent[#1,#2,#3,#4]%
    \f@gstroke%
    \advance\v@lmin\p@rtent\delt@% Initialized in \subB@zierTD
    \v@leur=\v@lmin\advance\v@leur0.33333 \delt@\edef\unti@rs{\repdecn@mb{\v@leur}}%
    \v@leur=\v@lmin\advance\v@leur0.66666 \delt@\edef\deti@rs{\repdecn@mb{\v@leur}}%
    \figptcopyDD-8:/-10/\c@lsubBzarc\unti@rs,\deti@rs[#1,#2,#3,#4]%
    \figptcopyDD-8:/-4/\figptcopyDD-9:/-3/\Ps@rrowB@zDD\mili@u[-7,-8,-9,-10]\fi%
    \PSc@mment{End arrowBezierTD}\resetc@ntr@l\et@tpsarrowBezierTD\fi\fi}}
\ctr@ld@f\def\c@larclengthDD#1,#2[#3,#4,#5,#6]{{\p@rtent=#2\figptcopyDD-5:/#3/%
    \delt@=\p@\divide\delt@\p@rtent\c@rre=\z@\v@leur=\z@\s@mme=\z@%
    \loop\ifnum\s@mme<\p@rtent\advance\s@mme\@ne\advance\v@leur\delt@%
    \edef\T@{\repdecn@mb{\v@leur}}\figptBezierDD-6::\T@[#3,#4,#5,#6]%
    \figvectPDD-1[-5,-6]\n@rmeucDD{\mili@u}{-1}\advance\c@rre\mili@u%
    \figptcopyDD-5:/-6/\repeat\global\result@t=\ptT@unit@\c@rre}#1=\result@t}
\ctr@ld@f\def\Ps@rrowB@zDD#1[#2,#3,#4,#5]{{\Q@s@tfillmode{no}%
    \if@rrowratio\delt@=\@rrowheadratio#1\else\delt@=\@rrowheadlength pt\fi%
    \v@leur=\C@AHANG\delt@\edef\R@dius{\repdecn@mb{\v@leur}}%
    \FigptintercircB@zDD-5::0,\R@dius[#5,#4,#3,#2]%
    \Q@s@tarrowheadlength{\repdecn@mb{\delt@}}\Q@arrowheadDD[-5,#5]%
    \let\n@rmeuc=\n@rmeucDD\figgetdist\R@dius[#5,-3]%
    \FigptintercircB@zDD-6::0,\R@dius[#5,#4,#3,#2]%
    \figptBezierDD-5::0.33333[#5,#4,#3,#2]\figptBezierDD-3::0.66666[#5,#4,#3,#2]%
    \figptscontrolDD-5[-6,-5,-3,#2]\Q@BezierDD1[-6,-5,-4,#2]}}
\ctr@ln@m\figdrawarrowcirc
\ctr@ld@f\def\Q@arrowcircDD#1;#2(#3,#4){{\ifCUR@PS\ifGR@cri\s@uvc@ntr@l\et@tpsarrowcircDD%
    \PSc@mment{arrowcircDD Center=#1 ; Radius=#2 (Ang1=#3,Ang2=#4)}%
    \Q@s@tfillmode{no}\Pscirc@rrowhead#1;#2(#3,#4)%
    \setc@ntr@l{2}\figvectPDD -4[#1,-3]\vecunit@{-4}{-4}%
    \Figg@tXY{-4}\arct@n\v@lmin(\v@lX,\v@lY)%
    \v@lmin=\rdT@deg\v@lmin\v@leur=#4pt\advance\v@leur-\v@lmin%
    \maxim@m{\v@leur}{\v@leur}{-\v@leur}%
    \ifdim\v@leur>\DemiPI@deg\relax\ifdim\v@lmin<#4pt\advance\v@lmin\DePI@deg%
    \else\advance\v@lmin-\DePI@deg\fi\fi\edef\ar@ngle{\repdecn@mb{\v@lmin}}%
    \ifdim#3pt<#4pt\figdrawarccirc#1;#2(#3,\ar@ngle)\else\figdrawarccirc#1;#2(\ar@ngle,#3)\fi%
    \PSc@mment{End arrowcircDD}\resetc@ntr@l\et@tpsarrowcircDD\fi\fi}}
\ctr@ld@f\def\Q@arrowcircTD#1,#2,#3;#4(#5,#6){{\ifCUR@PS\ifGR@cri\s@uvc@ntr@l\et@tpsarrowcircTD%
    \PSc@mment{arrowcircTD Center=#1,P1=#2,P2=#3 ; Radius=#4 (Ang1=#5, Ang2=#6)}%
    \resetc@ntr@l{2}\c@lExtAxes#1,#2,#3(#4)\let\c@lprojSP=\relax%
    \figvectPTD-11[#1,-4]\figvectPTD-12[#1,-5]\c@lNbarcs{#5}{#6}%
    \if@rrowratio\v@lmax=\degT@rd\v@lmax\edef\D@lpha{\repdecn@mb{\v@lmax}}\fi%
    \advance\p@rtent\m@ne\mili@u=\z@%
    \v@leur=#5pt\c@lptellP{#1}{-11}{-12}\Figptpr@j-9:/-3/%
    \f@gnewpath\PSwrit@cmdS{-9}{\c@mmoveto}{\fwf@g}{\X@un}{\Y@un}%
    \edef\C@nt@r{#1}\s@mme=\z@\bcl@rcircTD\f@gstroke%
    \advance\v@leur\delt@\c@lptellP{#1}{-11}{-12}\Figptpr@j-5:/-3/%
    \advance\v@leur\delt@\c@lptellP{#1}{-11}{-12}\Figptpr@j-6:/-3/%
    \advance\v@leur\delt@\c@lptellP{#1}{-11}{-12}\Figptpr@j-10:/-3/%
    \figptscontrolDD-8[-9,-5,-6,-10]%
    \if@rrowratio\c@lcurvradDD0.5[-9,-8,-7,-10]\advance\mili@u\result@t%
    \maxim@m{\mili@u}{\mili@u}{-\mili@u}\mili@u=\ptT@unit@\mili@u%
    \mili@u=\D@lpha\mili@u\advance\p@rtent\@ne\divide\mili@u\p@rtent\fi%
    \Ps@rrowB@zDD\mili@u[-9,-8,-7,-10]%
    \PSc@mment{End arrowcircTD}\resetc@ntr@l\et@tpsarrowcircTD\fi\fi}}
\ctr@ld@f\def\bcl@rcircTD{\relax%
    \ifnum\s@mme<\p@rtent\advance\s@mme\@ne%
    \advance\v@leur\delt@\c@lptellP{\C@nt@r}{-11}{-12}\Figptpr@j-5:/-3/%
    \advance\v@leur\delt@\c@lptellP{\C@nt@r}{-11}{-12}\Figptpr@j-6:/-3/%
    \advance\v@leur\delt@\c@lptellP{\C@nt@r}{-11}{-12}\Figptpr@j-10:/-3/%
    \figptscontrolDD-8[-9,-5,-6,-10]\BdingB@xfalse%
    \PSwrit@cmdS{-8}{}{\fwf@g}{\X@de}{\Y@de}\PSwrit@cmdS{-7}{}{\fwf@g}{\X@tr}{\Y@tr}%
    \BdingB@xtrue\PSwrit@cmdS{-10}{\c@mcurveto}{\fwf@g}{\X@qu}{\Y@qu}%
    \if@rrowratio\c@lcurvradDD0.5[-9,-8,-7,-10]\advance\mili@u\result@t\fi%
    \B@zierBB@x{1}{\Y@un}(\X@un,\X@de,\X@tr,\X@qu)%
    \B@zierBB@x{2}{\X@un}(\Y@un,\Y@de,\Y@tr,\Y@qu)%
    \edef\X@un{\X@qu}\edef\Y@un{\Y@qu}\figptcopyDD-9:/-10/\bcl@rcircTD\fi}
\ctr@ld@f\def\Pscirc@rrowhead#1;#2(#3,#4){{%
    \PSc@mment{circ@rrowhead Center=#1 ; Radius=#2 (Ang1=#3,Ang2=#4)}%
    \v@leur=#2\unit@\edef\s@glen{\repdecn@mb{\v@leur}}\v@lY=\z@\v@lX=\v@leur%
    \resetc@ntr@l{2}\Figv@ctCreg-3(\v@lX,\v@lY)\figpttraDD-5:=#1/1,-3/%
    \figptrotDD-5:=-5/#1,#4/%
    \figvectPDD-3[#1,-5]\Figg@tXY{-3}\v@leur=\v@lX%
    \ifdim#3pt<#4pt\v@lX=\v@lY\v@lY=-\v@leur\else\v@lX=-\v@lY\v@lY=\v@leur\fi%
    \Figv@ctCreg-3(\v@lX,\v@lY)\vecunit@{-3}{-3}%
    \if@rrowratio\v@leur=#4pt\advance\v@leur-#3pt\maxim@m{\mili@u}{-\v@leur}{\v@leur}%
    \mili@u=\degT@rd\mili@u\v@leur=\s@glen\mili@u\edef\s@glen{\repdecn@mb{\v@leur}}%
    \mili@u=#2\mili@u\mili@u=\@rrowheadratio\mili@u\else\mili@u=\@rrowheadlength pt\fi%
    \figpttraDD-6:=-5/\s@glen,-3/\v@leur=#2pt\v@leur=2\v@leur%
    \invers@{\v@leur}{\v@leur}\c@rre=\repdecn@mb{\v@leur}\mili@u% = sin = L/(2R)
    \mili@u=\c@rre\mili@u=\repdecn@mb{\c@rre}\mili@u%
    \v@leur=\p@\advance\v@leur-\mili@u% \v@leur = cos*cos
    \invers@{\mili@u}{2\v@leur}\delt@=\c@rre\delt@=\repdecn@mb{\mili@u}\delt@%
    \xdef\sDcc@ngle{\repdecn@mb{\delt@}}% sin/(2*cos*cos) used in \Ps@rrowhead
    \sqrt@{\mili@u}{\v@leur}\arct@n\v@leur(\mili@u,\c@rre)%
    \v@leur=\rdT@deg\v@leur% \cor@ngle = atan(L/sqrt(4R*R-L*L))
    \ifdim#3pt<#4pt\v@leur=-\v@leur\fi%
    \if@rrowhout\v@leur=-\v@leur\fi\edef\cor@ngle{\repdecn@mb{\v@leur}}%
    \figptrotDD-6:=-6/-5,\cor@ngle/\Q@arrowheadDD[-6,-5]%
    \PSc@mment{End circ@rrowhead}}}
\ctr@ln@m\figdrawarrowcircP
\ctr@ld@f\def\Q@arrowcircPDD#1;#2[#3,#4]{{\ifCUR@PS\ifGR@cri%
    \PSc@mment{arrowcircPDD Center=#1; Radius=#2, [P1=#3,P2=#4]}%
    \s@uvc@ntr@l\et@tpsarrowcircPDD\Ps@ngleparam#1;#2[#3,#4]%
    \ifdim\v@leur>\z@\ifdim\v@lmin>\v@lmax\advance\v@lmax\DePI@deg\fi%
    \else\ifdim\v@lmin<\v@lmax\advance\v@lmin\DePI@deg\fi\fi%
    \edef\@ngdeb{\repdecn@mb{\v@lmin}}\edef\@ngfin{\repdecn@mb{\v@lmax}}%
    \figdrawarrowcirc#1;\r@dius(\@ngdeb,\@ngfin)%
    \PSc@mment{End arrowcircPDD}\resetc@ntr@l\et@tpsarrowcircPDD\fi\fi}}
\ctr@ld@f\def\Q@arrowcircPTD#1;#2[#3,#4,#5]{{\ifCUR@PS\ifGR@cri\s@uvc@ntr@l\et@tpsarrowcircPTD%
    \PSc@mment{arrowcircPTD Center=#1; Radius=#2, [P1=#3,P2=#4,P3=#5]}%
    \figgetangleTD\@ngfin[#1,#3,#4,#5]\v@leur=#2pt%
    \maxim@m{\mili@u}{-\v@leur}{\v@leur}\edef\r@dius{\repdecn@mb{\mili@u}}%
    \ifdim\v@leur<\z@\v@lmax=\@ngfin pt\advance\v@lmax-\DePI@deg%
    \edef\@ngfin{\repdecn@mb{\v@lmax}}\fi\Q@arrowcircTD#1,#3,#5;\r@dius(0,\@ngfin)%
    \PSc@mment{End arrowcircPTD}\resetc@ntr@l\et@tpsarrowcircPTD\fi\fi}}
\ctr@ld@f\def\figdrawaxes#1(#2){{\ifCUR@PS\ifGR@cri\s@uvc@ntr@l\et@tpsaxes%
    \PSc@mment{axes Origin=#1 Range=(#2)}\an@lys@xes#2,:\resetc@ntr@l{2}%
    \ifx\t@xt@\empty\ifTr@isDim\Q@@xes#1(0,#2,0,#2,0,#2)\else\Q@@xes#1(0,#2,0,#2)\fi%
    \else\Q@@xes#1(#2)\fi\PSc@mment{End axes}\resetc@ntr@l\et@tpsaxes\fi\fi}}
\ctr@ld@f\def\an@lys@xes#1,#2:{\def\t@xt@{#2}}
\ctr@ln@m\Q@@xes
\ctr@ld@f\def\Q@@xesDD#1(#2,#3,#4,#5){%
    \figpttraC-5:=#1/#2,0/\figpttraC-6:=#1/#3,0/\Q@arrowDD[-5,-6]%
    \figpttraC-5:=#1/0,#4/\figpttraC-6:=#1/0,#5/\Q@arrowDD[-5,-6]}
\ctr@ld@f\def\Q@@xesTD#1(#2,#3,#4,#5,#6,#7){%
    \figpttraC-7:=#1/#2,0,0/\figpttraC-8:=#1/#3,0,0/\Q@arrowTD[-7,-8]%
    \figpttraC-7:=#1/0,#4,0/\figpttraC-8:=#1/0,#5,0/\Q@arrowTD[-7,-8]%
    \figpttraC-7:=#1/0,0,#6/\figpttraC-8:=#1/0,0,#7/\Q@arrowTD[-7,-8]}
\ctr@ln@m\newGr@FN
\ctr@ld@f\def\newGr@FNPDF#1{\s@mme=\Gr@FNb\advance\s@mme\@ne\xdef\Gr@FNb{\number\s@mme}}
\ctr@ld@f\def\newGr@FNDVI#1{\newGr@FNPDF{}\xdef#1{\jobname GI\Gr@FNb.anx}}
\ctr@ld@f\def\figdrawbegin#1{\newGr@FN\DefGIfilen@me\gdef\@utoFN{0}%
    \def\t@xt@{#1}\relax\ifx\t@xt@\empty\GRupdatem@detrue%
    \gdef\@utoFN{1}\Psb@ginfig\DefGIfilen@me\else\expandafter\Psb@ginfigNu@#1 :\fi}
\ctr@ld@f\def\Psb@ginfigNu@#1 #2:{\def\t@xt@{#1}\relax\ifx\t@xt@\empty\def\t@xt@{#2}%
    \ifx\t@xt@\empty\GRupdatem@detrue\gdef\@utoFN{1}\Psb@ginfig\DefGIfilen@me%
    \else\Psb@ginfigNu@#2:\fi\else\Psb@ginfig{#1}\fi}
\ctr@ln@m\PSfilen@me \ctr@ln@m\auxfilen@me
\ctr@ld@f\def\Psb@ginfig#1{\ifCUR@PS\else%
    \edef\PSfilen@me{#1}\edef\auxfilen@me{\jobname.anx}%
    \ifGRupdatem@de\GR@critrue\else\openin\frf@g=\PSfilen@me\relax%
    \ifeof\frf@g\GR@critrue\else\GR@crifalse\fi\closein\frf@g\fi%
    \CUR@PStrue\c@ldefproj\expandafter\setupd@te\D@FTupdate:%
    \ifGR@cri\initb@undb@x%
    \immediate\openout\fwf@g=\auxfilen@me\initpss@ttings\fi%
    \fi}
\ctr@ld@f\def\Gr@FNb{0}
\ctr@ld@f\def\figforTeXFileno{\Gr@FNb}
\ctr@ld@f\def\figforTeXFigno{0 }
\ctr@ld@f\def\figforTeXnextFigno{1 }
\ctr@ld@f\edef\DefGIfilen@me{\jobname GI.anx}
\ctr@ld@f\def\initpss@ttings{\figreset{altitude,arrowhead,curve,general,flowchart,mesh,trimesh}%
    \Use@llipsefalse}
\ctr@ld@f\def\B@zierBB@x#1#2(#3,#4,#5,#6){{\c@rre=\t@n\epsil@n% Do not reduce this value
    \v@lmax=#4\advance\v@lmax-#5\v@lmax=\thr@@\v@lmax\advance\v@lmax#6\advance\v@lmax-#3%
    \mili@u=#4\mili@u=-\tw@\mili@u\advance\mili@u#3\advance\mili@u#5%
    \v@lmin=#4\advance\v@lmin-#3\maxim@m{\v@leur}{-\v@lmax}{\v@lmax}%
    \maxim@m{\delt@}{-\mili@u}{\mili@u}\maxim@m{\v@leur}{\v@leur}{\delt@}%
    \maxim@m{\delt@}{-\v@lmin}{\v@lmin}\maxim@m{\v@leur}{\v@leur}{\delt@}%
    \ifdim\v@leur>\c@rre\invers@{\v@leur}{\v@leur}\edef\Uns@rM@x{\repdecn@mb{\v@leur}}%
    \v@lmax=\Uns@rM@x\v@lmax\mili@u=\Uns@rM@x\mili@u\v@lmin=\Uns@rM@x\v@lmin%
    \maxim@m{\v@leur}{-\v@lmax}{\v@lmax}\ifdim\v@leur<\c@rre%
    \maxim@m{\v@leur}{-\mili@u}{\mili@u}\ifdim\v@leur<\c@rre\else%
    \invers@{\mili@u}{\mili@u}\v@leur=-0.5\v@lmin%
    \v@leur=\repdecn@mb{\mili@u}\v@leur\m@jBBB@x{\v@leur}{#1}{#2}(#3,#4,#5,#6)\fi%
    \else\delt@=\repdecn@mb{\mili@u}\mili@u\v@leur=\repdecn@mb{\v@lmax}\v@lmin%
    \advance\delt@-\v@leur\ifdim\delt@<\z@\else\invers@{\v@lmax}{\v@lmax}%
    \edef\Uns@rAp{\repdecn@mb{\v@lmax}}\sqrt@{\delt@}{\delt@}%
    \v@leur=-\mili@u\advance\v@leur\delt@\v@leur=\Uns@rAp\v@leur%
    \m@jBBB@x{\v@leur}{#1}{#2}(#3,#4,#5,#6)%
    \v@leur=-\mili@u\advance\v@leur-\delt@\v@leur=\Uns@rAp\v@leur%
    \m@jBBB@x{\v@leur}{#1}{#2}(#3,#4,#5,#6)\fi\fi\fi}}
\ctr@ld@f\def\m@jBBB@x#1#2#3(#4,#5,#6,#7){{\relax\ifdim#1>\z@\ifdim#1<\p@%
    \edef\T@{\repdecn@mb{#1}}\v@lX=\p@\advance\v@lX-#1\edef\UNmT@{\repdecn@mb{\v@lX}}%
    \v@lX=#4\v@lY=#5\v@lZ=#6\v@lXa=#7\v@lX=\UNmT@\v@lX\advance\v@lX\T@\v@lY%
    \v@lY=\UNmT@\v@lY\advance\v@lY\T@\v@lZ\v@lZ=\UNmT@\v@lZ\advance\v@lZ\T@\v@lXa%
    \v@lX=\UNmT@\v@lX\advance\v@lX\T@\v@lY\v@lY=\UNmT@\v@lY\advance\v@lY\T@\v@lZ%
    \v@lX=\UNmT@\v@lX\advance\v@lX\T@\v@lY%
    \ifcase#2\or\v@lY=#3\or\v@lY=\v@lX\v@lX=#3\fi\b@undb@x{\v@lX}{\v@lY}\fi\fi}}
\ctr@ld@f\def\PsB@zier#1[#2]{{\f@gnewpath%
    \s@mme=\z@\def\list@num{#2,0}\extrairelepremi@r\p@int\de\list@num%
    \PSwrit@cmdS{\p@int}{\c@mmoveto}{\fwf@g}{\X@un}{\Y@un}\p@rtent=#1\bclB@zier}}
\ctr@ld@f\def\bclB@zier{\relax%
    \ifnum\s@mme<\p@rtent\advance\s@mme\@ne\BdingB@xfalse%
    \extrairelepremi@r\p@int\de\list@num\PSwrit@cmdS{\p@int}{}{\fwf@g}{\X@de}{\Y@de}%
    \extrairelepremi@r\p@int\de\list@num\PSwrit@cmdS{\p@int}{}{\fwf@g}{\X@tr}{\Y@tr}%
    \BdingB@xtrue%
    \extrairelepremi@r\p@int\de\list@num\PSwrit@cmdS{\p@int}{\c@mcurveto}{\fwf@g}{\X@qu}{\Y@qu}%
    \B@zierBB@x{1}{\Y@un}(\X@un,\X@de,\X@tr,\X@qu)%
    \B@zierBB@x{2}{\X@un}(\Y@un,\Y@de,\Y@tr,\Y@qu)%
    \edef\X@un{\X@qu}\edef\Y@un{\Y@qu}\bclB@zier\fi}
\ctr@ln@m\figdrawBezier
\ctr@ld@f\def\Q@BezierDD#1[#2]{\ifCUR@PS\ifGR@cri%
    \PSc@mment{BezierDD N arcs=#1, Control points=#2}%
    \iffillm@de\PsB@zier#1[#2]%
    \f@gfill%
    \else\PsB@zier#1[#2]\f@gstroke\fi%
    \PSc@mment{End BezierDD}\fi\fi}
\ctr@ln@m\et@tpsBezierTD% ou doubler les {}
\ctr@ld@f\def\Q@BezierTD#1[#2]{\ifCUR@PS\ifGR@cri\s@uvc@ntr@l\et@tpsBezierTD%
    \PSc@mment{BezierTD N arcs=#1, Control points=#2}%
    \iffillm@de\PsB@zierTD#1[#2]%
    \f@gfill%
    \else\PsB@zierTD#1[#2]\f@gstroke\fi%
    \PSc@mment{End BezierTD}\resetc@ntr@l\et@tpsBezierTD\fi\fi}
\ctr@ld@f\def\PsB@zierTD#1[#2]{\ifnum\CUR@proj<\tw@\PsB@zier#1[#2]\else\PsB@zier@TD#1[#2]\fi}
\ctr@ld@f\def\PsB@zier@TD#1[#2]{{\f@gnewpath%
    \s@mme=\z@\def\list@num{#2,0}\extrairelepremi@r\p@int\de\list@num%
    \let\c@lprojSP=\relax\setc@ntr@l{2}\Figptpr@j-7:/\p@int/%
    \PSwrit@cmd{-7}{\c@mmoveto}{\fwf@g}%
    \loop\ifnum\s@mme<#1\advance\s@mme\@ne\extrairelepremi@r\p@intun\de\list@num%
    \extrairelepremi@r\p@intde\de\list@num\extrairelepremi@r\p@inttr\de\list@num%
    \subB@zierTD\NBz@rcs[\p@int,\p@intun,\p@intde,\p@inttr]\edef\p@int{\p@inttr}\repeat}}
\ctr@ld@f\def\subB@zierTD#1[#2,#3,#4,#5]{\delt@=\p@\divide\delt@\NBz@rcs\v@lmin=\z@%
    {\Figg@tXY{-7}\edef\X@un{\the\v@lX}\edef\Y@un{\the\v@lY}%
    \s@mme=\z@\loop\ifnum\s@mme<#1\advance\s@mme\@ne%
    \v@leur=\v@lmin\advance\v@leur0.33333 \delt@\edef\unti@rs{\repdecn@mb{\v@leur}}%
    \v@leur=\v@lmin\advance\v@leur0.66666 \delt@\edef\deti@rs{\repdecn@mb{\v@leur}}%
    \advance\v@lmin\delt@\edef\trti@rs{\repdecn@mb{\v@lmin}}%
    \figptBezierTD-8::\trti@rs[#2,#3,#4,#5]\Figptpr@j-8:/-8/%
    \c@lsubBzarc\unti@rs,\deti@rs[#2,#3,#4,#5]\BdingB@xfalse%
    \PSwrit@cmdS{-4}{}{\fwf@g}{\X@de}{\Y@de}\PSwrit@cmdS{-3}{}{\fwf@g}{\X@tr}{\Y@tr}%
    \BdingB@xtrue\PSwrit@cmdS{-8}{\c@mcurveto}{\fwf@g}{\X@qu}{\Y@qu}%
    \B@zierBB@x{1}{\Y@un}(\X@un,\X@de,\X@tr,\X@qu)%
    \B@zierBB@x{2}{\X@un}(\Y@un,\Y@de,\Y@tr,\Y@qu)%
    \edef\X@un{\X@qu}\edef\Y@un{\Y@qu}\figptcopyDD-7:/-8/\repeat}}
\ctr@ld@f\def\NBz@rcs{2}
\ctr@ld@f\def\c@lsubBzarc#1,#2[#3,#4,#5,#6]{\figptBezierTD-5::#1[#3,#4,#5,#6]%
    \figptBezierTD-6::#2[#3,#4,#5,#6]\Figptpr@j-4:/-5/\Figptpr@j-5:/-6/%
    \figptscontrolDD-4[-7,-4,-5,-8]}
\ctr@ln@m\figdrawcirc
\ctr@ld@f\def\Q@circDD#1(#2){\ifCUR@PS\ifGR@cri\PSc@mment{circDD Center=#1 (Radius=#2)}%
    \Q@arccircDD#1;#2(0,360)\PSc@mment{End circDD}\fi\fi}
\ctr@ld@f\def\Q@circTD#1,#2,#3(#4){\ifCUR@PS\ifGR@cri%
    \PSc@mment{circTD Center=#1,P1=#2,P2=#3 (Radius=#4)}%
    \Q@arccircTD#1,#2,#3;#4(0,360)\PSc@mment{End circTD}\fi\fi}
\ctr@ln@m\p@urcent
{\catcode`\%=12\gdef\p@urcent{%}}
\ctr@ld@f\def\PSc@mment#1{\ifGRdebugm@de\immediate\write\fwf@g{\p@urcent\space#1}\fi}
\ctr@ln@m\acc@louv \ctr@ln@m\acc@lfer
{\catcode`\[=1\catcode`\{=12\gdef\acc@louv[{}}
{\catcode`\]=2\catcode`\}=12\gdef\acc@lfer{}]]
\ctr@ld@f\def\PSdict@{\ifUse@llipse%
    \immediate\write\fwf@g{/ellipsedict 9 dict def ellipsedict /mtrx matrix put}%
    \immediate\write\fwf@g{/ellipse \acc@louv ellipsedict begin}%
    \immediate\write\fwf@g{ /endangle exch def /startangle exch def}%
    \immediate\write\fwf@g{ /yrad exch def /xrad exch def}%
    \immediate\write\fwf@g{ /rotangle exch def /y exch def /x exch def}%
    \immediate\write\fwf@g{ /savematrix mtrx currentmatrix def}%
    \immediate\write\fwf@g{ x y translate rotangle rotate xrad yrad scale}%
    \immediate\write\fwf@g{ 0 0 1 startangle endangle arc}%
    \immediate\write\fwf@g{ savematrix setmatrix end\acc@lfer def}%
    \fi\PShe@der{EndProlog}}
\ctr@ld@f\def\Pssetc@rve#1=#2|{\keln@mun#1|%
    \def\n@mref{r}\ifx\l@debut\n@mref\update@ttr\D@FTroundness\Q@s@troundness{#2}\else% roundness
    \W@rnmesAttr{figset curve}{#1}\fi}
\ctr@ln@m\curv@roundness
\ctr@ld@f\def\Q@s@troundness#1{\edef\curv@roundness{#1}}
\ctr@ld@f\def\D@FTroundness{0.2} % Valeur par defaut
\ctr@ln@m\figdrawcurve
\ctr@ld@f\def\Q@curveDD[#1]{{\ifCUR@PS\ifGR@cri\PSc@mment{curveDD Points=#1}%
    \s@uvc@ntr@l\et@tpscurveDD%
    \iffillm@de\Psc@rveDD\curv@roundness[#1]%
    \f@gfill%
    \else\Psc@rveDD\curv@roundness[#1]\f@gstroke\fi%
    \PSc@mment{End curveDD}\resetc@ntr@l\et@tpscurveDD\fi\fi}}
\ctr@ld@f\def\Q@curveTD[#1]{{\ifCUR@PS\ifGR@cri%
    \PSc@mment{curveTD Points=#1}\s@uvc@ntr@l\et@tpscurveTD\let\c@lprojSP=\relax%
    \iffillm@de\Psc@rveTD\curv@roundness[#1]%
    \f@gfill%
    \else\Psc@rveTD\curv@roundness[#1]\f@gstroke\fi%
    \PSc@mment{End curveTD}\resetc@ntr@l\et@tpscurveTD\fi\fi}}
\ctr@ld@f\def\Psc@rveDD#1[#2]{%
    \def\list@num{#2}\extrairelepremi@r\Ak@\de\list@num%
    \extrairelepremi@r\Ai@\de\list@num\extrairelepremi@r\Aj@\de\list@num%
    \f@gnewpath\PSwrit@cmdS{\Ai@}{\c@mmoveto}{\fwf@g}{\X@un}{\Y@un}%
    \setc@ntr@l{2}\figvectPDD -1[\Ak@,\Aj@]%
    \@ecfor\Ak@:=\list@num\do{\figpttraDD-2:=\Ai@/#1,-1/\BdingB@xfalse%
       \PSwrit@cmdS{-2}{}{\fwf@g}{\X@de}{\Y@de}%
       \figvectPDD -1[\Ai@,\Ak@]\figpttraDD-2:=\Aj@/-#1,-1/%
       \PSwrit@cmdS{-2}{}{\fwf@g}{\X@tr}{\Y@tr}\BdingB@xtrue%
       \PSwrit@cmdS{\Aj@}{\c@mcurveto}{\fwf@g}{\X@qu}{\Y@qu}%
       \B@zierBB@x{1}{\Y@un}(\X@un,\X@de,\X@tr,\X@qu)%
       \B@zierBB@x{2}{\X@un}(\Y@un,\Y@de,\Y@tr,\Y@qu)%
       \edef\X@un{\X@qu}\edef\Y@un{\Y@qu}\edef\Ai@{\Aj@}\edef\Aj@{\Ak@}}}
\ctr@ld@f\def\Psc@rveTD#1[#2]{\ifnum\CUR@proj<\tw@\Psc@rvePPTD#1[#2]\else\Psc@rveCPTD#1[#2]\fi}
\ctr@ld@f\def\Psc@rvePPTD#1[#2]{\setc@ntr@l{2}%
    \def\list@num{#2}\extrairelepremi@r\Ak@\de\list@num\Figptpr@j-5:/\Ak@/%
    \extrairelepremi@r\Ai@\de\list@num\Figptpr@j-3:/\Ai@/%
    \extrairelepremi@r\Aj@\de\list@num\Figptpr@j-4:/\Aj@/%
    \f@gnewpath\PSwrit@cmdS{-3}{\c@mmoveto}{\fwf@g}{\X@un}{\Y@un}%
    \figvectPDD -1[-5,-4]%
    \@ecfor\Ak@:=\list@num\do{\Figptpr@j-5:/\Ak@/\figpttraDD-2:=-3/#1,-1/%
       \BdingB@xfalse\PSwrit@cmdS{-2}{}{\fwf@g}{\X@de}{\Y@de}%
       \figvectPDD -1[-3,-5]\figpttraDD-2:=-4/-#1,-1/%
       \PSwrit@cmdS{-2}{}{\fwf@g}{\X@tr}{\Y@tr}\BdingB@xtrue%
       \PSwrit@cmdS{-4}{\c@mcurveto}{\fwf@g}{\X@qu}{\Y@qu}%
       \B@zierBB@x{1}{\Y@un}(\X@un,\X@de,\X@tr,\X@qu)%
       \B@zierBB@x{2}{\X@un}(\Y@un,\Y@de,\Y@tr,\Y@qu)%
       \edef\X@un{\X@qu}\edef\Y@un{\Y@qu}\figptcopyDD-3:/-4/\figptcopyDD-4:/-5/}}
\ctr@ld@f\def\Psc@rveCPTD#1[#2]{\setc@ntr@l{2}%
    \def\list@num{#2}\extrairelepremi@r\Ak@\de\list@num%
    \extrairelepremi@r\Ai@\de\list@num\extrairelepremi@r\Aj@\de\list@num%
    \Figptpr@j-7:/\Ai@/%
    \f@gnewpath\PSwrit@cmd{-7}{\c@mmoveto}{\fwf@g}%
    \figvectPTD -9[\Ak@,\Aj@]%
    \@ecfor\Ak@:=\list@num\do{\figpttraTD-10:=\Ai@/#1,-9/%
       \figvectPTD -9[\Ai@,\Ak@]\figpttraTD-11:=\Aj@/-#1,-9/%
       \subB@zierTD\NBz@rcs[\Ai@,-10,-11,\Aj@]\edef\Ai@{\Aj@}\edef\Aj@{\Ak@}}}
\ctr@ld@f\def\figdrawend{\ifCUR@PS\ifGR@cri\immediate\closeout\fwf@g%
    \immediate\openout\fwf@g=\PSfilen@me\relax%
    \ifPDFm@ke\PSBdingB@x\else%
    \immediate\write\fwf@g{\p@urcent\string!PS-Adobe-2.0 EPSF-2.0}%
    \PShe@der{Creator\string: TeX (fig4tex.tex)}%
    \PShe@der{Title\string: \PSfilen@me}%
    \PShe@der{CreationDate\string: \the\day/\the\month/\the\year}%
    \PSBdingB@x%
    \PShe@der{EndComments}\PSdict@\fi%
    \immediate\write\fwf@g{\c@mgsave}%
    \openin\frf@g=\auxfilen@me\c@pypsfile\fwf@g\frf@g\closein\frf@g%
    \immediate\write\fwf@g{\c@mgrestore}%
    \PSc@mment{End of file.}\immediate\closeout\fwf@g%
    \immediate\openout\fwf@g=\auxfilen@me\immediate\closeout\fwf@g%
    \immediate\write16{File \PSfilen@me\space created.}\fi\fi\CUR@PSfalse\GR@critrue}
\ctr@ld@f\def\PShe@der#1{\immediate\write\fwf@g{\p@urcent\p@urcent#1}}
\ctr@ld@f\def\PSBdingB@x{{\v@lX=\ptT@ptps\c@@rdXmin\v@lY=\ptT@ptps\c@@rdYmin%
     \v@lXa=\ptT@ptps\c@@rdXmax\v@lYa=\ptT@ptps\c@@rdYmax%
     \PShe@der{BoundingBox\string: \repdecn@mb{\v@lX}\space\repdecn@mb{\v@lY}%
     \space\repdecn@mb{\v@lXa}\space\repdecn@mb{\v@lYa}}}}
\ctr@ld@f\def\figdrawfcconnect[#1]{{\ifCUR@PS\ifGR@cri\PSc@mment{fcconnect Points=#1}%
    \Q@s@tfillmode{no}\s@uvc@ntr@l\et@tpsfcconnect\resetc@ntr@l{2}%
    \fcc@nnect@[#1]\resetc@ntr@l\et@tpsfcconnect\PSc@mment{End fcconnect}\fi\fi}}
\ctr@ld@f\def\fcc@nnect@[#1]{\let\N@rm=\n@rmeucDD\def\list@num{#1}%
    \extrairelepremi@r\Ai@\de\list@num\edef\pr@m{\Ai@}\v@leur=\z@\p@rtent=\@ne\c@llgtot%
    \ifcase\fclin@typ@\edef\list@num{[\pr@m,#1,\Ai@}\expandafter\figdrawcurve\list@num]%
    \else\ifdim\fclin@r@d\p@>\z@\Pslin@conge[#1]\else\figdrawline[#1]\fi\fi%
    \v@leur=\@rrowp@s\v@leur\edef\list@num{#1,\Ai@,0}%
    \extrairelepremi@r\Ai@\de\list@num\mili@u=\epsil@n\c@llgpart%
    \advance\mili@u-\epsil@n\advance\mili@u-\delt@\advance\v@leur-\mili@u%
    \ifcase\fclin@typ@\invers@\mili@u\delt@%
    \ifnum\@rrowr@fpt>\z@\advance\delt@-\v@leur\v@leur=\delt@\fi%
    \v@leur=\repdecn@mb\v@leur\mili@u\edef\v@lt{\repdecn@mb\v@leur}%
    \extrairelepremi@r\Ak@\de\list@num%
    \figvectPDD-1[\pr@m,\Aj@]\figpttraDD-6:=\Ai@/\curv@roundness,-1/%
    \figvectPDD-1[\Ak@,\Ai@]\figpttraDD-7:=\Aj@/\curv@roundness,-1/%
    \delt@=\@rrowheadlength\p@\delt@=\C@AHANG\delt@\edef\R@dius{\repdecn@mb{\delt@}}%
    \ifcase\@rrowr@fpt%
    \FigptintercircB@zDD-8::\v@lt,\R@dius[\Ai@,-6,-7,\Aj@]\Q@arrowheadDD[-5,-8]\else%
    \FigptintercircB@zDD-8::\v@lt,\R@dius[\Aj@,-7,-6,\Ai@]\Q@arrowheadDD[-8,-5]\fi%
    \else\advance\delt@-\v@leur%
    \p@rtentiere{\p@rtent}{\delt@}\edef\C@efun{\the\p@rtent}%
    \p@rtentiere{\p@rtent}{\v@leur}\edef\C@efde{\the\p@rtent}%
    \figptbaryDD-5:[\Ai@,\Aj@;\C@efun,\C@efde]\ifcase\@rrowr@fpt%
    \delt@=\@rrowheadlength\unit@\delt@=\C@AHANG\delt@\edef\t@ille{\repdecn@mb{\delt@}}%
    \figvectPDD-2[\Ai@,\Aj@]\vecunit@{-2}{-2}\figpttraDD-5:=-5/\t@ille,-2/\fi%
    \Q@arrowheadDD[\Ai@,-5]\fi}
\ctr@ld@f\def\c@llgtot{\@ecfor\Aj@:=\list@num\do{\figvectP-1[\Ai@,\Aj@]\N@rm\delt@{-1}%
    \advance\v@leur\delt@\advance\p@rtent\@ne\edef\Ai@{\Aj@}}}
\ctr@ld@f\def\c@llgpart{\extrairelepremi@r\Aj@\de\list@num\figvectP-1[\Ai@,\Aj@]\N@rm\delt@{-1}%
    \advance\mili@u\delt@\ifdim\mili@u<\v@leur\edef\pr@m{\Ai@}\edef\Ai@{\Aj@}\c@llgpart\fi}
\ctr@ld@f\def\Pslin@conge[#1]{\ifnum\p@rtent>\tw@{\def\list@num{#1}%
    \extrairelepremi@r\Ai@\de\list@num\extrairelepremi@r\Aj@\de\list@num%
    \figptcopy-6:/\Ai@/\figvectP-3[\Ai@,\Aj@]\vecunit@{-3}{-3}\v@lmax=\result@t%
    \@ecfor\Ak@:=\list@num\do{\figvectP-4[\Aj@,\Ak@]\vecunit@{-4}{-4}%
    \minim@m\v@lmin\v@lmax\result@t\v@lmax=\result@t%
    \det@rm\delt@[-3,-4]\maxim@m\mili@u{\delt@}{-\delt@}\ifdim\mili@u>\Cepsil@n%
    \ifdim\delt@>\z@\figgetangleDD\Angl@[\Aj@,\Ak@,\Ai@]\else%
    \figgetangleDD\Angl@[\Aj@,\Ai@,\Ak@]\fi%
    \v@leur=\PI@deg\advance\v@leur-\Angl@\p@\divide\v@leur\tw@%
    \edef\Angl@{\repdecn@mb\v@leur}\c@ssin{\C@}{\S@}{\Angl@}\v@leur=\fclin@r@d\unit@%
    \v@leur=\S@\v@leur\mili@u=\C@\p@\invers@\mili@u\mili@u%
    \v@leur=\repdecn@mb{\mili@u}\v@leur%
    \minim@m\v@leur\v@leur\v@lmin\edef\t@ille{\repdecn@mb{\v@leur}}%
    \figpttra-5:=\Aj@/-\t@ille,-3/\figdrawline[-6,-5]\figpttra-6:=\Aj@/\t@ille,-4/%
    \figvectNVDD-3[-3]\figvectNVDD-8[-4]\inters@cDD-7:[-5,-3;-6,-8]%
    \ifdim\delt@>\z@\figdrawarccircP-7;\fclin@r@d[-5,-6]\else\figdrawarccircP-7;\fclin@r@d[-6,-5]\fi%
    \else\figdrawline[-6,\Aj@]\figptcopy-6:/\Aj@/\fi% Points alignes
    \edef\Ai@{\Aj@}\edef\Aj@{\Ak@}\figptcopy-3:/-4/}\figdrawline[-6,\Aj@]}\else\figdrawline[#1]\fi}
\ctr@ld@f\def\figdrawfcnode[#1]#2{{\ifCUR@PS\ifGR@cri\PSc@mment{fcnode Points=#1}%
    \s@uvc@ntr@l\et@tpsfcnode\resetc@ntr@l{2}%
    \def\t@xt@{#2}\ifx\t@xt@\empty\def\g@tt@xt{\setbox\Gb@x=\hbox{\Figg@tT{\p@int}}}%
    \else\def\g@tt@xt{\setbox\Gb@x=\hbox{#2}}\fi%
    \v@lmin=\h@rdfcXp@dd\advance\v@lmin\Xp@dd\unit@\multiply\v@lmin\tw@%
    \v@lmax=\h@rdfcYp@dd\advance\v@lmax\Yp@dd\unit@\multiply\v@lmax\tw@%
    \Figv@ctCreg-8(\unit@,-\unit@)\def\list@num{#1}%
    \delt@=\CUR@width bp\divide\delt@\tw@%
    \fcn@de\PSc@mment{End fcnode}\resetc@ntr@l\et@tpsfcnode\fi\fi}}
\ctr@ld@f\def\d@butn@de{\g@tt@xt\v@lX=\wd\Gb@x%
    \v@lY=\ht\Gb@x\advance\v@lY\dp\Gb@x\advance\v@lX\v@lmin\advance\v@lY\v@lmax}
\ctr@ld@f\def\fcn@deE{%
    \@ecfor\p@int:=\list@num\do{\d@butn@de\v@lX=\unssqrttw@\v@lX\v@lY=\unssqrttw@\v@lY%
    \ifdim\thickn@ss\p@>\z@% Begin thickness
    \v@lXa=\v@lX\advance\v@lXa\delt@\v@lXa=\ptT@unit@\v@lXa\edef\XR@d{\repdecn@mb\v@lXa}%
    \v@lYa=\v@lY\advance\v@lYa\delt@\v@lYa=\ptT@unit@\v@lYa\edef\YR@d{\repdecn@mb\v@lYa}%
    \arct@n\v@leur(\v@lXa,\v@lYa)\v@leur=\rdT@deg\v@leur\edef\@nglde{\repdecn@mb\v@leur}%
    {\c@lptellDD-2::\p@int;\XR@d,\YR@d(\@nglde)}% \v@lmin & \v@lmax modified in \c@lptellDD
    \advance\v@leur-\PI@deg\edef\@nglun{\repdecn@mb\v@leur}%
    {\c@lptellDD-3::\p@int;\XR@d,\YR@d(\@nglun)}%
    \figptstra-6=-3,-2,\p@int/\thickn@ss,-8/\Q@s@tfillmode{yes}%
    \Pss@tspecifSt{color=\DDV@thickcolor}%
    \figdrawline[-2,-3,-6,-5]\figdrawarcell-4;\XR@d,\YR@d(\@nglun,\@nglde,0)%
    \Psrest@reSt{color=\DDV@thickcolor}\fi% End thickness
    \v@lX=\ptT@unit@\v@lX\v@lY=\ptT@unit@\v@lY%
    \edef\XR@d{\repdecn@mb\v@lX}\edef\YR@d{\repdecn@mb\v@lY}%
    \Q@s@tfillmode{yes}\Pss@tspecifSt{color=\fcbgc@lor}%
    \figdrawarcell\p@int;\XR@d,\YR@d(0,360,0)%
    \Q@s@tfillmode{no}\Psrest@reSt{color=\fcbgc@lor}\figdrawarcell\p@int;\XR@d,\YR@d(0,360,0)}}
\ctr@ld@f\def\fcn@deL{\delt@=\ptT@unit@\delt@\edef\t@ille{\repdecn@mb\delt@}%
    \@ecfor\p@int:=\list@num\do{\Figg@tXYa{\p@int}\d@butn@de%
    \ifdim\v@lX>\v@lY\itis@Ktrue\else\itis@Kfalse\fi%
    \advance\v@lXa-\v@lX\Figp@intreg-1:(\v@lXa,\v@lYa)%
    \advance\v@lXa\v@lX\advance\v@lYa-\v@lY\Figp@intreg-2:(\v@lXa,\v@lYa)%
    \advance\v@lXa\v@lX\advance\v@lYa\v@lY\Figp@intreg-3:(\v@lXa,\v@lYa)%
    \advance\v@lXa-\v@lX\advance\v@lYa\v@lY\Figp@intreg-4:(\v@lXa,\v@lYa)%
    \ifdim\thickn@ss\p@>\z@% Begin thickness
    \Figg@tXYa{\p@int}\Q@s@tfillmode{yes}\Pss@tspecifSt{color=\DDV@thickcolor}%
    \c@lpt@xt{-1}{-4}\c@lpt@xt@\v@lXa\v@lYa\v@lX\v@lY\c@rre\delt@%
    \Figp@intregDD-9:(\v@lZ,\v@lYa)\Figp@intregDD-11:(\v@lZa,\v@lYa)%
    \c@lpt@xt{-4}{-3}\c@lpt@xt@\v@lYa\v@lXa\v@lY\v@lX\delt@\c@rre%
    \Figp@intregDD-12:(\v@lXa,\v@lZ)\Figp@intregDD-10:(\v@lXa,\v@lZa)%
    \ifitis@K\figptstra-7=-9,-10,-11/\thickn@ss,-8/\figdrawline[-9,-11,-5,-6,-7]\else%
    \figptstra-7=-10,-11,-12/\thickn@ss,-8/\figdrawline[-10,-12,-5,-6,-7]\fi%
    \Psrest@reSt{color=\DDV@thickcolor}\fi% End thickness
    \Q@s@tfillmode{yes}\Pss@tspecifSt{color=\fcbgc@lor}\figdrawline[-1,-2,-3,-4]%
    \Q@s@tfillmode{no}\Psrest@reSt{color=\fcbgc@lor}\figdrawline[-1,-2,-3,-4,-1]}}
\ctr@ld@f\def\c@lpt@xt#1#2{\figvectN-7[#1,#2]\vecunit@{-7}{-7}\figpttra-5:=#1/\t@ille,-7/%
    \figvectP-7[#1,#2]\Figg@tXY{-7}\c@rre=\v@lX\delt@=\v@lY\Figg@tXY{-5}}
\ctr@ld@f\def\c@lpt@xt@#1#2#3#4#5#6{\v@lZ=#6\invers@{\v@lZ}{\v@lZ}\v@leur=\repdecn@mb{#5}\v@lZ%
    \v@lZ=#2\advance\v@lZ-#4\mili@u=\repdecn@mb{\v@leur}\v@lZ%
    \v@lZ=#3\advance\v@lZ\mili@u\v@lZa=-\v@lZ\advance\v@lZa\tw@#1}
\ctr@ld@f\def\fcn@deR{\@ecfor\p@int:=\list@num\do{\Figg@tXYa{\p@int}\d@butn@de%
    \advance\v@lXa-0.5\v@lX\advance\v@lYa-0.5\v@lY\Figp@intreg-1:(\v@lXa,\v@lYa)%
    \advance\v@lXa\v@lX\Figp@intreg-2:(\v@lXa,\v@lYa)%
    \advance\v@lYa\v@lY\Figp@intreg-3:(\v@lXa,\v@lYa)%
    \advance\v@lXa-\v@lX\Figp@intreg-4:(\v@lXa,\v@lYa)%
    \ifdim\thickn@ss\p@>\z@% Begin thickness
    \Q@s@tfillmode{yes}\Pss@tspecifSt{color=\DDV@thickcolor}%
    \Figv@ctCreg-5(-\delt@,-\delt@)\figpttra-9:=-1/1,-5/%
    \Figv@ctCreg-5(\delt@,-\delt@)\figpttra-10:=-2/1,-5/%
    \Figv@ctCreg-5(\delt@,\delt@)\figpttra-11:=-3/1,-5/%
    \figptstra-7=-9,-10,-11/\thickn@ss,-8/\figdrawline[-9,-11,-5,-6,-7]%
    \Psrest@reSt{color=\DDV@thickcolor}\fi% End thickness
    \Q@s@tfillmode{yes}\Pss@tspecifSt{color=\fcbgc@lor}\figdrawline[-1,-2,-3,-4]%
    \Q@s@tfillmode{no}\Psrest@reSt{color=\fcbgc@lor}\figdrawline[-1,-2,-3,-4,-1]}}
\ctr@ld@f\def\Pssetfl@wchart#1=#2|{\keln@mtr#1|%
    \def\n@mref{arr}\ifx\l@debut\n@mref\expandafter\keln@mtr\l@suite|%
     \def\n@mref{owp}\ifx\l@debut\n@mref\update@ttr\D@FTfcarrowposition\P@setfcarrowposition{#2}\else% arrowposition
     \def\n@mref{owr}\ifx\l@debut\n@mref\update@ttr\D@FTfcarrowrefpt\P@setfcarrowrefpt{#2}\else% arrowrefpt
     \W@rnmesAttr{figset flowchart}{#1}\fi\fi\else%
    \def\n@mref{bgc}\ifx\l@debut\n@mref\update@ttr\D@FTfcbgcolor\P@setfcbgcolor{#2}\else% background color
    \def\n@mref{lin}\ifx\l@debut\n@mref\update@ttr\D@FTfcline\P@setfcline{#2}\else% line
    \def\n@mref{pad}\ifx\l@debut\n@mref\update@ttr\D@FTfcxpadding\P@setfcxpadding{#2}%
                                       \update@ttr\D@FTfcypadding\P@setfcypadding{#2}\else% padding
    \def\n@mref{rad}\ifx\l@debut\n@mref\update@ttr\D@FTfcradius\P@setfcradius{#2}\else% connection radius
    \def\n@mref{sha}\ifx\l@debut\n@mref\update@ttr\D@FTfcshape\P@setfcshape{#2}\else% shape
    \def\n@mref{thi}\ifx\l@debut\n@mref\expandafter\keln@mtr\l@suite|%
     \def\n@mref{ckc}\ifx\l@debut\n@mref\update@ttr\D@FTref\P@setfcthickcolor{#2}\else% thickness color
     \def\n@mref{ckn}\ifx\l@debut\n@mref\update@ttr\D@FTfcthickness\P@setfcthickness{#2}\else% thickness
     \W@rnmesAttr{figset flowchart}{#1}\fi\fi\else%
    \def\n@mref{xpa}\ifx\l@debut\n@mref\update@ttr\D@FTfcxpadding\P@setfcxpadding{#2}\else% xpadding
    \def\n@mref{ypa}\ifx\l@debut\n@mref\update@ttr\D@FTfcypadding\P@setfcypadding{#2}\else% ypadding
    \W@rnmesAttr{figset flowchart}{#1}\fi\fi\fi\fi\fi\fi\fi\fi\fi}
\ctr@ln@m\@rrowp@s
\ctr@ld@f\def\P@setfcarrowposition#1{\edef\@rrowp@s{#1}}
\ctr@ln@m\@rrowr@fpt
\ctr@ld@f\def\P@setfcarrowrefpt#1{\setfcr@fpt#1|}
\ctr@ld@f\def\setfcr@fpt#1#2|{\if#1e\def\@rrowr@fpt{1}\else\def\@rrowr@fpt{0}\fi}
\ctr@ln@m\fcbgc@lor
\ctr@ld@f\def\P@setfcbgcolor#1{\edef\fcbgc@lor{#1}}
\ctr@ln@m\fclin@typ@
\ctr@ld@f\def\P@setfcline#1{\setfccurv@#1|}
\ctr@ld@f\def\setfccurv@#1#2|{\if#1c\def\fclin@typ@{0}\else\def\fclin@typ@{1}\fi}
\ctr@ln@m\fclin@r@d
\ctr@ld@f\def\P@setfcradius#1{\edef\fclin@r@d{#1}}
\ctr@ln@m\fcn@de \ctr@ln@m\fcsh@pe
\ctr@ln@m\h@rdfcXp@dd \ctr@ln@m\h@rdfcYp@dd
\ctr@ld@f\def\P@setfcshape#1{\setfcshap@#1|}
\ctr@ld@f\def\setfcshap@#1#2|{%
    \if#1e\let\fcn@de=\fcn@deE\def\h@rdfcXp@dd{4pt}\def\h@rdfcYp@dd{4pt}%
     \edef\fcsh@pe{ellipse}\else%
    \if#1l\let\fcn@de=\fcn@deL\def\h@rdfcXp@dd{4pt}\def\h@rdfcYp@dd{4pt}%
     \edef\fcsh@pe{lozenge}\else%
          \let\fcn@de=\fcn@deR\def\h@rdfcXp@dd{6pt}\def\h@rdfcYp@dd{6pt}%
     \edef\fcsh@pe{rectangle}\fi\fi}
\ctr@ln@m\DDV@thickcolor
\ctr@ld@f\def\P@setfcthickcolor#1{\edef\DDV@thickcolor{#1}}
\ctr@ln@m\thickn@ss
\ctr@ld@f\def\P@setfcthickness#1{\edef\thickn@ss{#1}}
\ctr@ln@m\Xp@dd
\ctr@ld@f\def\P@setfcxpadding#1{\edef\Xp@dd{#1}}
\ctr@ln@m\Yp@dd
\ctr@ld@f\def\P@setfcypadding#1{\edef\Yp@dd{#1}}
\ctr@ld@f\def\figdrawline[#1]{{\ifCUR@PS\ifGR@cri\PSc@mment{line Points=#1}%
    \let\figdrawlign@=\Pslign@P\Pslin@{#1}\PSc@mment{End line}\fi\fi}}
\ctr@ld@f\def\figdrawlineF#1{{\ifCUR@PS\ifGR@cri\PSc@mment{lineF Filename=#1}%
    \let\figdrawlign@=\Pslign@F\Pslin@{#1}\PSc@mment{End lineF}\fi\fi}}
\ctr@ld@f\def\figdrawlineC(#1){{\ifCUR@PS\ifGR@cri\PSc@mment{lineC}%
    \let\figdrawlign@=\Pslign@C\Pslin@{#1}\PSc@mment{End lineC}\fi\fi}}
\ctr@ld@f\def\Pslin@#1{\iffillm@de\figdrawlign@{#1}%
    \f@gfill%
    \else\figdrawlign@{#1}\ifx\derp@int\premp@int%
    \f@gclosestroke%
    \else\f@gstroke\fi\fi}
\ctr@ld@f\def\Pslign@P#1{\def\list@num{#1}\extrairelepremi@r\p@int\de\list@num%
    \edef\premp@int{\p@int}\f@gnewpath%
    \PSwrit@cmd{\p@int}{\c@mmoveto}{\fwf@g}%
    \@ecfor\p@int:=\list@num\do{\PSwrit@cmd{\p@int}{\c@mlineto}{\fwf@g}%
    \edef\derp@int{\p@int}}}
\ctr@ld@f\def\Pslign@F#1{\s@uvc@ntr@l\et@tPslign@F\setc@ntr@l{2}\openin\frf@g=#1\relax%
    \ifeof\frf@g\message{*** File #1 not found !}\end\else%
    \read\frf@g to\tr@c\edef\premp@int{\tr@c}\expandafter\extr@ctCF\tr@c:%
    \f@gnewpath\PSwrit@cmd{-1}{\c@mmoveto}{\fwf@g}%
    \loop\read\frf@g to\tr@c\ifeof\frf@g\mored@tafalse\else\mored@tatrue\fi%
    \ifmored@ta\expandafter\extr@ctCF\tr@c:\PSwrit@cmd{-1}{\c@mlineto}{\fwf@g}%
    \edef\derp@int{\tr@c}\repeat\fi\closein\frf@g\resetc@ntr@l\et@tPslign@F}
\ctr@ln@m\extr@ctCF
\ctr@ld@f\def\extr@ctCFDD#1 #2:{\v@lX=#1\unit@\v@lY=#2\unit@\Figp@intregDD-1:(\v@lX,\v@lY)}
\ctr@ld@f\def\extr@ctCFTD#1 #2 #3:{\v@lX=#1\unit@\v@lY=#2\unit@\v@lZ=#3\unit@%
    \Figp@intregTD-1:(\v@lX,\v@lY,\v@lZ)}
\ctr@ld@f\def\Pslign@C#1{\s@uvc@ntr@l\et@tPslign@C\setc@ntr@l{2}%
    \def\list@num{#1}\extrairelepremi@r\p@int\de\list@num%
    \edef\premp@int{\p@int}\f@gnewpath%
    \expandafter\Pslign@C@\p@int:\PSwrit@cmd{-1}{\c@mmoveto}{\fwf@g}%
    \@ecfor\p@int:=\list@num\do{\expandafter\Pslign@C@\p@int:%
    \PSwrit@cmd{-1}{\c@mlineto}{\fwf@g}\edef\derp@int{\p@int}}%
    \resetc@ntr@l\et@tPslign@C}
\ctr@ld@f\def\Pslign@C@#1 #2:{{\def\t@xt@{#1}\ifx\t@xt@\empty\Pslign@C@#2:% Discard leading spaces
    \else\extr@ctCF#1 #2:\fi}}
\ctr@ld@f\def\Pssetm@sh#1=#2|{\keln@mde#1|%
    \def\n@mref{co}\ifx\l@debut\n@mref\update@ttr\D@FTref\P@setmeshcolor{#2}\else% color
    \def\n@mref{da}\ifx\l@debut\n@mref\update@ttr\D@FTref\P@setmeshdash{#2}\else% dash
    \def\n@mref{di}\ifx\l@debut\n@mref\update@ttr\D@FTmeshdiag\Q@s@tmeshdiag{#2}\else% diag
    \def\n@mref{wi}\ifx\l@debut\n@mref\update@ttr\D@FTref\P@setmeshwidth{#2}\else% width
    \W@rnmesAttr{figset mesh}{#1}\fi\fi\fi\fi}
\ctr@ln@m\c@ntrolmesh
\ctr@ld@f\def\Q@s@tmeshdiag#1{\edef\c@ntrolmesh{#1}}
\ctr@ld@f\def\D@FTmeshdiag{0}    % Valeur par defaut
\ctr@ln@m\DDV@meshcolor
\ctr@ld@f\def\P@setmeshcolor#1{\edef\DDV@meshcolor{#1}}
\ctr@ln@m\DDV@meshdash
\ctr@ld@f\def\P@setmeshdash#1{\edef\DDV@meshdash{#1}}
\ctr@ln@m\DDV@meshwidth
\ctr@ld@f\def\P@setmeshwidth#1{\edef\DDV@meshwidth{#1}}
\ctr@ld@f\def\figdrawmesh#1,#2[#3,#4,#5,#6]{{\ifCUR@PS\ifGR@cri%
    \PSc@mment{mesh N1=#1, N2=#2, Quadrangle=[#3,#4,#5,#6]}\s@uvc@ntr@l\et@tpsmesh%
    \Pss@tspecifSt{color=\DDV@meshcolor,dash=\DDV@meshdash,width=\DDV@meshwidth}%
    \setc@ntr@l{2}%
    \ifnum#1>\@ne\Psmeshp@rt#1[#3,#4,#5,#6]\fi%
    \ifnum#2>\@ne\Psmeshp@rt#2[#4,#5,#6,#3]\fi%
    \ifnum\c@ntrolmesh>\z@\Psmeshdi@g#1,#2[#3,#4,#5,#6]\fi%
    \ifnum\c@ntrolmesh<\z@\Psmeshdi@g#2,#1[#4,#5,#6,#3]\fi%
    \Psrest@reSt{color=\DDV@meshcolor,dash=\DDV@meshdash,width=\DDV@meshwidth}%
    \figdrawline[#3,#4,#5,#6,#3]\PSc@mment{End mesh}\resetc@ntr@l\et@tpsmesh\fi\fi}}
\ctr@ld@f\def\Psmeshp@rt#1[#2,#3,#4,#5]{{\l@mbd@un=\@ne\l@mbd@de=#1\loop%
    \ifnum\l@mbd@un<#1\advance\l@mbd@de\m@ne\figptbary-1:[#2,#3;\l@mbd@de,\l@mbd@un]%
    \figptbary-2:[#5,#4;\l@mbd@de,\l@mbd@un]\figdrawline[-1,-2]\advance\l@mbd@un\@ne\repeat}}
\ctr@ld@f\def\Psmeshdi@g#1,#2[#3,#4,#5,#6]{\figptcopy-2:/#3/\figptcopy-3:/#6/%
    \l@mbd@un=\z@\l@mbd@de=#1\loop\ifnum\l@mbd@un<#1%
    \advance\l@mbd@un\@ne\advance\l@mbd@de\m@ne\figptcopy-1:/-2/\figptcopy-4:/-3/%
    \figptbary-2:[#3,#4;\l@mbd@de,\l@mbd@un]%
    \figptbary-3:[#6,#5;\l@mbd@de,\l@mbd@un]\Psmeshdi@gp@rt#2[-1,-2,-3,-4]\repeat}
\ctr@ld@f\def\Psmeshdi@gp@rt#1[#2,#3,#4,#5]{{\l@mbd@un=\z@\l@mbd@de=#1\loop%
    \ifnum\l@mbd@un<#1\figptbary-5:[#2,#5;\l@mbd@de,\l@mbd@un]%
    \advance\l@mbd@de\m@ne\advance\l@mbd@un\@ne%
    \figptbary-6:[#3,#4;\l@mbd@de,\l@mbd@un]\figdrawline[-5,-6]\repeat}}
\ctr@ln@m\figdrawnormal
\ctr@ld@f\def\Q@normalDD#1,#2[#3,#4]{{\ifCUR@PS\ifGR@cri%
    \PSc@mment{normal Length=#1, Lambda=#2 [Pt1,Pt2]=[#3,#4]}%
    \s@uvc@ntr@l\et@tpsnormal\resetc@ntr@l{2}\figptendnormal-6::#1,#2[#3,#4]%
    \figptcopyDD-5:/-1/\figdrawarrow[-5,-6]%
    \PSc@mment{End normal}\resetc@ntr@l\et@tpsnormal\fi\fi}}
\ctr@ld@f\def\figreset#1{\trtlis@rg{#1}{\Psreset@}}
\ctr@ld@f\def\Psreset@#1|{\def\t@xt@{#1}\ifx\t@xt@\empty\P@resetg@n% general settings
    \else\keln@mde#1|%
    \def\n@mref{al}\ifx\l@debut\n@mref%
        \figset altitude(blcolor=default,bldash=default,blwidth=default,%
        sqcolor=default,sqdash=default,sqwidth=default)\else% altitude
    \def\n@mref{ar}\ifx\l@debut\n@mref\figresetarrowhead\else% arrowhead
    \def\n@mref{cu}\ifx\l@debut\n@mref\figset curve(roundness=\D@FTroundness)\else% curve
    \def\n@mref{ge}\ifx\l@debut\n@mref\P@resetg@n\else% general settings
    \def\n@mref{fl}\ifx\l@debut\n@mref%
        \figset flowchart(arrowp=\D@FTfcarrowposition,arrowr=\D@FTfcarrowrefpt,%
	bgcolor=\D@FTfcbgcolor,line=\D@FTfcline,radius=\D@FTfcradius,%
	shape=\D@FTfcshape,thickcolor=default,thickness=\D@FTfcthickness,%
	xpadd=\D@FTfcxpadding,ypadd=\D@FTfcypadding)\else% flow chart
    \def\n@mref{me}\ifx\l@debut\n@mref\figset mesh(diag=\D@FTmeshdiag,%
        color=default,dash=default,width=default)\else% mesh
    \def\n@mref{tr}\ifx\l@debut\n@mref%
        \figset trimesh(color=default,dash=default,width=default)\else% trimesh
    \W@rnmeskwd{figreset}{#1}\fi\fi\fi\fi\fi\fi\fi\fi}
\ctr@ld@f\def\P@resetg@n{\figset (color=\D@FTcolor,dash=\D@FTdash,fill=\D@FTfill,%
    join=\D@FTjoin,width=\D@FTwidth)}
\ctr@ld@f\def\figset#1(#2){\def\t@xt@{#1}\ifx\t@xt@\empty\trtlis@rg{#2}{\Pssetg@n}% general settings
    \else\keln@mde#1|%
    \def\n@mref{al}\ifx\l@debut\n@mref\trtlis@rg{#2}{\Psset@lti}\else% altitude
    \def\n@mref{ar}\ifx\l@debut\n@mref\trtlis@rg{#2}{\Psset@rrowhe@d}\else% arrowhead
    \def\n@mref{cu}\ifx\l@debut\n@mref\trtlis@rg{#2}{\Pssetc@rve}\else% curve
    \def\n@mref{fl}\ifx\l@debut\n@mref\trtlis@rg{#2}{\Pssetfl@wchart}\else% flow chart
    \def\n@mref{ge}\ifx\l@debut\n@mref\trtlis@rg{#2}{\Pssetg@n}\else% general settings
    \def\n@mref{me}\ifx\l@debut\n@mref\trtlis@rg{#2}{\Pssetm@sh}\else% mesh
    \def\n@mref{pr}\ifx\l@debut\n@mref\ifCUR@PS\W@rnmesIgn{figset proj(...)}%
     \else\trtlis@rg{#2}{\Figsetpr@j}\fi\else% projection
    \def\n@mref{tr}\ifx\l@debut\n@mref\trtlis@rg{#2}{\Pssettrim@sh}\else% trimesh
    \def\n@mref{wr}\ifx\l@debut\n@mref\let\M@cro=\Figsetwr@te\trtlis@rgtok{#2,|}\else% write
    \W@rnmeskwd{figset}{#1}\fi\fi\fi\fi\fi\fi\fi\fi\fi\fi\ignorespaces}
\ctr@ld@f\def\figsetdefault#1(#2){\ifCUR@PS\W@rnmesIgn{figsetdefault}\else%
    \def\t@xt@{#1}\ifx\t@xt@\empty\trtlis@rg{#2}{\Pssd@g@n}\else\keln@mun#1|% general settings
    \def\n@mref{a}\ifx\l@debut\n@mref\trtlis@rg{#2}{\Pssd@@rrowhe@d}\else% arrowhead
    \def\n@mref{c}\ifx\l@debut\n@mref\trtlis@rg{#2}{\Pssd@c@rve}\else% curve
    \def\n@mref{g}\ifx\l@debut\n@mref\trtlis@rg{#2}{\Pssd@g@n}\else% general settings
    \def\n@mref{f}\ifx\l@debut\n@mref\trtlis@rg{#2}{\Pssd@fl@wchart}\else% flow chart
    \def\n@mref{m}\ifx\l@debut\n@mref\trtlis@rg{#2}{\Pssd@m@sh}\else% mesh
    \W@rnmeskwd{figsetdefault}{#1}\fi\fi\fi\fi\fi\fi\initpss@ttings\fi}
\ctr@ld@f\def\Pssd@g@n#1=#2|{\keln@mun#1|%
    \def\n@mref{c}\ifx\l@debut\n@mref\edef\D@FTcolor{#2}\else% color
    \def\n@mref{d}\ifx\l@debut\n@mref\edef\D@FTdash{#2}\else% line dash
    \def\n@mref{f}\ifx\l@debut\n@mref\edef\D@FTfill{#2}\else% fillmode
    \def\n@mref{j}\ifx\l@debut\n@mref\edef\D@FTjoin{#2}\else% line join
    \def\n@mref{u}\ifx\l@debut\n@mref\edef\D@FTupdate{#2}\Q@s@tupdate{#2}\else% update
    \def\n@mref{w}\ifx\l@debut\n@mref\edef\D@FTwidth{#2}\else% line width
    \W@rnmesAttr{figsetdefault}{#1}\fi\fi\fi\fi\fi\fi}
\ctr@ld@f\def\Pssd@@rrowhe@d#1=#2|{\keln@mun#1|%
    \def\n@mref{a}\ifx\l@debut\n@mref\edef\D@FTarrowheadangle{#2}\else% angle
    \def\n@mref{f}\ifx\l@debut\n@mref\edef\D@FTarrowheadfill{#2}\else% fillmode
    \def\n@mref{l}\ifx\l@debut\n@mref\y@tiunit{#2}\ifunitpr@sent%
     \edef\D@FTh@rdahlength{#2}\else\edef\D@FTh@rdahlength{#2pt}%
     \message{*** \BS@ figsetdefault (..., #1=#2, ...) : unit is missing, pt is assumed.}%
     \fi\else% length
    \def\n@mref{o}\ifx\l@debut\n@mref\edef\D@FTarrowheadout{#2}\else% out
    \def\n@mref{r}\ifx\l@debut\n@mref\edef\D@FTarrowheadratio{#2}\else% ratio
    \W@rnmesAttr{figsetdefault arrowhead}{#1}\fi\fi\fi\fi\fi}
\ctr@ld@f\def\Pssd@c@rve#1=#2|{\keln@mun#1|%
    \def\n@mref{r}\ifx\l@debut\n@mref\edef\D@FTroundness{#2}\else%
    \W@rnmesAttr{figsetdefault curve}{#1}\fi}
\ctr@ld@f\def\Pssd@fl@wchart#1=#2|{\keln@mtr#1|%
    \def\n@mref{arr}\ifx\l@debut\n@mref\expandafter\keln@mtr\l@suite|%
     \def\n@mref{owp}\ifx\l@debut\n@mref\edef\D@FTfcarrowposition{#2}\else% arrowposition
     \def\n@mref{owr}\ifx\l@debut\n@mref\edef\D@FTfcarrowrefpt{#2}\else% arrowrefpt
                     \W@rnmesAttr{figsetdefault flowchart}{#1}\fi\fi\else%
    \def\n@mref{bgc}\ifx\l@debut\n@mref\edef\D@FTfcbgcolor{#2}\else% background color
    \def\n@mref{lin}\ifx\l@debut\n@mref\edef\D@FTfcline{#2}\else% line
    \def\n@mref{pad}\ifx\l@debut\n@mref\edef\D@FTfcxpadding{#2}%
                    \edef\D@FTfcypadding{#2}\else% padding
    \def\n@mref{rad}\ifx\l@debut\n@mref\edef\D@FTfcradius{#2}\else% connection radius
    \def\n@mref{sha}\ifx\l@debut\n@mref\edef\D@FTfcshape{#2}\else% shape
    \def\n@mref{thi}\ifx\l@debut\n@mref\expandafter\keln@mtr\l@suite|%
     \def\n@mref{ckn}\ifx\l@debut\n@mref\edef\D@FTfcthickness{#2}\else% thickness
                     \W@rnmesAttr{figsetdefault flowchart}{#1}\fi\else%
    \def\n@mref{xpa}\ifx\l@debut\n@mref\edef\D@FTfcxpadding{#2}\else% xpadding
    \def\n@mref{ypa}\ifx\l@debut\n@mref\edef\D@FTfcypadding{#2}\else% ypadding
    \W@rnmesAttr{figsetdefault flowchart}{#1}\fi\fi\fi\fi\fi\fi\fi\fi\fi}
\ctr@ld@f\def\D@FTfcarrowposition{0.5}
\ctr@ld@f\def\D@FTfcarrowrefpt{start}
\ctr@ld@f\def\D@FTfcbgcolor{1}
\ctr@ld@f\def\D@FTfcline{polygon}
\ctr@ld@f\def\D@FTfcradius{0}
\ctr@ld@f\def\D@FTfcshape{rectangle}
\ctr@ld@f\def\D@FTfcthickness{0}
\ctr@ld@f\def\D@FTfcxpadding{0}
\ctr@ld@f\def\D@FTfcypadding{0}
\ctr@ld@f\def\Pssd@m@sh#1=#2|{\keln@mun#1|%
    \def\n@mref{d}\ifx\l@debut\n@mref\edef\D@FTmeshdiag{#2}\else%
    \W@rnmesAttr{figsetdefault mesh}{#1}\fi}
\ctr@ln@w{newif}\iffillm@de
\ctr@ld@f\def\Q@s@tfillmode#1{\expandafter\setfillm@de#1:}
\ctr@ld@f\def\setfillm@de#1#2:{\if#1n\fillm@defalse\else\fillm@detrue\fi}
\ctr@ld@f\def\D@FTfill{no}     % Valeur par defaut
\ctr@ln@w{newif}\ifGRupdatem@de
\ctr@ld@f\def\Q@s@tupdate#1{\ifCUR@PS\W@rnmesIgn{figset (update=...)}%
    \else\expandafter\setupd@te#1:\fi}
\ctr@ld@f\def\setupd@te#1#2:{\if#1n\GRupdatem@defalse\else\GRupdatem@detrue\fi}
\ctr@ld@f\def\D@FTupdate{no}     % Valeur par defaut
\ctr@ln@m\CUR@color \ctr@ln@m\CUR@colorc@md
\ctr@ld@f\def\s@uvcolor#1{\edef#1{\CUR@color}}
\ctr@ld@f\def\D@FTcolor{0}       % Valeur par defaut
\ctr@ld@f\def\Pssetc@lor#1{\ifGR@cri\result@tent=\@ne\expandafter\c@lnbV@l#1 :%
    \def\CUR@color{}\def\CUR@colorc@md{}%
    \ifcase\result@tent\or\Q@s@tgray{#1}\or\or\Q@s@trgb{#1}\or\Q@s@tcmyk{#1}\fi\fi}
\ctr@ln@m\CUR@colorc@mdStroke
\ctr@ld@f\def\Q@s@tcmyk#1{\ifGR@cri\def\CUR@color{#1}\def\CUR@colorc@md{\c@msetcmykcolor}%
    \def\CUR@colorc@mdStroke{\c@msetcmykcolorStroke}%
    \ifCUR@PS\PSc@mment{setcmyk Color=#1}\us@primarC@lor\fi\fi}
\ctr@ld@f\def\Q@s@trgb#1{\ifGR@cri\def\CUR@color{#1}\def\CUR@colorc@md{\c@msetrgbcolor}%
    \def\CUR@colorc@mdStroke{\c@msetrgbcolorStroke}%
    \ifCUR@PS\PSc@mment{setrgb Color=#1}\us@primarC@lor\fi\fi}
\ctr@ld@f\def\Q@s@tgray#1{\ifGR@cri\def\CUR@color{#1}\def\CUR@colorc@md{\c@msetgray}%
    \def\CUR@colorc@mdStroke{\c@msetgrayStroke}%
    \ifCUR@PS\PSc@mment{setgray Gray level=#1}\us@primarC@lor\fi\fi}
\ctr@ln@m\fillc@md
\ctr@ld@f\def\us@primarC@lor{\immediate\write\fwf@g{\d@fprimarC@lor}%
    \let\fillc@md=\prfillc@md}
\ctr@ld@f\def\prfillc@md{\d@fprimarC@lor\space\c@mfill}
\ctr@ld@f\def\c@lnbV@l#1 #2:{\def\t@xt@{#1}\relax\ifx\t@xt@\empty\c@lnbV@l#2:% Discard leading spaces
    \else\c@lnbV@l@#1 #2:\fi}
\ctr@ld@f\def\c@lnbV@l@#1 #2:{\def\t@xt@{#2}\ifx\t@xt@\empty%
    \def\t@xt@{#1}\ifx\t@xt@\empty\advance\result@tent\m@ne\fi% Discard trailing spaces
    \else\advance\result@tent\@ne\c@lnbV@l@#2:\fi}
\ctr@ld@f\def\Blackcmyk{0 0 0 1}
\ctr@ld@f\def\Whitecmyk{0 0 0 0}
\ctr@ld@f\def\Cyancmyk{1 0 0 0}
\ctr@ld@f\def\Magentacmyk{0 1 0 0}
\ctr@ld@f\def\Yellowcmyk{0 0 1 0}
\ctr@ld@f\def\Redcmyk{0 1 1 0}
\ctr@ld@f\def\Greencmyk{1 0 1 0}
\ctr@ld@f\def\Bluecmyk{1 1 0 0}
\ctr@ld@f\def\Graycmyk{0 0 0 0.50}
\ctr@ld@f\def\BrickRedcmyk{0 0.89 0.94 0.28} % PANTONE 1805
\ctr@ld@f\def\Browncmyk{0 0.81 1 0.60} % PANTONE 1615
\ctr@ld@f\def\ForestGreencmyk{0.91 0 0.88 0.12} % PANTONE 349
\ctr@ld@f\def\Goldenrodcmyk{ 0 0.10 0.84 0} % PANTONE 109
\ctr@ld@f\def\Marooncmyk{0 0.87 0.68 0.32} % PANTONE 201
\ctr@ld@f\def\Orangecmyk{0 0.61 0.87 0} % PANTONE ORANGE-021
\ctr@ld@f\def\Purplecmyk{0.45 0.86 0 0} % PANTONE PURPLE
\ctr@ld@f\def\RoyalBluecmyk{1. 0.50 0 0} % No PANTONE match
\ctr@ld@f\def\Violetcmyk{0.79 0.88 0 0} % PANTONE VIOLET
\ctr@ld@f\def\Blackrgb{0 0 0}
\ctr@ld@f\def\Whitergb{1 1 1}
\ctr@ld@f\def\Redrgb{1 0 0}
\ctr@ld@f\def\Greenrgb{0 1 0}
\ctr@ld@f\def\Bluergb{0 0 1}
\ctr@ld@f\def\Cyanrgb{0 1 1}
\ctr@ld@f\def\Magentargb{1 0 1}
\ctr@ld@f\def\Yellowrgb{1 1 0}
\ctr@ld@f\def\Grayrgb{0.5 0.5 0.5}
\ctr@ld@f\def\Chocolatergb{0.824 0.412 0.118}
\ctr@ld@f\def\DarkGoldenrodrgb{0.722 0.525 0.043}
\ctr@ld@f\def\DarkOrangergb{1 0.549 0}
\ctr@ld@f\def\Firebrickrgb{0.698 0.133 0.133}
\ctr@ld@f\def\ForestGreenrgb{0.133 0.545 0.133}
\ctr@ld@f\def\Goldrgb{1 0.843 0}
\ctr@ld@f\def\HotPinkrgb{1 0.412 0.706}
\ctr@ld@f\def\Maroonrgb{0.690 0.188 0.376}
\ctr@ld@f\def\Pinkrgb{1 0.753 0.796}
\ctr@ld@f\def\RoyalBluergb{0.255 0.412 0.882}
\ctr@ld@f\def\Pssetg@n#1=#2|{\keln@mun#1|%
    \def\n@mref{c}\ifx\l@debut\n@mref\update@ttr\D@FTcolor\Pssetc@lor{#2}\else% color
    \def\n@mref{d}\ifx\l@debut\n@mref\update@ttr\D@FTdash\Q@s@tdash{#2}\else% line dash
    \def\n@mref{f}\ifx\l@debut\n@mref\update@ttr\D@FTfill\Q@s@tfillmode{#2}\else% fillmode
    \def\n@mref{j}\ifx\l@debut\n@mref\update@ttr\D@FTjoin\Q@s@tjoin{#2}\else% line join
    \def\n@mref{u}\ifx\l@debut\n@mref\update@ttr\D@FTupdate\Q@s@tupdate{#2}\else% update
    \def\n@mref{w}\ifx\l@debut\n@mref\update@ttr\D@FTwidth\Q@s@twidth{#2}\else% line width
    \W@rnmesAttr{figset}{#1}\fi\fi\fi\fi\fi\fi}
\ctr@ln@m\CUR@dash
\ctr@ld@f\def\s@uvdash#1{\edef#1{\CUR@dash}}
\ctr@ld@f\def\D@FTdash{1}        % Valeur par defaut (numero sans espace)
\ctr@ld@f\def\Q@s@tdash#1{\ifGR@cri\edef\CUR@dash{#1}\ifCUR@PS\expandafter\Pssetd@sh#1 :\fi\fi}
\ctr@ld@f\def\Pssetd@shI#1{\PSc@mment{setdash Index=#1}\ifcase#1%
    \or\immediate\write\fwf@g{[] 0 \c@msetdash}%         Index=1
    \or\immediate\write\fwf@g{[6 2] 0 \c@msetdash}%      Index=2
    \or\immediate\write\fwf@g{[4 2] 0 \c@msetdash}%      Index=3
    \or\immediate\write\fwf@g{[2 2] 0 \c@msetdash}%      Index=4
    \or\immediate\write\fwf@g{[1 2] 0 \c@msetdash}%      Index=5
    \or\immediate\write\fwf@g{[2 4] 0 \c@msetdash}%      Index=6
    \or\immediate\write\fwf@g{[3 5] 0 \c@msetdash}%      Index=7
    \or\immediate\write\fwf@g{[3 3] 0 \c@msetdash}%      Index=8
    \or\immediate\write\fwf@g{[3 5 1 5] 0 \c@msetdash}%  Index=9
    \or\immediate\write\fwf@g{[6 4 2 4] 0 \c@msetdash}%  Index=10
    \fi}
\ctr@ld@f\def\Pssetd@sh#1 #2:{{\def\t@xt@{#1}\ifx\t@xt@\empty\Pssetd@sh#2:% Discard leading spaces
    \else\def\t@xt@{#2}\ifx\t@xt@\empty\Pssetd@shI{#1}\else\s@mme=\@ne\def\debutp@t{#1}%
    \an@lysd@sh#2:\ifodd\s@mme\edef\debutp@t{\debutp@t\space\finp@t}\def\finp@t{0}\fi%
    \PSc@mment{setdash Pattern=#1 #2}%
    \immediate\write\fwf@g{[\debutp@t] \finp@t\space\c@msetdash}\fi\fi}}
\ctr@ld@f\def\an@lysd@sh#1 #2:{\def\t@xt@{#2}\ifx\t@xt@\empty\def\finp@t{#1}\else%
    \edef\debutp@t{\debutp@t\space#1}\advance\s@mme\@ne\an@lysd@sh#2:\fi}
\ctr@ln@m\CUR@width
\ctr@ld@f\def\s@uvwidth#1{\edef#1{\CUR@width}}
\ctr@ld@f\def\D@FTwidth{0.4}     % Valeur par defaut
\ctr@ld@f\def\Q@s@twidth#1{\ifGR@cri\edef\CUR@width{#1}\ifCUR@PS%
    \PSc@mment{setwidth Width=#1}\immediate\write\fwf@g{#1 \c@msetlinewidth}\fi\fi}
\ctr@ln@m\CUR@join
\ctr@ld@f\def\s@uvjoin#1{\edef#1{\CUR@join}}
\ctr@ld@f\def\D@FTjoin{miter}   % Valeur par defaut
\ctr@ld@f\def\Q@s@tjoin#1{\ifGR@cri\edef\CUR@join{#1}\ifCUR@PS\expandafter\Pssetj@in#1:\fi\fi}
\ctr@ld@f\def\Pssetj@in#1#2:{\PSc@mment{setjoin join=#1}%
    \if#1r\def\t@xt@{1}\else\if#1b\def\t@xt@{2}\else\def\t@xt@{0}\fi\fi%
    \immediate\write\fwf@g{\t@xt@\space\c@msetlinejoin}}
\ctr@ld@f\def\Pss@tspecifSt#1{\trtlis@rg{#1}{\Pss@tspecifSt@}}
\ctr@ld@f\def\Pss@tspecifSt@#1=#2|{\keln@mun#1|%
    \def\n@mref{c}\ifx\l@debut\n@mref\def\n@mref{#2}\ifx\n@mref\D@FTref\else%
     \s@uvcolor{\typ@color}\Pssetc@lor{#2}\fi\else% color
    \def\n@mref{d}\ifx\l@debut\n@mref\def\n@mref{#2}\ifx\n@mref\D@FTref\else%
     \s@uvdash{\typ@dash}\Q@s@tdash{#2}\fi\else% line dash
    \def\n@mref{j}\ifx\l@debut\n@mref\def\n@mref{#2}\ifx\n@mref\D@FTref\else%
     \s@uvjoin{\typ@join}\Q@s@tjoin{#2}\fi\else% line join
    \def\n@mref{w}\ifx\l@debut\n@mref\def\n@mref{#2}\ifx\n@mref\D@FTref\else%
     \s@uvwidth{\typ@width}\Q@s@twidth{#2}\fi\else% line width
    \W@rnmeskwd{Pss@tspecifSt}{#1}\fi\fi\fi\fi}
\ctr@ld@f\def\Psrest@reSt#1{\trtlis@rg{#1}{\Psrest@reSt@}}
\ctr@ld@f\def\Psrest@reSt@#1=#2|{\keln@mun#1|%
    \def\n@mref{c}\ifx\l@debut\n@mref\def\n@mref{#2}\ifx\n@mref\D@FTref\else%
     \Pssetc@lor{\typ@color}\fi\else% color
    \def\n@mref{d}\ifx\l@debut\n@mref\def\n@mref{#2}\ifx\n@mref\D@FTref\else%
     \Q@s@tdash{\typ@dash}\fi\else% line dash
    \def\n@mref{j}\ifx\l@debut\n@mref\def\n@mref{#2}\ifx\n@mref\D@FTref\else%
     \Q@s@tjoin{\typ@join}\fi\else% line join
    \def\n@mref{w}\ifx\l@debut\n@mref\def\n@mref{#2}\ifx\n@mref\D@FTref\else%
     \Q@s@twidth{\typ@width}\fi\else% line width
    \W@rnmeskwd{Psrest@reSt}{#1}\fi\fi\fi\fi}
\ctr@ld@f\def\Pssettrim@sh#1=#2|{\keln@mde#1|%
    \def\n@mref{co}\ifx\l@debut\n@mref\update@ttr\D@FTref\P@settmeshcolor{#2}\else% color
    \def\n@mref{da}\ifx\l@debut\n@mref\update@ttr\D@FTref\P@settmeshdash{#2}\else% dash
    \def\n@mref{wi}\ifx\l@debut\n@mref\update@ttr\D@FTref\P@settmeshwidth{#2}\else% width
    \W@rnmesAttr{figset trimesh}{#1}\fi\fi\fi}
\ctr@ln@m\DDV@tmeshcolor
\ctr@ld@f\def\P@settmeshcolor#1{\edef\DDV@tmeshcolor{#1}}
\ctr@ln@m\DDV@tmeshdash
\ctr@ld@f\def\P@settmeshdash#1{\edef\DDV@tmeshdash{#1}}
\ctr@ln@m\DDV@tmeshwidth
\ctr@ld@f\def\P@settmeshwidth#1{\edef\DDV@tmeshwidth{#1}}
\ctr@ld@f\def\figdrawtrimesh#1[#2,#3,#4]{{\ifCUR@PS\ifGR@cri%
    \PSc@mment{trimesh Type=#1, Triangle=[#2,#3,#4]}%
    \s@uvc@ntr@l\et@tpstrimesh\ifnum#1>\@ne%
    \Pss@tspecifSt{color=\DDV@tmeshcolor,dash=\DDV@tmeshdash,width=\DDV@tmeshwidth}%
    \setc@ntr@l{2}%
    \Pstrimeshp@rt#1[#2,#3,#4]\Pstrimeshp@rt#1[#3,#4,#2]\Pstrimeshp@rt#1[#4,#2,#3]%
    \Psrest@reSt{color=\DDV@tmeshcolor,dash=\DDV@tmeshdash,width=\DDV@tmeshwidth}%
    \fi\figdrawline[#2,#3,#4,#2]%
    \PSc@mment{End trimesh}\resetc@ntr@l\et@tpstrimesh\fi\fi}}
\ctr@ld@f\def\Pstrimeshp@rt#1[#2,#3,#4]{{\l@mbd@un=\@ne\l@mbd@de=#1\loop\ifnum\l@mbd@de>\@ne%
    \advance\l@mbd@de\m@ne\figptbary-1:[#2,#3;\l@mbd@de,\l@mbd@un]%
    \figptbary-2:[#2,#4;\l@mbd@de,\l@mbd@un]\figdrawline[-1,-2]%
    \advance\l@mbd@un\@ne\repeat}}
\initpr@lim\initpss@ttings\initPDF@rDVI% Initialisation preliminaire
\ctr@ln@w{newbox}\figBoxA
\ctr@ln@w{newbox}\figBoxB
\ctr@ln@w{newbox}\figBoxC
\catcode`\@=12

\usepackage{tikz}
\usetikzlibrary{shapes,backgrounds}
\def\firstcircle{(0,0) circle (1.5cm)}
\def\secondcircle{(45:2cm) circle (1.5cm)}
\def\thirdcircle{(0:2cm) circle (1.5cm)}

\usepackage{mathtools}
\mathtoolsset{showonlyrefs=true}

\newcommand\uhr{\upharpoonright}
\newcommand\arr{\rightarrow}
\def\aa{\alpha}
\def\bb{\beta}
\def\frm{\mathfrak{m}}
\def\iff{{\it iff}\,}
\def\sfh{\mathsf{h}}
\def\I{\mathbb{I}}
\def\pp{\prime}
\def\dpp{{\prime\prime}}
\def\anot{\alpha_0}
\def\M{{\mathcal{M}}}
\def\ee{\mathsf{e}}

\def\s{\sigma}
\def\sd{\sigma_{\rm d}}

\def\lm{\lambda}

\def\radius{3cm}
\def\softness{0.4}

\definecolor{softred}{rgb}{1,\softness,\softness}
\definecolor{softgreen}{rgb}{\softness,1,\softness}
\definecolor{softblue}{rgb}{\softness,\softness,1}
\definecolor{softrg}{rgb}{1,1,\softness}
\definecolor{softrb}{rgb}{1,\softness,1}
\definecolor{softgb}{rgb}{\softness,1,1}

\newcounter{counter_a}
\newenvironment{myenum}{\begin{list}{{\rm(\roman{counter_a})}}%
{\usecounter{counter_a}
\setlength{\itemsep}{1.ex}\setlength{\topsep}{1.0ex}
\setlength{\leftmargin}{5ex}\setlength{\labelwidth}{5ex}}}{\end{list}}

\newcommand\ds{\displaystyle}

\newcommand{\red}{\color{darkred2}}

\newcommand{\eg}{{\it e.g.}\,}
\newcommand{\ie}{{\it i.e.}\,}
\newcommand{\cf}{{\it cf.}\,}

\numberwithin{figure}{section}
\numberwithin{equation}{section}
\theoremstyle{plain}% default
\newtheorem*{thm*}{Theorem}
\newtheorem{thm}{Theorem}[section]
\newtheorem{hyp}[thm]{Assumption}
\newtheorem{lem}[thm]{Lemma}
\newtheorem{prop}[thm]{Proposition}

\newtheorem{dfn}[thm]{Definition}
\theoremstyle{remark}
\newtheorem{remark}[thm]{Remark}
\theoremstyle{plain}

\newcommand{\dsp}{\displaystyle}
\newcommand{\spec}{{\mathrm{spec}}}
\newcommand{\rmd}{\mathrm{d}}
\newcommand{\dd}{\mathrm{d}}
\newcommand{\rmi}{\mathrm{i}}
\newcommand{\supp}{\operatorname{supp}}
\newcommand{\beu}{\begin{equation*}}
\newcommand{\eeu}{\end{equation*}}
\newcommand{\besu}{\begin{equation*}
\begin{aligned}}
\newcommand{\eesu}{\end{aligned}
\end{equation*}}
\newcommand{\bes}{\begin{equation}
\begin{aligned}}
\newcommand{\ees}{\end{aligned}
\end{equation}}

\newcommand\cA{\mathcal A}
\newcommand\cB{\mathcal B}
\newcommand\cD{\mathcal D}
\newcommand\cF{\mathcal F}
\newcommand\cG{\mathcal G}
\newcommand\cH{\mathcal H}
\newcommand\cK{\mathcal K}
\newcommand\cJ{\mathcal J}
\newcommand\cL{\mathcal L}
\newcommand\cM{\mathcal M}
\newcommand\cN{\mathcal N}
\newcommand\cP{\mathcal P}
\newcommand\cR{\mathcal R}
\newcommand\cS{\mathcal S}
\newcommand\CC{\mathbb C}
\newcommand\NN{\mathbb N}
\newcommand\RR{\mathbb R}
\newcommand\ZZ{\mathbb Z}
\newcommand\frA{\mathfrak A}
\newcommand\frB{\mathfrak B}
\newcommand\frS{\mathfrak S}
\newcommand\fra{\mathfrak a}
\newcommand\frq{\mathfrak q}
\newcommand\frs{\mathfrak s}
\newcommand\frp{\mathfrak p}
\newcommand\frh{\mathfrak h}
\newcommand\frf{\mathfrak f}
\newcommand\fraD{\mathfrak{a}_{\rm D}}
\newcommand\fraN{\mathfrak{a}}
\newcommand\eps{\varepsilon}
\newcommand\dis{\displaystyle}
\newcommand\ov{\overline}
\newcommand\wt{\widetilde}
\newcommand\wh{\widehat}
\newcommand{\defeq}{\mathrel{\mathop:}=}
\newcommand{\eqdef}{=\mathrel{\mathop:}}
\newcommand{\defequ}{\mathrel{\mathop:}\hspace*{-0.72ex}&=}
\newcommand\vectn[2]{\begin{pmatrix} #1 \\ #2 \end{pmatrix}}
\newcommand\vect[2]{\begin{pmatrix} #1 \\[1ex] #2 \end{pmatrix}}
\newcommand\sess{\sigma_{\rm ess}}
\newcommand\sap{\sigma_{\rm app}}
\newcommand\restr[1]{\!\bigm|_{#1}}
\newcommand\RE{\text{\rm Re}}

\newcommand\sign{{\rm sign\,}}

\newcommand\void[1]{}

\def\ov{\overline}
\def\eps{\varepsilon}
\def\sess{\sigma_{\rm ess}}
\def\ran{{\rm ran\,}}
\def\sigp{\sigma_{\rm p}}
\DeclareMathOperator\tr{tr}

\def\sA{{\mathfrak A}}   \def\sB{{\mathfrak B}}   \def\sC{{\mathfrak C}}
\def\sD{{\mathfrak D}}   \def\sE{{\mathfrak E}}   \def\sF{{\mathfrak F}}
\def\sG{{\mathfrak G}}   \def\sH{{\mathfrak H}}   \def\sI{{\mathfrak I}}
\def\sJ{{\mathfrak J}}   \def\sK{{\mathfrak K}}   \def\sL{{\mathfrak L}}
\def\sM{{\mathfrak M}}   \def\sN{{\mathfrak N}}   \def\sO{{\mathfrak O}}
\def\sP{{\mathfrak P}}   \def\sQ{{\mathfrak Q}}   \def\sR{{\mathfrak R}}
\def\sS{{\mathfrak S}}   \def\sT{{\mathfrak T}}   \def\sU{{\mathfrak U}}
\def\sV{{\mathfrak V}}   \def\sW{{\mathfrak W}}   \def\sX{{\mathfrak X}}
\def\sY{{\mathfrak Y}}   \def\sZ{{\mathfrak Z}}

\def\frb{{\mathfrak b}}
\def\frl{{\mathfrak l}}
\def\frs{{\mathfrak s}}
\def\frt{{\mathfrak t}}

\def\dA{{\mathbb A}}   \def\dB{{\mathbb B}}   \def\dC{{\mathbb C}}
\def\dD{{\mathbb D}}   \def\dE{{\mathbb E}}   \def\dF{{\mathbb F}}
\def\dG{{\mathbb G}}   \def\dH{{\mathbb H}}   \def\dI{{\mathbb I}}
\def\dJ{{\mathbb J}}   \def\dK{{\mathbb K}}   \def\dL{{\mathbb L}}
\def\dM{{\mathbb M}}   \def\dN{{\mathbb N}}   \def\dO{{\mathbb O}}
\def\dP{{\mathbb P}}   \def\dQ{{\mathbb Q}}   \def\dR{{\mathbb R}}
\def\dS{{\mathbb S}}   \def\dT{{\mathbb T}}   \def\dU{{\mathbb U}}
\def\dV{{\mathbb V}}   \def\dW{{\mathbb W}}   \def\dX{{\mathbb X}}
\def\dY{{\mathbb Y}}   \def\dZ{{\mathbb Z}}

\def\cA{{\mathcal A}}   \def\cB{{\mathcal B}}   \def\cC{{\mathcal C}}
\def\cD{{\mathcal D}}   \def\cE{{\mathcal E}}   \def\cF{{\mathcal F}}
\def\cG{{\mathcal G}}   \def\cH{{\mathcal H}}   \def\cI{{\mathcal I}}
\def\cJ{{\mathcal J}}   \def\cK{{\mathcal K}}   \def\cL{{\mathcal L}}
\def\cM{{\mathcal M}}   \def\cN{{\mathcal N}}   \def\cO{{\mathcal O}}
\def\cP{{\mathcal P}}   \def\cQ{{\mathcal Q}}   \def\cR{{\mathcal R}}
\def\cS{{\mathcal S}}   \def\cT{{\mathcal T}}   \def\cU{{\mathcal U}}
\def\cV{{\mathcal V}}   \def\cW{{\mathcal W}}   \def\cX{{\mathcal X}}
\def\cY{{\mathcal Y}}   \def\cZ{{\mathcal Z}}

\renewcommand{\div}{\mathrm{div}\,}
\newcommand{\grad}{\mathrm{grad}\,}
\newcommand{\Tr}{\mathrm{Tr}\,}

\newcommand{\dom}{\operatorname{dom}}
\newcommand{\mes}{\operatorname{mes}}

\newcommand{\N}{\mathbb{N}}
\newcommand{\R}{{\mathbb{R}}}
\newcommand{\C}{{\mathbb{C}}}
\newcommand{\Z}{{\mathbb{Z}}}
\newcommand{\T}{{\mathcal{T}}}
\newcommand{\rmr}{\mathrm{r}}
\newcommand{\vsp}{\vspace{0.8ex}}

\def\sfG{\mathsf{G}}
\def\sfT{\mathsf{T}}
\def\sfK{\mathsf{K}}
\def\sfL{\mathsf{L}}
\def\sfGamma{\mathsf{\Gamma}}
\def\sfH{\mathsf{H}}
\def\sfE{\mathsf{E}}
\def\sfS{\mathsf{S}}
\def\sfR{\mathsf{R}}
\def\sfD{\mathsf{D}}
\def\bm1{\mathbbm{1}}

\newcommand{\ii}{{\rm i}}
\renewcommand{\P}{{\mathcal{P}}}
\newcommand{\PT}{{\mathcal{P}\mathcal{T}}}
\newcommand{\scB}{\mathscr{B}}\newcommand{\scC}{\mathscr{C}}
\newcommand{\BH}{\mathscr{B}(\mathcal{H})}
\newcommand{\BHk}[1]{\mathscr{B}(\mathcal{H}_#1)}
\newcommand{\CH}{\mathcal{C}(\mathcal{H})}
\newcommand{\CHk}[1]{\mathcal{C}(\mathcal{H}_#1)}
\renewcommand{\ker}{{\rm{ker}\,}}
\newcommand{\ulim}{\mathop{\mathrm{u}\mbox{--}\lim}}
\newcommand{\slim}{\mathop{\mathrm{s}\mbox{--}\lim}}
\newcommand{\wlim}{\mathop{\mathrm{w}\mbox{--}\lim}}
\newcommand{\mC}{\mathcal{C}}
\newcommand{\mO}{\mathcal{O}}
\renewcommand{\Re}{\text{\rm Re}\,}
\renewcommand{\Im}{\text{\rm Im}\,}
\newcommand{\diag}{{\rm diag}}

\newcommand{\dist}{\operatorname{dist}}
\newcommand{\mgp}{\marginpar}
\newcommand{\sgn}{\operatorname{sgn}}
\newcommand{\spn}{\operatorname{span}}
\newcommand{\Hf}{(\cH, (\cdot,\cdot))}

\def\SchW{\sfH^\Omega_{\ii\anot, V_0}}
\def\LapW{\sfH^\Omega_{\ii\anot, V_0}}
\def\WT{\sfH^\I_{\ii\anot}}
\def\WL{\sfH^\R_{V_0}}

\newcommand{\spp}{\sigma_{++}}
\newcommand{\smm}{\sigma_{--}}
\def\snd{\sigma_{00}}
\def\sdef{\sigma_{\rm def}}
\newcommand\OpV{\sfH_{\aa,V}^\Omega}
\newcommand\Op{\sfH_{\aa}^\Omega}
\newcommand\frmV{\frh^\Omega_{\aa, V}}

\renewcommand\thethmL{\Alph{thmL}}

\numberwithin{equation}{section}

\title{Spectra of definite type in waveguide models}

\newcommand{\petr}[1]{{\color{darkblue}#1}}

\def\PT{{\mathcal{P}\mathcal{T}}}

\author{Vladimir Lotoreichik}
\address[Vladimir Lotoreichik]{Nuclear Physics Institute CAS, 25068 \v Re\v z, Czech Republic}
\email{lotoreichik@ujf.cas.cz}

\author{Petr Siegl}
\address[Petr Siegl]{Mathematisches Institut, Universit\"at Bern, Alpeneggstr.~22, 3012 Bern, Switzerland \& On leave from Nuclear Physics Institute CAS, 25068 \v Re\v z, Czech Republic}
\email{petr.siegl@math.unibe.ch}

	
	
\date{February 28, 2016}

\begin{document}
	
%\graphicspath{{C:/Data/00Synchronized/Figures/2015/PTWGpm/}}
	
\begin{abstract}
We develop an abstract method to identify spectral points of definite type 
in the spectrum of the operator $T_1\otimes I_2 + I_1\otimes T_2$. 
The method is applicable in particular for non-self-adjoint waveguide type operators with symmetries. 
Using the remarkable properties of the spectral points of definite type, 
we obtain new results on realness of weakly coupled bound states 
and of low lying  essential spectrum in the $\PT$-symmetric waveguide.
Moreover, we show that the pseudospectrum has a normal tame behavior 
near the low lying essential spectrum 
and exclude the accumulation of non-real eigenvalues to this part of the essential spectrum. 
The advantage of our approach is particularly visible when the resolvent of 
the unperturbed operator cannot be explicitly expressed and most of 
the mentioned spectral conclusions are extremely hard to prove using direct methods.
\end{abstract}

\maketitle

%********************************************************
\section{Introduction}
%********************************************************	

Spectral points of a closed non-self-adjoint operator $T \in \CH$ in a Hilbert space $\cH$ may have 
special properties if $T$ possesses a symmetry like $J$-self-adjointness, \ie~~there is a bounded symmetric 
\emph{linear involution} $J$ such that
\begin{equation}\label{eq:Jsa}
	T = J T^* J.
\end{equation}
An isolated eigenvalue $\lambda$ of a $J$-self-adjoint 
operator $T$ is of definite type, namely, $\lambda$ is of positive (resp., negative) 
type if for all $0\neq f \in \ker(T-\lambda)$
\begin{equation}\label{EV.type}
	(Jf,f) > 0, \qquad (\text{resp., }  (Jf,f) < 0 ).
\end{equation}
If there is a neutral eigenelement, \ie~$(Jf,f)=0$ for some $0 \neq f \in \ker(T-\lambda)$, then 
$\lambda$ is sometimes called \emph{critical}, see \eg~\cite{Jonas-2003-2, Lancaster-1995-131, Langer-1997-308}. 
Eigenvalues of definite type of a $J$-self-adjoint operator are remarkable since they are \emph{real} 
and the type is \emph{stable} with respect to ``sufficiently small'' $J$-symmetric perturbations. 

The notion of spectral points of definite type is not limited to eigenvalues, 
it can be further generalized for $\lambda$ from the \emph{approximate point spectrum} $\sap(T)$, 
\cf~\cite{Azizov-2005-226,Jonas-2003-2, Lancaster-1995-131, Langer-1997-308, Philipp-2014}, 
requiring that a condition similar to~\eqref{EV.type} is satisfied for all approximate 
eigensequences for $\lambda \in \sap(T)$. Moreover, the realness of such spectral points, 
their stability with respect to perturbations and related resolvent estimates were proved 
in a series of works~\cite{Azizov-2011-83, Azizov-2005-226, Philipp-2014}, 
see also the review~\cite{Trunk-2015}.
% and the references therein.

Here we identify spectral points of definite type in tensor product type operators 
\begin{equation}\label{T.tensor}
	T_1\otimes I_2 + I_1\otimes T_2,
\end{equation}
when certain information on $T_1$ and $T_2$ is given. 
As demonstrated in an application, the remarkable properties of these spectral 
points enable us to draw spectral conclusions on perturbations of multi-dimensional 
non-self-adjoint differential operators with non-empty essential spectrum and 
without the convenient tensor product structure (since the latter is 
destroyed by usual perturbations \eg~in boundary conditions or potential). 

In the simplest setting of $\PT$-symmetric waveguide (\ie~without potentials), 
see \cite{Borisov-2008-62,Borisov-2012-76, Krejcirik-2008-41, Novak-2014}, 
we prove that the lowest part of the essential spectrum of the unperturbed waveguide 
is of positive type; \cf~Theorem \ref{thm:PTWG.pn} and 
Figure~\ref{fig:PTWG.V0}. Intuitively, the essential spectrum consists of layers $[\mu_k,\infty)$ 
of definite type, where $\mu_k$ are the eigenvalues of the transversal operator, however, 
the overlaps of layers spoil the definiteness and only $[\mu_0,\mu_1)$ remains of definite type. 
As a consequence, we receive for ``small'' or ``compact'' $\PT$-symmetric perturbations that: 
\begin{enumerate}[(i)]
	\item \label{wcb} eigenvalues emerging from the lowest threshold are real;
	\item essential spectrum in a neighborhood of $\mu_0$
	  may change, but it remains real;
	\item \label{nra} accumulation of non-real eigenvalues to real essential spectrum
      	in a neighborhood of  $\mu_0$ is excluded;
	\item \emph{pseudospectra} in a neighborhood of $\mu_0$ have a normal tame behavior;
\end{enumerate}
for precise claims see~\eqref{Msets.0} and Theorems~\ref{thm:B_ext},~\ref{thm:C_ext}. 
It appears that only~\ref{wcb} and moreover only in the simplest setting has been known, 
\cf~\cite{Borisov-2008-62,Novak-2014}. 
As our approach does not rely on an explicit knowledge of Green's function for the unperturbed operator, 
we obtain analogous conclusions without additional efforts also if rather general regular potentials 
are included (in which case the picture of definite type spectra may be much richer, 
see Figure~\ref{fig:PTWG.V}). 
Consideration of more
general differential expressions in this context is motivated by applications to curved waveguides;
\cf~\cite{Duclos-1995, Krejcirik-2008} 
for self-adjoint case. 
Finally, notice also that the stability results do not apply 
if the spectrum is not of definite type; see Remark~\ref{rem:wild}.

%******************************************************************************
\subsection{Notations and basic concepts}\label{sec:concepts}
%******************************************************************************
%
We denote by $\cH$ a Hilbert space with the scalar product $(\cdot,\cdot)$ (linear in the first entry) 
and the corresponding norm $\|\cdot\|$. $\BH$ and $\CH$ stand for bounded
(everywhere defined) and densely defined closed linear operators in $\cH$, respectively.  
We denote by $\s(T)$, $\sap(T)$ and $\rho(T)$ the spectrum, approximate point spectrum  
and resolvent set of $T\in \CH$, respectively. The \emph{$\eps$-pseudospectrum} $\s_\eps(T)$ of $T\in\CH$ 
reads
\[
	\s_\eps(T) := 
	\big\{
	\lm\in\rho(T)\colon \|(T - \lm)^{-1}\| > \eps^{-1}
	\big\}\cup\s(T).
\]
A bounded symmetric involution $J$ can be used to define a new, typically indefinite, 
inner product $[\cdot,\cdot]_J := (J \cdot, \cdot)$ in $\cH$ and a Krein space $(\cH,[\cdot,\cdot]_J)$; 
see \eg~\cite[\S I.3]{Azizov-1989}. 
A $J$-self-adjoint operator $T$, \ie~$T$ satisfying~\eqref{eq:Jsa}, is in fact a 
self-adjoint operator in $(\cH,[\cdot,\cdot]_J)$, nevertheless, 
we deliberately avoid the Krein space terminology here.

Finally, we recall the concept of spectra of definite type. 
\begin{dfn}\label{def:sdef}
	For $T\in\CH$ a point $\lambda\in\sap(T)$ is a spectral point of \emph{positive (negative)} type 
	(with respect to $J$) if every approximate eigensequence $\{f_n\}_n$ for $T$ 
	corresponding to $\lambda$ satisfies
	\begin{equation*}
	\liminf\limits_{n\rightarrow\infty} (Jf_n,f_n) > 0, 
	\qquad (\text{resp.},~ \limsup\limits_{n\rightarrow\infty} (Jf_n,f_n) < 0).
	\end{equation*}
	The set of all spectral points of $T$ of positive (negative) type is denoted by 
	$\spp(T)$ (resp., $\smm(T)$). 
	The union of spectral points of positive and negative type are 
	spectral points of definite type; the complement of the latter set 
	in the approximate point spectrum (\ie~spectral points of not definite type) 
	is denoted by  
	\begin{equation*}
	\snd(T) := \sap(T) \setminus (\spp(T) \cup \smm(T)).
	\end{equation*}
\end{dfn}

%*********************************************************
\section{Spectra of definite type and tensor products}
\label{sec:main}
%*********************************************************
We show that, based on certain information on $T_1$ and $T_2$, 
some parts in the spectrum of~\eqref{T.tensor} 
are or are not of definite type. We denote by $\otimes$ (resp., by $\odot$) tensor 
(resp., pre-tensor) products of Hilbert spaces and operators; 
see \eg~\cite[\S 2.4, 4.5, 5.7]{BEH} or \cite[\S III.7.5]{Schmuedgen-2012-265} for details.
Our basic assumption reads as follows.
\begin{hyp}\label{hyp.0}
Let $T_k \in \cC(\cH_k)$, $k=1,2$, be m-sectorial in $(\cH_k, (\cdot,\cdot)_k)$ and let
\begin{equation}\label{S.def}
S := T_1 \odot I_2 + I_1 \odot T_2, \qquad 	\sfS := \ov{S}.
\end{equation}
Moreover, let $J_k$ be bounded symmetric involutions in $\cH_k$ and let $J := J_1\otimes J_2$.
\end{hyp} 
Notice that $S$ is indeed closable and the closure $\sfS$ is m-sectorial in $\cH_1 \otimes \cH_2$ 
by \cite[\S XIII.9, Cor.~2]{Reed4}. 
In the sequel, the spectral points of positive and negative type are defined w.r.t.~$J_k$ 
for operators acting in $\cH_k$, $k=1,2$, and w.r.t.~$J := J_1\otimes J_2$ 
for operators acting in $\cH_1 \otimes \cH_2$.

The ``negative'' results, \ie~identification of spectral points of $\sfS$ that cannot be of definite type, 
are derived directly from the definition and the tensor-like structure of $\sfS$. 
To express claims in a more compact form, we define 
subsets of $\C$
\begin{equation}\label{Mi.def}
\begin{aligned}
	\M_+ & := 
	\big(\spp(T_1) + \spp(T_2)\big)
	\cup
	\big(\smm(T_1) + \smm(T_2)\big),\\
	\M_- & := 
	\big(\spp(T_1) + \smm(T_2)\big)
	\cup
	\big(\smm(T_1) + \spp(T_2)\big),\\
	\M_0 & := \big(\sap(T_1) + \sap(T_2)\big) \setminus
	(\M_+\cup\M_-).
\end{aligned}
\end{equation}
%
%Stopped!!!! 26.2.2016
%
\begin{prop}\label{prop:ndef}
	Let Assumption~\ref{hyp.0} hold and $\M_{\pm}$, $\M_{0}$ be as in~\eqref{Mi.def}. 
	Then
	\begin{equation*}
		\M_0 \cup (\M_+ \cap \M_-)\subset \snd(\sfS), 
		\quad
		\M_+ \cap \smm(\sfS) = \varnothing,
		\quad
		\M_- \cap \spp(\sfS) = \varnothing.
	\end{equation*}
\end{prop}
\begin{proof}
	We prove only the first inclusion and only for $\lambda = \lm_1 + \lm_2 \in \M_0 \cup (\M_+ \cap \M_-)$ 
	with $\lm_1 \in \snd(T_1)$ and $\lambda_2 \in \spp(T_2)$; 
	the rest is analogous. 
	Since $\lambda_1 \in \snd (T_1)$, there are two approximate eigensequences 
	$\{f_n\}_n, \{g_n\}_n \subset \dom T_1$  for $T_1$ corresponding to $\lambda_1$  
	such that $\lim_{n \to \infty} (J_1 f_n,f_n)_1 \leq 0$ and $\lim_{n \to \infty} (J_1 g_n,g_n)_1 \geq 0$.
	Similarly for $\lambda_2 \in \spp(T_2)$, there is an approximate eigensequence 
	$\{h_n\}_n \subset \dom T_2$ for $T_2$ corresponding to $\lambda_2$ such that
	$\lim_{n \to \infty} (J_2 h_n,h_n)_2 > 0$.
	It is easy to verify that $\{f_n \otimes h_n \}_n \subset \dom \sfS$ and  $\{g_n \otimes h_n \}_n \subset \dom \sfS$ are approximate eigensequences for $\sfS$ corresponding to 
	$\lambda = \lambda_1 + \lambda_2$. Thus, we get $\lm \in \snd(\sfS)$ since
	\begin{align}
		\lim_{n \to \infty} \big(J (f_n \otimes h_n),f_n \otimes h_n\big) &= 
		\lim_{n \to \infty} (J_1 f_n, f_n)_1 (J_2 h_n, h_n)_2 \leq 0,\\
		\lim_{n \to \infty} \big(J (g_n \otimes h_n),g_n \otimes h_n\big) &= 
		\lim_{n \to \infty} (J_1 g_n, g_n)_1 (J_2 h_n, h_n)_2 \geq 0. \qedhere
	\end{align}
\end{proof}
The second assumption is essential for passing to ``positive'' results.
\begin{hyp}\label{hyp}
	Let Assumption~\ref{hyp.0} hold and let there exist projections $P^\pm_k$ in $\cH_k$ and constants $\varkappa_k^\pm > 0$,  $k=1,2$,  such that:
	\begin{myenum}
		\item $P_k^\mu P_k^\nu = \delta_{\mu\nu}P_k^\mu$ for $\mu,\nu \in \{+,-\}$, $k=1,2$;
		\item $T_k P_k^\mu \supset P_k^\mu T_k$ for $\mu \in \{+,-\}$, $k=1,2$; 
		\item \label{hyp:dec.kap} for all $f_k \in \cH_k^\pm := P^\pm_k \cH_k$, we have
			\begin{equation}\label{hyp:Hk.pn}
				\pm (J_k f_k,f_k)_{k} \ge  \varkappa_k^\pm \|f_k\|^2_{k},  
				\quad k=1,2.				
			\end{equation}
	\end{myenum}
\end{hyp}
Several remarks on Assumption~\ref{hyp} are given below. 
In what follows we always take $k=1,2$ without repeating it everywhere.
The concept of a uniformly definite subspace (of a Krein space) implicitly employed
in Assumption~\ref{hyp}~(iii) is rather standard; see \eg~\cite[\S I.5]{Azizov-1989}.
Note that $\cH_{k}^\pm$ are reducing subspaces for the operators $T_k$; \cf \cite[\S III.5.6]{Kato-1966}. 
Finding projections $P_k^\pm$ is simple for isolated eigenvalues of $T_k$; 
the Riesz projections corresponding to a finite number of isolated eigenvalues of finite multiplicity 
and of the same definite type satisfy the assumption; this is used in our application, 
see Theorem~\ref{thm:PTWG.pn} and its proof. 
In a more general case, $P_k^{\pm}$ can be obtained from~\cite[Thm.~2.7]{Azizov-2012-285} 
or~\cite[Thm.~5.2]{Philipp-2014}, where the existence of a local spectral function is proved. 
In detail, if $T \in \CH$ is $J$-self-adjoint and 
$[a,b] \cap \sap(T) \subset \sigma_{\mu \mu}(T)$, $\mu \in \{+,-\}$, 
then for all bounded subintervals $\Delta$ of $(a,b)$ with $\overline \Delta \subset (a,b)$, 
there is a $J$-self-adjoint projection $E(\Delta)$ such that 
$TE(\Delta) \supset E(\Delta)T$ and $\pm (J f,f) \geq \varkappa^{\pm} \|f\|^2$ 
for all $f \in E(\Delta) \cH$ with some $\varkappa^{\mu} > 0$. 

The idea of Assumption~\ref{hyp} is to extract ``$++$'' and ``$--$'' parts of $T_k$. 
Nevertheless, we emphasize that we do not need to have spectral projections on the \emph{entire} 
$\spp(T_k)$ or $\smm(T_k)$ (which allows for avoiding to require 
\eg~a Riesz basis property of eigenvectors for $T_k$ with purely discrete spectrum).

With $P_k^\mu$ from Assumption~\ref{hyp}, define the projections and subspaces by
\begin{equation}\label{Prmr}
	P_k^\rmr := I_\cH - P_k^+ - P_k^-\quad \text{and} \quad \cH_k^\rmr := P_k^\rmr\cH, \quad k=1,2.
\end{equation}
One can verify that the two families of projections $\{P_k^\mu, P_k^-, P_k^\rmr\}$ satisfy 
\begin{equation}\label{Pk.com.rel}
	P_k^\mu P_k^\nu = \delta_{\mu\nu}P_k^\mu, \quad \mu, \nu \in \cI:=\{+,-,\rmr \}, \quad k=1,2,
\end{equation}
and we get the direct sum decompositions of $\cH_k$;  
$\cH_k = \cH_k^+\dotplus \cH_k^- \dotplus \cH_k^\rmr$, \cf~\cite[\S III.5.6]{Kato-1966}.
By Assumption~\ref{hyp}\,(i)-(ii) the operators
\begin{equation}\label{Tmu}
	T_k^\mu u  := T_k u  \quad \dom T_k^\mu := \dom T_k\cap \cH_k^\mu,\qquad \mu \in \cI, \quad k=1,2,
\end{equation}
are closed and following~\cite[\S III.5.6]{Kato-1966} we end up with the direct sum decomposition 
\begin{equation}\label{Tdecomp}
	T_k = T_k^+\dotplus T_k^- \dotplus T_k^{\rmr}, \quad k=1,2.
\end{equation}

To explain next steps we first prove a simple technical lemma.
\begin{lem}\label{lem:Theta}
	Let $\{P_k\}_{k=1}^n$, $n \in \N$, be a family of projections in $\cH$ such that 
	$P_i P_j = \delta_{ij} P_i$ for $i,j \in\{1,2,\dots,n\}$
	and $\sum_{k=1}^n P_k = I_{\cH}$. Define the operator 
	\begin{equation}\label{eq:Theta}
		\Theta := \sum_{k=1}^n P_k^*P_k.
	\end{equation}
	Then $\Theta \in \BH$, it is uniformly positive and satisfies the commutation relation
	\begin{equation}\label{P.Theta}
		\Theta P_j = P_j^* \Theta, \qquad j = 1,2,\dots,n.
	\end{equation}
\end{lem}
\begin{proof}
	Clearly, $\Theta \in \BH$. Since $f = \sum_{k=1}^n P_k f$, we get
	\begin{equation}\label{Th.ineq}
		(\Theta f,f) = \sum_{k=1}^n \|P_kf\|^2 \ge \frac{1}{n} \|f\|^2,
	\end{equation}
	where we use Cauchy-Schwarz inequality in the second step, hence, $\Theta$ is uniformly positive. 
	Moreover, using $P_i P_j = \delta_{ij} P_i$ and $P_i^*P_j^* = \delta_{ij} P_i^*$, 
	we obtain~\eqref{P.Theta}.
\end{proof}
Using Lemma~\ref{lem:Theta}, we construct uniformly positive bounded operators 
\begin{equation}
	\Theta_k := \sum_{\mu \in\cI}(P_k^\mu)^*P_k^\mu, \quad k = 1,2.
\end{equation}
Hence, the scalar products
%
%\begin{equation*}\label{eq:scalarTheta}
$(\cdot,\cdot)_{\Theta_k} := (\Theta_k \cdot,\cdot)_k$
%\end{equation*}
%
are well-defined on $\cH_k$ and topologically equivalent to $(\cdot,\cdot)_k$. 
Due to commutation relations~\eqref{Pk.com.rel}, the subspaces $\cH_k^\mu$, $\mu \in \cI$,
are mutually orthogonal w.r.t.~$(\cdot,\cdot)_{\Theta_k}$.
Thus, the decompositions~\eqref{Tdecomp} 
are orthogonal in $(\cdot,\cdot)_{\Theta_k}$, \ie~
$T = T^+ \oplus_{\Theta_k} T^- \oplus_{\Theta_k} T^\rmr$, and 
$\s(T_k) = \s(T_k^+)\cup\s(T_k^-)\cup\s(T_k^\rmr)$.

Consider now $\cH := \cH_1 \otimes \cH_2$ endowed with the natural scalar product
$(\cdot,\cdot)_\cH$ induced from the scalar products $(\cdot,\cdot)_k$. 
By Lemma~\ref{lem:Theta} and~\cite[Thm.~5.7.4]{BEH}
\begin{equation}\label{Phi}
	\Phi := \Theta_1\otimes\Theta_2
\end{equation}
is bounded and uniformly positive in $\cH$. 
Thus, the scalar product 
\begin{equation}\label{scalar_Phi}
	(\cdot,\cdot)_\Phi := (\Phi\cdot,\cdot)
\end{equation}
is well-defined on $\cH$ and topologically equivalent to the initial product $(\cdot,\cdot)$. 

Define the subspaces $\cH^{\mu\nu} := \cH^\mu_1 \otimes \cH^\nu_2$, $\mu,\nu\in \cI$ 
and observe that $\cH = \oplus^{\mu,\nu\in \cI}_\Phi\cH^{\mu\nu}$.
We introduce operators (the closures are m-sectorial in respective Hilbert spaces)
\begin{equation}\label{ClosedOps}
	S^{\mu \nu}  := T_1^\mu \odot I_2^\nu + I_1^\mu \odot T_2^\nu, \quad	
	\sfS^{\mu\nu} := \ov{S^{\mu\nu}},\qquad \mu,\nu\in\cI,
\end{equation}
and also some specific orthogonal sums
\begin{equation}
\label{Smu}
\begin{split}
	\sfS^\rmr & := \sfS^{\rmr +} \oplus_\Phi \sfS^{\rmr -} \oplus_\Phi \sfS^{\rmr\rmr}
				\oplus_\Phi \sfS^{+\rmr} \oplus_\Phi \sfS^{-\rmr},\\
	\sfS^+  & := \sfS^{++} \oplus_\Phi \sfS^{--},\\
	\sfS^-  & := \sfS^{+-} \oplus_\Phi \sfS^{-+}.
\end{split}	
\end{equation}
Next, using the tensor product rules (see \eg~\cite[Prop.~4.5.6]{BEH}) 
we obtain that $S$ in~\eqref{S.def} can be decomposed as
$S = \oplus_\Phi^{\mu,\nu\in\cI} S^{\mu\nu}$ and taking the closures we get
\begin{equation}\label{Cl.S.dec.1}
	  \sfS = \ov{S} =  \oplus_\Phi^{\mu,\nu\in\cI} \sfS^{\mu\nu}.
\end{equation}
From Assumption~\ref{hyp} and with $J=J_1\otimes J_2$, we verify that for $\mu\in\{+,-\}$ 
\begin{equation}
\label{eq:J.1}
	(Jf,f)  \ge \varkappa_1^\mu\varkappa_2^\mu \|f\|^2,
	\qquad \text{for all}~f\in\cH^{\mu\mu},\\
\end{equation}
and that for $\mu, \nu\in\{+,-\}$, $\mu\ne \nu$,
\begin{equation}
\label{eq:J.2}
	(Jf,f) \le -\varkappa_1^\mu\varkappa_2^\nu \|f\|^2,
	\qquad \text{for all}~f\in\cH^{\mu\nu}.
\end{equation}

The following theorem is the key result of this section.
We get a slightly sharper result in~item (iii) 
since also the mutual collocations of subspaces $\cH_k^\pm$ are used.
%The constants $\varkappa^{+-}_k$ introduced in~\eqref{cond_item_iii} are
%related to the concept of \emph{the angle between two subspaces} of a Hilbert space \cite{D95}.

%-----------------------------------------------------------------------------
\begin{thm}\label{thm:abstract}
	Let Assumption~\ref{hyp} hold and let the m-sectorial operators $\sfS^{\mu \nu}$, $\mu,\nu\in\cI$, 
	$\sfS^\mu$, $\mu\in\cI$ and $\sfS$ be as in~\eqref{ClosedOps},~\eqref{Smu} 
	and~\eqref{Cl.S.dec.1}, respectively.
	Let $\spp(\sfS)$ and $\smm(\sfS)$
	be defined w.r.t.~$J=J_1 \otimes  J_2$. 
	Then for $\mu,\nu\in\{+,-\}$, $\mu\ne \nu$, the following statements hold.
	\begin{enumerate}[{\upshape (i)}]
		\item $\sap(\sfS^{\mu\mu}) \setminus 
		     	\big(\sigma(\sfS^-)\cup\sigma(\sfS^\rmr)\cup\sigma(\sfS^{\nu\nu})\big) 
		   		 \subset \sigma_{++}(\sfS)$.

		\item $\sap(\sfS^{\mu\nu})\setminus 
				\big(\sigma(\sfS^+) \cup \sigma(\sfS^\rmr)\cup \sigma(\sfS^{\nu\mu}) \big)
				\subset \sigma_{--}(\sfS)$.

		\item If, in addition, there exist constants 
			$\varkappa^{+-}_k \ge 0$, for $k=1,2$,  such that 
			\begin{equation}
			\label{cond_item_iii}			
				|(J_k f_k^+,f_k^-)_k| \le \varkappa_k^{+-} \|f_k^+\|_k \|f_k^-\|_k,
						\qquad \text{for all}~f_k^\pm\in \cH_k^\pm,
			\end{equation}
			and, moreover, that 
			$(\varkappa_1^{+-} \varkappa_2^{+-})^2 < 
			\varkappa_1^+ \varkappa_2^+\varkappa_1^- \varkappa_2^-$.
			Then 
			\begin{equation}
			\label{sp.incl.2}
				\sap(\sfS^\mu)
				\setminus 
				\big(\sigma(\sfS^\nu) \cup \sigma(\sfS^\rmr) \big) 
				\subset \sigma_{\mu\mu}(\sfS).
			\end{equation}
	\end{enumerate}
\end{thm}
%-----------------------------------------------------------------------------

%
\begin{proof}
	Let $\Phi$ and $(\cdot,\cdot)_\Phi$ be as in~\eqref{Phi} and~\eqref{scalar_Phi}, 		
	respectively. 

%	\vsp
	
	\noindent (i)
	We show the first inclusion ($\mu = +$, $\nu = -$) only, 
	the rest is fully analogous. 
	The decomposition~\eqref{Cl.S.dec.1} can be restructured as
	\begin{equation}\label{Cl.S.dec.2}
		\sfS =  \sfS^{++} \oplus_\Phi \sfS^0,
	\end{equation}
	where $\sfS^0:= \sfS^{--} \oplus_\Phi \sfS^- \oplus_\Phi \sfS^\rmr$. 
	Let $\lambda \in \sap(\sfS^{++}) \setminus \sigma(\sfS^0)\subset \sap(\sfS)$. 
	Pick an arbitrary approximate eigensequence 
	$\{f_n\}_n$ for the operator $\sfS$ corresponding to $\lambda$. 
	Each element $f_n$ can be decomposed uniquely as
	$f_n = f_n^{++} + f_n^0$ with $f_n^{++} \in\cH^{++}$ and
	$f_n^0 \,\, \bot_\Phi \,\, \cH^{++}$.
	By the equivalence of the norms 
	$\|\cdot\|$ and $\|\cdot\|_\Phi$, where $\|f\|_\Phi^2:=(f,f)_\Phi$, we get 
	\begin{equation}\label{Cl.S.theta.sing}
		\|(\sfS - \lambda) f_n\|_\Phi \le M \|(\sfS  - \lambda) f_n\| \to 0
	\end{equation}
	with some $M > 0$.
	Furthermore, the orthogonal decomposition~\eqref{Cl.S.dec.2} gives
	\begin{equation}\label{Cl.S.fn.dec}
	\begin{aligned}
		\|(\sfS - \lambda) f_n\|_\Phi^2 
		&= 
		\|(\sfS^{++}  - \lambda) f_n^{++}\|_\Phi^2 
		+ 
		\|(\sfS^0  - \lambda) f_n^0\|_\Phi^2.
	\end{aligned}
	\end{equation}
	Since $\lambda \in \rho(\sfS^0)$, there exists $k_\lambda>0$ such that
	$
		\|(\sfS^0  - \lambda) f_n^0\|\ge k_\lambda \|f_n^0\|.
	$
	Thus, using equivalence of the norms ($\|\cdot\|$ and $\|\cdot\|_\Phi$),~\eqref{Cl.S.theta.sing}, 
	and~\eqref{Cl.S.fn.dec}, 
	we conclude that $\|f_n^0\| \to 0$ and hence $\|f_n^{++}\| \to 1$. 
	Finally, the former and~\eqref{eq:J.1} (with $\mu = +$) 
	imply that $\lambda \in \sigma_{++}(\sfS)$ since
	\[
		\liminf_{n \to \infty} (J f_n, f_n) 
		= 
		\liminf_{n \to \infty} (J f_n^{++}, f_n^{++}) 
		\geq 
		\liminf_{n \to \infty} \varkappa_1^+ \varkappa_2^+ \|f_n^{++}\|^2 
		= 
		\varkappa_1^+\varkappa_2^+ >0.
	\]
	%
	%
%	\vsp	
	
	\noindent (ii)
	The proof is analogous to the one of item~(i). 
	In the case $\mu = +$ and $\nu = -$ the decomposition~\eqref{Cl.S.dec.1} 
	is restructured as
	%
%	\begin{equation*}\label{Cl.S.dec.3}
	$\sfS = \sfS^{+-} \oplus_\Phi \sfS^0$
%	\end{equation*}
	%
	where $\sfS^0 := \sfS^{-+}\oplus_\Phi \sfS^+\oplus_\Phi \sfS^\rmr$. The case $\mu = -$ and
	$\nu = +$ is analogous.
	
	\noindent (iii)
	We follow the lines of the proofs of items~(i) and~(ii). 
	In the case $\mu = +$ and $\nu = -$ 
	the decomposition in~\eqref{Cl.S.dec.1} is restructured as
	%
%	\begin{equation*}\label{Cl.S.OGdec.4}
	$\sfS = \sfS^+ \oplus_\Phi \sfS^0$ 
%	\end{equation*}
	%
	where $\sfS^0 = \sfS^- \oplus_\Phi \sfS^\rmr$. 
	A straightforward adaptation of the above arguments shows that 
	any approximate eigensequence $\{f_n\}_n$ for the operator $\sfS$
	corresponding to a point $\lambda \in \sap(\sfS^+)\setminus\sigma(\sfS^0)$
	can be uniquely decomposed as 
	$f_n = f_n^{++} + f_n^{--} + f_n^0$ with
	$f_n^{++}\in\cH^{++}$, $f_n^{--}\in\cH^{--}$
	and $f_n^0 \,\, \bot_\Phi \,\, (\cH^{++}\oplus_\Phi \cH^{--})$.
	Analogously, we get $\|f_n^0\| \to 0$ and  $\|f_n^{++} + f_n^{--}\| \to 1$.
	Similarly to~\eqref{eq:J.1} and~\eqref{eq:J.2} one can derive 
	\[
		|(Jf^{++},f^{--})| \le 
		\varkappa_1^{+-} \varkappa_2^{+-} \| f^{++} \| \| f^{--} \|,
		\qquad f^{++}\in\cH^{++},~f^{--}\in\cH^{--}. 
	\]	
	Thanks to the condition
	$(\varkappa_1^{+-} \varkappa_2^{+-})^2 < 
	\varkappa_1^+ \varkappa_2^+\varkappa_1^- \varkappa_2^-$,
	for all sufficiently small $\delta > 0$, we have 
	$\varkappa(\delta) := \big[
		(\varkappa_1^+ \varkappa_2^+ - \delta)
		(\varkappa_1^- \varkappa_2^- -\delta)\big]^{1/2} - 
		\varkappa_1^{+-} \varkappa_2^{+-} > 0$ 
	and, for such $\delta > 0$, 
	we then obtain, using~\eqref{eq:J.1} and Cauchy-Schwarz inequality in addition, that for $f_n^+ = f_n^{++} + f_n^{--}$ holds
	\begin{equation*}
	\begin{aligned}
		(J f_n^+, f_n^+) & = 
		(J(f_n^{++} + f_n^{--}), f_n^{++} + f_n^{--})\\
		&\ge 
		\varkappa_1^+ \varkappa_2^+ \|f_n^{++}\|^2 + 
		\varkappa_1^- \varkappa_2^- \|f_n^{--}\|^2 
		- 2|(Jf_n^{++}, f_n^{--})|\\ 
		& \ge 
		\varkappa_1^+ \varkappa_2^+\|f_n^{++}\|^2 +  
		\varkappa_1^- \varkappa_2^- \|f_n^{--}\|^2
		-2\varkappa_1^{+-}\varkappa_2^{+-}\|f_n^{++}\|\|f_n^{--}\|,\\
		&\ge 2\varkappa(\delta)\|f_n^{++}\|\|f_n^{--}\|  
		+ \delta \big(\|f_n^{++}\|^2 + \|f_n^{--}\|^2\big)\\
		& \ge
		\delta\big(\|f_n^{++}\|^2 + \|f_n^{--}\|^2\big)
		\ge
		\frac{\delta}{2} \|f_n^{++} + f_n^{--}\|^2 = 
		\frac{\delta}{2} \|f_n^+\|^2.
	\end{aligned}
	\end{equation*}
	Hence, using $\|f_n^0\| \rightarrow 0$ and $\|f_n^+\| \rightarrow 1$ we obtain
	\[
		\liminf_{n\rightarrow\infty} (J f_n, f_n) 
		= 
		\liminf_{n\rightarrow\infty} (J f_n^{+}, 
		f_n^{+})\\
		 \ge 
		\frac{\delta}{2}
		\liminf_{n\rightarrow\infty} 
		\|f_n^+\|^2
		=
		\frac{\delta}{2} > 0.	
	\]
	Thus $\lambda\in\spp(\sfS)$ and the first inclusion in~\eqref{sp.incl.2} is proven; 
	the rest is analogous. 
\end{proof}

%%-------------------------------------------------------------------------
\section{Application: $\PT$-symmetric waveguides}
\label{sec:appl}
%-------------------------------------------------------------------------
%
We investigate the two-dimensional $\PT$-symmetric waveguide suggested 
in~\cite{Borisov-2008-62} and studied further 
in~\cite{Borisov-2012-76, Krejcirik-2008-41, Novak-2014}. 
We view the waveguide as a perturbation of a certain ``unperturbed'' 
operator defined in Subsection~\ref{ssec:unperturb_def}, 
having the tensor product structure~\eqref{T.tensor}. 
Based on our abstract machinery from Section~\ref{sec:main} and known properties 
of the transversal one-dimensional operator given in Subsection~\ref{ssec:transversal}, 
we fully characterize the spectra of definite type for this unperturbed operator in 
Theorem~\ref{thm:PTWG.pn}. In the proof we deliberately avoid using the Riesz basis property 
of the eigensystem of the transversal operator $\WT$, which is typically not available 
in other similar problems. Finally, in Subsection~\ref{ssec:peturb}
we apply perturbation results to reveal intervals of definite type spectra
for the original (perturbed) operator, not having 
the convenient tensor product structure, and employ
properties of definite type spectra to draw spectral conclusions on the original operator.

%****************************************************************************	
\subsection{Definition of the waveguide}\label{sec:PT_results}
%****************************************************************************
%
Let $\I := (-a,a)$, $a \in (0, \infty)$, $\Omega := \R \times \I$ and
$\Sigma_\pm := \R \times \{\pm a\}$; the latter are the opposite sides of the strip $\Omega$. 
We denote the inner product in $L^2(\Sigma_\pm)$ by $(\cdot,\cdot)_{\Sigma_\pm}$
and in both $L^2(\Omega)$ and $L^2(\Omega;\dC^2)$ by $(\cdot,\cdot)_{\Omega}$.

Let $J_1 : = I_{L^2(\R)}$ and $J_2 := \P$ where 	%
\begin{equation}\label{P.def}
(\P\psi)(y) := \psi(-y), \qquad \psi \in L^2(\I).
\end{equation}
Both $J_k$, $k=1, 2$, are bounded symmetric involutions; 
$J_1$ is uniformly positive and $J_2$ is indefinite. 
The tensor product $J= J_1 \otimes J_2$ is easily seen to act as
\begin{equation}\label{def:JWG}
(J u) (x, y)= u(x,-y), 
\quad u \in L^2(\Omega).
\end{equation}	
The complex conjugation operator on any of the used functional spaces is denoted by $\T \psi := \overline{\psi}$.
Further, we introduce the space of bounded $\PT$-symmetric functions
\begin{equation*}\label{eq:Linfty_odd}
L^\infty_{\rm \PT}(\Omega) := \big\{V\in L^\infty(\Omega;\C) \colon V(x,y) = \ov{V(x,-y)}\big\}.
\end{equation*}
\begin{dfn}\label{def:Op}
Let $V \in L^\infty_{\rm \PT}(\Omega)$ and $\aa \in L^\infty(\R;\dC)$. 
The m-sectorial operator $\OpV$ in $L^2(\Omega)$ associated to the densely defined, closed, sectorial form 
\begin{equation}\label{eq:PTWG.form}
	H^1(\Omega) \mapsto 
	\|\nabla u\|_\Omega^2 +  (Vu,u)_\Omega
	+ 
	(\aa  u|_{\Sigma_+}, u|_{\Sigma_+})_{\Sigma_+}
	+ 
	(\ov{\aa}  u|_{\Sigma_-}, u|_{\Sigma_-})_{\Sigma_-},
\end{equation}
\cf~\cite[Thm.~VI.2.1]{Kato-1966} and \cite{Borisov-2008-62}, represents the $\PT$-symmetric waveguide with the coupling function $\aa$
and the potential $V$. If $V \equiv 0$, we write $\Op$ instead of 
$\sfH_{\aa,0}^\Omega$.
\end{dfn}

\begin{remark}\label{rem:WG.sym}
It can be verified by the first representation theorem, 
\cf~\cite[Thm.~VI.2.1]{Kato-1966}, that $\OpV$ is 
$\PT$-symmetric w.r.t parity 
$(\P u)(x,y) := u(x,-y)$, 
\ie~$\PT (\OpV) \subseteq (\OpV) \PT$ or, using formally the commutator, $[\PT, \OpV] = 0$. 
Moreover, $\OpV$ is also $\P$-self-adjoint and $\T$-self-adjoint, \ie~$\OpV = \T (\OpV)^*\T$.
\end{remark}

\subsection{The unperturbed operator}
\label{ssec:unperturb_def}

Let $V_0 \in L^\infty(\R;\R)$ and $\aa\in \dC$ be fixed.
In what follows we denote the functions $\dR\ni x \mapsto \aa$ and
$\Omega\ni (x,y) \mapsto V_0(x)$  again by $\aa$ and $V_0$.
Define the self-adjoint operator $\WL$ in $L^2(\R)$ 
and the m-sectorial operator $\sfH^\I_{\aa}$ 
in $L^2(\I)$ (\cf~\cite{Borisov-2008-62})
as:
\begin{align}
	\WL\psi   & := -\psi^\dpp + V_0\psi,& \dom\WL & := H^2(\dR); 
	\label{DeltaR} \\
	\sfH_{\aa}^\I\psi  & := -\psi^\dpp, & 
	\dom\sfH_{\aa}^\I & := 
		\big\{ \psi \in H^2(\I) \colon 
			\psi^\pp(\pm a) = - \aa \psi( \pm a)  
		\big\}.
	\label{DeltaM}
\end{align}
Henceforth, $\WL$ and $\sfH_{\aa}^\I$
are called the \emph{longitudinal} and the \emph{transversal}
operators, respectively. Introduce the sectorial operator
$S_{\aa,V_0} := \WL \odot I_{L^2(\I)}	+ I_{L^2(\R)} \odot \sfH_\aa^\I$
acting in $L^2(\Omega) = L^2(\R) \otimes L^2(\I)$, which is closable and its closure
\begin{equation}\label{def:OpWG}
	\sfH_{\aa,V_0}^\Omega := \ov{ S_{\aa,V_0} }
\end{equation}
is m-sectorial in $L^2(\Omega)$, see  \cite[\S XIII.9, Cor.~2]{Reed4}. 
It can be shown by standard arguments that this new definition of $\sfH_{\aa,V_0}^\Omega$
coincides with Definition~\ref{def:Op} in the special case
of constant coupling function and potential dependent only 
on $x$-variable.

Except for Remark~\ref{rem:wild}, we restrict our analysis to the case $\aa = \ii\aa_0$
with $\aa_0\in\dR$. 
	
%-------------------------------------------------------------------------
\subsection{The transversal operator}
%-------------------------------------------------------------------------
\label{ssec:transversal}
	
First, we collect known results.
	
%-------------------------------------------------------------------------
\begin{prop}[{\cite{Krejcirik-2006-39}}]\label{prop:PTWG.trans.1}
	Let $\sfH^\I_{\ii\anot}$, $\anot\in\dR$, be as in~\eqref{DeltaM} 
	and $\P$ as in~\eqref{P.def}. Then 
	\begin{enumerate}[{\upshape (i)}]
		\item we have $$(\WT)^* = \sfH^\I_{-\ii\anot}, \quad \WT = \P (\WT)^* \P, \quad \PT (\WT) \subseteq (\WT) \PT;$$
		\item
		$\s(\WT) = \cup_{n \in \N_0}\{ \lm_n \} \subset \R$, 	
		where $\lm_0 = \anot^2$ and  
		$\lm_n = \left(\frac{\pi n}{2a}\right)^2$, $n\in\dN$;
		\item if $\pm \frac{2a}{\pi }\alpha_0 \notin \N$, 
		then all the eigenvalues are simple; otherwise 
		 $\lm_0$ has the geometric multiplicity one 
		and the algebraic multiplicity two and all the other eigenvalues are simple.
	\end{enumerate}	
\end{prop}

Next, we classify the definiteness of eigenvalues
of $\WT$, see also Figure~\ref{fig:T1.pm}.
%
%-------------------------------------------------------------------------
\begin{prop}\label{prop:++}
	Let $\WT$ and $\{ \lm_n \}_{n \in \N_0}$ be as in~\eqref{DeltaM} and
	in Proposition~\ref{prop:PTWG.trans.1}\,(ii), respectively.  
	Let $\{\mu_n\}_{n \in \N_0}$ be eigenvalues $\{ \lm_n \}_{n \in \N_0}$ 
	ordered in non-decreasing order (with  algebraic multiplicities). Define the set 
	\begin{equation*}
	\cE(\alpha_0):=
	\begin{cases}
	\varnothing & \text{if } \quad \alpha_0^2 \notin \{\lambda_n\}_{n \in \N},
	\\  \{n_*-1,n_*\} & \text{if } \quad \alpha_0^2 = \lambda_{n_*} \text{ for some } n_* \in \N.
    \end{cases}
    \end{equation*}		
	Then, with respect to $J_2 = \P$ in~\eqref{P.def},  
	\begin{equation*}
	\begin{aligned}
		\spp(\WT) &= \{\mu_{2n}\colon n \in \N_0, 2n \notin \cE(\alpha_0) \}, 
	\\\ 
		\smm(\WT) &= \{\mu_{2n+1}\colon n \in \N_0, 2n+1 \notin \cE(\alpha_0)\},
	\\                                       
		\snd(\WT) &= \{\mu_n\colon n \in \cE(\alpha_0)\}.
	\end{aligned}
	\end{equation*}	
\end{prop}
%-------------------------------------------------------------------------
\begin{figure}[htb!]
	\includegraphics[width=0.48 \textwidth]{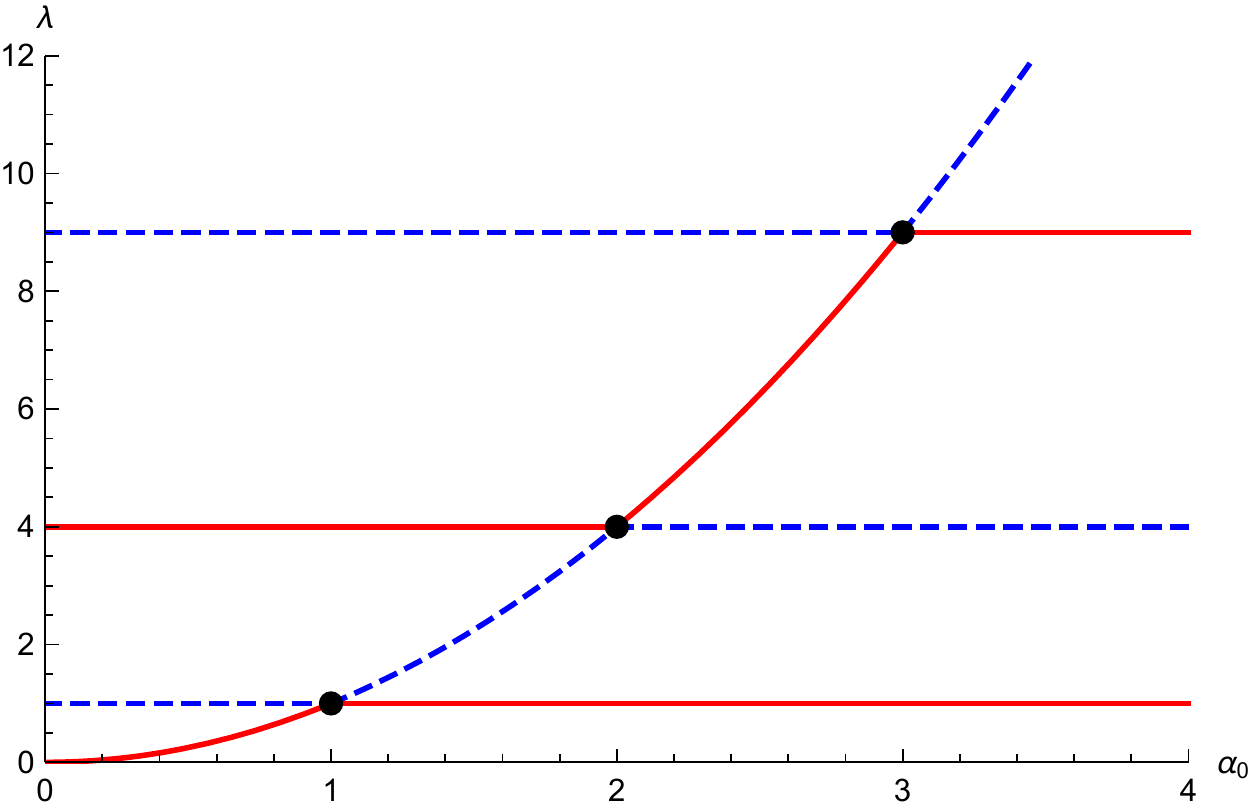}
	\caption{Lowest eigenvalues of $\WT$ as a function of $\aa_0 \in\dR_+$ 
		for $a=\pi/2$. 
		The red (full) curves correspond to $\spp(\WT)$, the blue (dashed) curves 
		to $\smm(\WT)$. The spectral points of not definite type (black balls) appear 
		for exceptional values $\aa_0= 1, 2, 3,\dots$ only. }
	\label{fig:T1.pm}
\end{figure}
\begin{proof}
The eigenfunctions corresponding to $\alpha^2_0$ and $\{\lambda_n\}_{n \in \dN}$ read, \cf~\cite[Prop.2.4]{Krejcirik-2014-8},
\begin{equation*}
\label{efunc}
\psi_0(x) = \ee^{-\ii \anot(x+a)},
\ \
\psi_n(x) = \cos (\sqrt{\lm_n} (x+a) ) - 
\frac{\ii \anot}{\sqrt{\lm_n}} \sin (\sqrt{\lm_n} (x+a)), \quad n \in \N.
\end{equation*}
The claims follow from the direct computations 
	\begin{equation*}\label{P.psi0}
		(\P \psi_0,\psi_0)_\I = 
		\frac{\sin(2\alpha_0a)}{\alpha_0},
		\qquad 
		(\P \psi_n,\psi_n)_\I  
		= a\frac{(-1)^n(\lm_n - \lm_0)}{\lm_n}, \quad n \in \N. \qedhere
	\end{equation*}	
\end{proof}
%

%-------------------------------------------------------------------------
\subsection{Spectra of definite type of the unperturbed waveguide}
%-------------------------------------------------------------------------
\label{ssec:unperturb}
Slightly extending \cite[Prop.~4.2, Rem.~4.2]{Borisov-2008-62}, one can straightforwardly check the following.
%
%------------------------------------------------------------------
\begin{lem}\label{lem:PTWG.sp}
	Let $\aa_0\in\R$, $V_0\in L^\infty(\R;\R)$ and $\WL$, $\WT$ and $\sfH^\Omega_{\ii\anot, V_0}$ be 
	as in~\eqref{DeltaR},~\eqref{DeltaM} and~\eqref{def:OpWG},
	respectively. Then $\sfH^\Omega_{\ii\anot, V_0} = J (\sfH^\Omega_{\ii\anot, V_0})^* J$  
	with $J$ as in~\eqref{def:JWG}. Moreover, we have 
	$$\s(\sfH^{\Omega}_{\ii\anot, V_0}) 	= 
			\s(\WT) + \s(\WL).$$
		In particular, for $V_0 \equiv 0$, 
		$\s(\sfH^{\Omega}_{\ii\anot}) = [\mu_0,\infty)$,
		where $\mu_0 = \min\s(\WT)$.
\end{lem}
%-------------------------------------------
%
Notice that for the operators $\WT$ and $\WL$ the sets defined in~\eqref{Mi.def} 
read as
\begin{equation}\label{eq:Mmu}
	\M_\mu = \s_{\mu\mu}(\WT) + \s(\WL),\quad \mu\in\{+,-,0\},
	\qquad
	\M := \M_+\cup\M_- \cup \M_0.
\end{equation}
The definiteness of spectral points of $\sfH_{\ii\aa_0,V_0}^\Omega$ 
can be then characterized completely. 
\begin{thm}\label{thm:PTWG.pn}
	Let $\aa_0\in\R$, $V_0\in L^\infty(\R;\R)$, $\sfH_{\ii\aa_0,V_0}^\Omega$ 
	be as in~\eqref{def:OpWG}, $\M_\mu$, $\mu \in \{+,-,0\}$ be as in~\eqref{eq:Mmu} 
	and $\{\mu_n\}_{n \in \N_0}$ be as in Proposition~\ref{prop:++}. Then,
	\begin{align*}
		\spp(\sfH_{\ii\aa_0,V_0}^\Omega) &= \M_+\setminus (\M_-\cup \M_0), \!\quad\!
		\smm(\sfH_{\ii\aa_0,V_0}^\Omega) = \M_-\setminus(\M_+\cup \M_0),
		\\
		\snd(\sfH_{\ii\aa_0,V_0}^\Omega) &= \M_0 \cup (\M_+ \cap \M_-),
	\end{align*}
	with respect to $J$ as in~\eqref{def:JWG}. 
	Thus, for $\sfH_{\ii\aa_0}^\Omega$ (\ie~$V_0\equiv 0$) in particular, we have
	\begin{equation*}\label{M.sets.V0}
		\spp(\sfH_{\ii\aa_0}^\Omega) = [\mu_0,\mu_1), \quad 
		\smm(\sfH_{\ii\aa_0}^\Omega) = \varnothing, 
		\quad 
		\snd(\sfH_{\ii\aa_0}^\Omega)= [\mu_1, \infty).
	\end{equation*}
\end{thm}
\begin{proof}
	In what follows, let $\cH_1 = L^2(\R)$, $\cH_2 = L^2(\I)$,
	$T_1 := \sfH_{V_0}^\R$, $T_2 := \WT$,
	$\sfS := \ov{T_1 \odot I_2 + I_1 \odot T_2}$,
 	$J_1 := I_1$ and $J_2 := \P$ with $\WL$, $\WT$, $\P$ as 
 	in~\eqref{DeltaR},~\eqref{DeltaM} 
 	and~\eqref{P.def}, respectively.
 	First, observe that $\M = \s(\sfS) = \sap(\sfS)$. 
 	Moreover, by Proposition~\ref{prop:ndef}, 
 	$\M_0 \cup (\M_+ \cap \M_-) \subset \snd(\sfS)$ 
 	and $\s_{\pm\pm}(\sfS) \subset \M_\pm \setminus (\M_\mp \cup \M_0)$. 
 	The opposite inclusions for the latter are shown below.
 	
	Define the projections $P_1^+ := I_1$ and $P_1^- := 0$ in $\cH_1$; 
	notice that $P_1^{\rm r} =0$ as well. 
	Obviously, $\cH_1, T_1, J_1, P^\pm_1$ 
	and $\cH^\pm_1:= P_1^{\pm} \cH_1$ satisfy Assumption~\ref{hyp} 
	with $\varkappa_1^+ = 1$. 

	Now we decompose $T_2$. We order the eigenvalues of $T_2$ 
	of positive and negative type, see Proposition~\ref{prop:++},  
	in the increasing order
	%	
%	\begin{equation*}
		$\sigma_{\pm \pm}(T_2) = \{\mu^\pm_n\}_{n \in \N_0}$
%	\end{equation*}
	%	
	and denote by $\{Q_n^{\pm}\}_{n \in \N_0}$ the corresponding Riesz (spectral) projections.
	Let $N \in \N_0$ be arbitrary and define  
	$P_2^{\pm}(N) := \sum_{n=0}^N Q^\pm_n$, $P_2^{\rm r}(N)   := I_2 - P_2^+(N) - P_2^-(N)$.  

	For every $N \in \N_0$, the family $\cH_2, T_2, J_2, P^\pm_2(N)$
	and $\cH^\pm_2(N):= P^\pm_2(N) \cH_2$ satisfies Assumption~\ref{hyp} 
	if we verify \eqref{hyp:Hk.pn}. 
	To this end, observe that, 
	for $\{Q_0^\pm, \dots, Q_N^\pm, I_2 - P^{\pm}_2(N)\}$, 
	Lemma~\ref{lem:Theta} yields operators $\Theta^\pm(N)$ 
	such that we have the orthogonal decompositions 
	\begin{equation*}
	\begin{split}
		\cH_2 & = 
		P_2^\pm(N) \cH_2 \oplus_{\Theta^\pm(N)}(I_2 - P_2^{\pm}(N)) \cH_2,\\
		\cH^\pm_2(N)  & = 
		Q_0^{\pm} \cH_2 \oplus_{\Theta^{\pm}(N)} Q_1^{\pm} \cH_2 \oplus_{\Theta^{\pm}(N)}
		\dots \oplus_{\Theta^{\pm}(N)} 		
		Q_N^{\pm} \cH_2,
	\end{split}	
	\end{equation*}
	w.r.t.~products $(\Theta^{\pm}(N) \cdot, \cdot )_\I$ 
	inducing norms equivalent to $\|\cdot\|_\I$. 
	Using the latter and the mutual $J_2$-orthogonality of $Q_n^{\pm}$ 
	(since the corresponding eigenfunctions of $T_2$ are $J_2$-orthogonal), 
	we obtain for arbitrary 
	$f^{\pm} \in \cH_2^{\pm}(N)$ that
	%
%	\marginpar{Check this computation!}
	%
	\begin{equation*}
	\begin{aligned}
		\pm (J_2 f^{\pm}, f^\pm)_{\I}
		& = 
%		\pm \sum_{m,n=0}^N 
%		(J_2 Q_m^\pm f^{\pm}, Q_n^\pm f^{\pm})_\I 
%		=
		\pm \sum_{n=0}^N 
		(J_2 Q_n^\pm f^\pm, Q_n^\pm f^\pm)_\I
		\geq  
		\sum_{n=0}^N \kappa_n^\pm \|Q_n^\pm f^\pm\|^2_\I
%		\\
%		&
		\\
		&
		\geq 
		\min_{n=0,\dots,N} \kappa_n^{\pm } 
		\sum_{n=0}^N \|Q_n^\pm f^\pm\|^2_\I
		\geq 
		\frac{\min_{n=0,\dots,N} \kappa_n^{\pm } }{N+1} \|f^\pm\|^2_\I,
	\end{aligned}
	\end{equation*}
	here $\kappa_n^\pm := \pm (J_2\psi, \psi)_\I$,
	where $\psi \in \dom \WT$ satisfies $\WT\psi = \mu_n^\pm\psi$
	and $\|\psi\|_\I = 1$, and the inequality in \eqref{Th.ineq} is used in the last step. 
	
	Next, as in~\eqref{Tmu}, \eqref{ClosedOps} and~\eqref{Smu}, 
	we introduce the operators $T_1^\mu$, $T_2^\mu(N)$ 
	with $\mu \in \cI = \{+,-,\rm r\}$, defined on the respective subspaces 
	$\cH_1^\mu := P_1^\mu\cH_1$, $\cH_2^\mu(N) := P_2^\mu(N) \cH_2$, 
	and the corresponding tensor products $\sfS^{\mu \nu}(N)$ 
	and $\sfS^{\mu}(N)$, $\mu, \nu \in \cI$. 

	It is straightforward to see that $\s(T_1^+) = \s(T_1)$, $\s(T_1^-) = \s(T_1^\rmr) =\varnothing$,
	$\s(T_2^\pm (N)) = \{\mu_n^\pm\}_{n=0}^N$ and 
	$\s(T_2^{\rm r}(N))  = \s(T_2) \setminus (\s(T_2^+(N)) \cup \s(T_2^-(N)))$.
	Hence, we obtain 
	\begin{align*}
		\s(\sfS^+(N)) & = \s(\sfS^{++}(N)) = \cup_{n=0}^N (\mu_n^+ + \s(T_1)),
		\\
		\s(\sfS^-(N)) &= \s(\sfS^{+-}(N)) = \cup_{n=0}^N (\mu_n^- + \s(T_1)),
		\\
		\s(\sfS^{\rm r}(N)) &= \s(\sfS) \setminus 
		\left(
			\left( \cup_{n=0}^N (\mu_n^+ + \s(T_1)) \right)
			\cup 
			\left( \cup_{n=0}^N (\mu_n^- + \s(T_1)) \right)
		\right).
	\end{align*}
	From Theorem~\ref{thm:abstract}\,(i) and (ii), we receive that, 
	for every $N \in \N_0$, 
	\begin{equation*}
		\big(\cup_{n=0}^N (\mu_n^\pm + \s(T_1))\big) 
		\setminus (\M_\mp \cup \M_0 \cup 
		(\cup_{n=N+1}^{\infty}(\mu_n^\pm + \s(T_1)))) \subset \s_{\pm\pm}(\sfS),
	\end{equation*}
	thus the proof is complete since $\M_\pm\setminus (\M_\mp\cup \M_0)  \subset  \s_{\pm \pm}(\sfS)$.
\end{proof}

The definiteness of the low-lying spectral points in 
$\sigma(\sfH_{\ii\alpha_0}^\Omega)$ for $V_0 \equiv 0$ is visualized in 
Figure~\ref{fig:PTWG.V0}. For a non-trivial $V_0 \neq 0$, $\s(\WL)$ 
can be more complicated than $[0,\infty)$, 
which yields much richer structure for $\sigma(\sfH_{\ii\alpha_0,V}^\Omega)$; 
see Figure~\ref{fig:PTWG.V} for an illustration. 
\begin{figure}[htb!]
	\includegraphics[width=0.8 \textwidth]{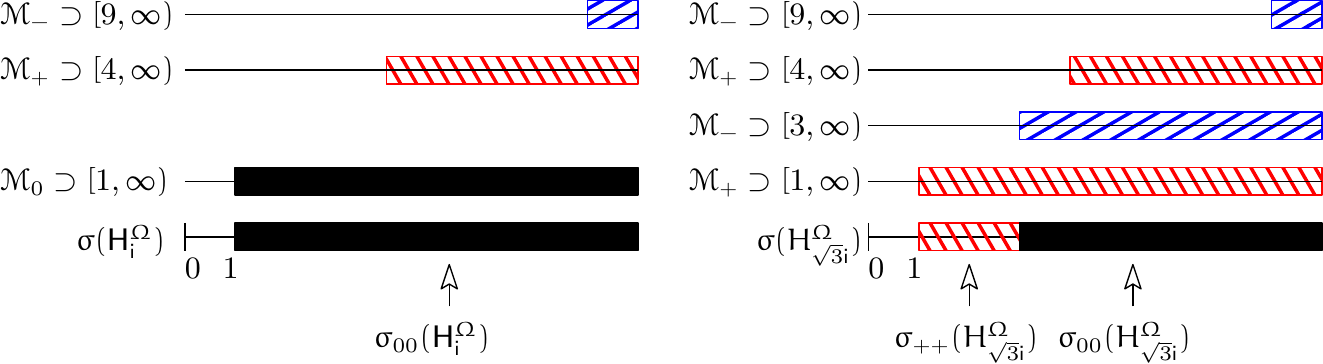}
	\caption{The bottom of $\sigma(\sfH_{\ii\alpha_0}^\Omega)$ 
	and parts of sets $\M_{\pm}, \M_{0}$ for $a=\pi/2$, 
	$\aa_0=1$ and $\aa_0=\sqrt 3$ (left to right).}
	\label{fig:PTWG.V0}
\end{figure}
\begin{figure}[htb!]
	\includegraphics[width=0.8 \textwidth]{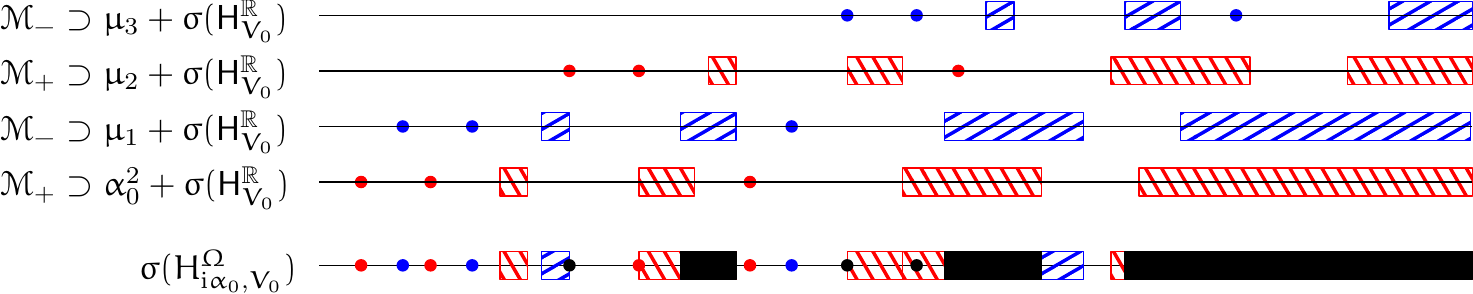}
	\caption{A possible structure and definiteness 
	of the bottom of $\sigma(\sfH_{\ii\alpha_0,V_0}^\Omega)$ 
	for $V_0 \neq 0$. The color codding (red for $++$, blue for $--$ and black for $00$) 
	is the same as in Figure~\ref{fig:PTWG.V0}.}
	\label{fig:PTWG.V}
\end{figure}

\subsection{Perturbations of $\PT$-symmetric waveguides}
\label{ssec:peturb}

We perturb the waveguide $\sfH_{\ii\alpha_0,V_0}$ both in boundary conditions and potential. 
As long as $\|\aa - \bb\|_\infty$ and $\|V - W\|_\infty$ are small, the gap distance $\wh \delta \,
(\sfH_{\aa,V}^\Omega, \sfH_{\beta,W}^\Omega)$, \cf~\cite[Sec.~IV \S2]{Kato-1966},
between $\sfH_{\aa,V}^\Omega$ and $\sfH_{\bb,W}^\Omega$ is small. Moreover, the resolvent difference
of $\sfH_{\aa,V}^\Omega$ and $\sfH_{\bb,W}^\Omega$
is compact if $\aa - \bb \in L^\infty_\infty(\R)$ and $V - W \in L^\infty_\infty(\Omega)$, where 
for $\dX \in \{\dR, \Omega\}$, we define  
\begin{equation}\label{eq:Linftyinfty}
L^\infty_\infty(\dX) 
:= 
\big\{ 
u \in L^\infty(\dX;\dC)
\colon 
\{ x \in \dX \colon |u(x)| \le \eps\}~\text{is bounded}~\forall \eps > 0
\big\}. 
\end{equation}
We indicate how this slight extension of \cite[Prop.~5.1]{Borisov-2008-62} and \cite[Prop. 4.7]{Novak-2014}, 
where less general perturbations in boundary conditions were considered, can be proved. 
%

%------------------------------------------------------------	
\begin{prop}\label{prop:gap}
		Let $a > 0$,
		$\aa \in L^\infty(\R;\C)$ and $V \in L^\infty_\PT(\Omega;\dC)$.
		Then 
		\begin{myenum}
			\item for any $\eps > 0$ there exists 
			$\delta = \delta(\eps, a, \aa, V) > 0$ such that for 
			$\beta \in L^\infty(\R;\C)$ and $W \in L^\infty_\PT(\Omega)$
			satisfying $\|\beta - \aa\|_\infty +  \|W - V\|_\infty \le \delta$
			the operators $\sfH_{\aa, V}^\Omega$ 
			and $\sfH_{\beta, W}^\Omega$ in Definition~\ref{def:Op}
			fulfill
			$\wh \delta \,
			(\sfH_{\aa,V}^\Omega, \sfH_{\beta,W}^\Omega) \le \eps$;
			\item if $\aa - \bb \in L^\infty_\infty(\R)$
			and $V - W \in L^\infty_\infty(\Omega)$, then 
			$(\OpV - \lm)^{-1} - (\sfH_{\beta,W}^\Omega -\lm)^{-1} $
			is compact for all  $\lm \in \rho(\sfH_{\aa,V}^\Omega) \cap \rho(\sfH_{\bb,W}^\Omega)$.		
		\end{myenum}
\end{prop}
\begin{proof}
(i) The claim follows in a straightforward way 
from \cite[Thm.~VI.3.6]{Kato-1966} and Ehrling-type lemma, see \eg~\cite[Lem.~3.1]{Borisov-2008-62} for details in this special situation.

(ii) 
Set $U := W - V$ and $\omega := \beta - \aa$. Denote
$\sfR_{\aa,V}^\Omega(\lm) := (\OpV - \lm)^{-1}$ for $\lm\in\rho(\OpV)$
and define the operators 
$
\sfT^\pm_{\aa,V}(\lm) \colon L^2(\Omega)\to L^2(\R)$,  
$\sfT^\pm_{\aa,V}(\lm) f := ( (\sfR_{\aa, V}^\Omega(\lm)f )|_{\Sigma_\pm}
$
and similarly for $\alpha$ and  $V$ replaced by $\beta$ and $W$.  
Since $\sfH_{\aa,V}^\Omega$ and $\sfH_{\bb,W}^\Omega$
are m-sectorial, there exists $a < 0$ such that
$ a\in \rho(\sfH_{\aa,V}^\Omega) \cap \rho(\sfH_{\bb,W}^\Omega) \cap \rho(\sfH_{\ov{\aa},\ov{V}}^\Omega)\cap \rho(\sfH_{\ov{\bb},\ov{W}}^\Omega).$
The resolvent difference is denoted by $\sfD := \sfR_{\aa,V}^\Omega(a) - \sfR_{\bb,W}^\Omega(a)$.
	
From the trace theorem (\cite[Chap.~3]{McL}), the operators $\sfT^\pm_{\aa,V}(a)$
are everywhere defined in $L^2(\Omega)$ and bounded, moreover, we have
$\ran \sfT^\pm_{\aa,V}(a) \subset H^{1/2}(\R)$.
Let $f,g \in L^2(\Omega)$ and set 
$u := \sfR_{\aa,V}^\Omega(a)f$, $v := \sfR_{\ov{\bb},\ov{W}}^\Omega(a)g$.	
Then, we have
	\begin{equation}\label{eq:D}
	\begin{split}
		(\sfD f,g)_\Omega & 
		= 
		\big( \sfR_{\aa,V}^\Omega(a)f, g\big)_\Omega -
		\big( \sfR_{\bb,W}^\Omega(a)f, g\big)_\Omega
		= \big(u, g\big)_\Omega 
			- 
			\big(f,v \big)_\Omega\\
		& = \big(u, (\sfH_{\ov{\bb},\ov{W}}^\Omega - a)v\big)_\Omega 
			-   
			 \big((\sfH_{\aa,V}^\Omega - a)u, v\big)_\Omega
		\\ & = (u,\sfH_{\ov{\bb},\ov{W}}^\Omega v)_\Omega 
						- (\sfH_{\aa,V}^\Omega u, v)_\Omega.
	\end{split}	
	\end{equation}
	%
%	This formula can be rewritten in a more suitable way. 
	Observe that $u, v  \in H^1(\Omega)$,
	which is the form domain of both the operators $\sfH_{\aa,V}^\Omega$
	and $\sfH_{\ov{\bb},\ov{W}}^\Omega$. 
	Hence, we can use \cite[Thm.~VI.2.1,~VI.2.5]{Kato-1966} 
	to rewrite~\eqref{eq:D} as
	\[
		(\sfD f,g)_\Omega 	= (U u,v)_\Omega + 
		(\omega u|_{\Sigma_+},v|_{\Sigma_+})_{\Sigma_+} + 
		(\ov{\omega}u|_{\Sigma_-},v|_{\Sigma_-})_{\Sigma_-},
	\]
	where we made use of~\eqref{eq:PTWG.form}. 
	In fact, we have shown the resolvent identity
	\begin{align}
		\sfD 
		= 
		\sfR_{\bb, W}^\Omega(a)U\sfR_{\aa,V}^\Omega(a)
		+
		(\sfT_{\ov{\beta},\ov{W}}^+(a))^* \omega \sfT_{\aa,V}^+(a)
		+		
		(\sfT_{\ov{\beta},\ov{W}}^-(a))^* \ov{\omega} \sfT_{\aa,V}^-(a).		
	\end{align}
	The compactness of $\sfD$ follows from $U \in L^\infty_\infty(\Omega)$, $\omega \in L^\infty_\infty(\R)$ and inclusions $\ran \sfT_{\aa,V}^\pm(a)\subset H^{1/2}(\R)$, $\ran \sfR_{\aa,V}^\Omega(a) = \dom \sfH_{\aa,V}^\Omega \subset H^1(\Omega)$. 
	In detail, for $\dX\in\{\Omega,\dR\}$ 
	the product of any $A \in L^\infty_\infty(\dX)$ and $B \in \scB(\cH,L^2(\dX))$ 
	with $\ran B \subset H^s(\dX)$, $s > 0$,
	is a compact operator from $\cH$ into $L^2(\dX)$ .
	The latter can be shown using compactness of Sobolev embeddings and the 
	definition of $L^\infty_\infty(\dX)$; \cf~\cite[Lem. 3.3\,(i)]{LR12} and its proof.
\end{proof}

\subsection{Spectral conclusions for the perturbed waveguide}
	
We draw a spectral conclusion on $\PT$-symmetric waveguides based on Theorem~\ref{thm:PTWG.pn} and
stability of definite type spectra.
Notice that in the case $V_0 \equiv 0$, 
the sets in the claims of Theorems~\ref{thm:B_ext} and~\ref{thm:C_ext} are explicit, 
namely (\cf~Theorem \ref{thm:PTWG.pn})
\begin{equation}\label{Msets.0}
\M_+\setminus (\M_-\cup\M_0) = [\mu_0,\mu_1), \quad \M_-\setminus (\M_+\cup\M_0) = \varnothing.
\end{equation}

\begin{thm}\label{thm:B_ext}
	Let $a > 0$, $\aa_0\in\dR$, $V_0\in L^\infty(\R;\R)$
	and $\M_\pm, \M_0, \M$ be as in~\eqref{eq:Mmu}.
	Then for any compact set $\cF\subset \dC$ satisfying
	either $\cF\cap \M \subset \M_+\setminus (\M_-\cup\M_0)$ or
	$\cF\cap \M \subset \M_-\setminus (\M_+\cup\M_0)$,
	there exists a constant $\delta = \delta(\cF,a,\aa_0,V_0) > 0$ 
	such that for any
	$\aa \in L^\infty(\dR;\dC)$ and any $V\in L^\infty_{\rm \PT}(\Omega;\dC)$
	with $\|\aa - \ii \aa_0\|_\infty + \|V - V_0\|_\infty \le \delta$
	the operator $\OpV$ in Definition~\ref{def:Op} satisfies:
	\begin{myenum}
		\item $\s(\OpV)\cap \cF \subset\dR$;
		\item $\s_\eps(\OpV)\cap \cF 
		\subset \big\{\lm\in \cF\colon |\Im\lm| \le \eps M\big\}$, $\eps >0$,
		with $M = M(\cF,a,\aa,V) > 0$.
	\end{myenum}
\end{thm}

\begin{proof}
	We prove the claim only for  
	$\cF\cap \M \subset \M_+\setminus (\M_-\cup\M_0)$; the other case is analogous.
	By Lemma~\ref{lem:PTWG.sp} and Theorem~\ref{thm:PTWG.pn} we have
	$\s(\sfH_{\ii\aa_0,V_0}) = \M$ and 
	$\spp(\sfH_{\ii\aa_0,V_0}^\Omega) = \M_+\setminus (\M_-\cup\M_0)$. 
	By Theorem~\ref{thm:perturb++} there exists 
	$\gamma = \gamma(\cF,a,\aa_0,V_0)$, 
	$\gamma\in(0,1)$,
	such that for any operator $\sfH\in \cC(L^2(\Omega))$
	satisfying $\wh \delta(\sfH_{\ii\aa_0,V_0}^\Omega,\sfH)\le \gamma$ we have
	$\cF \subset \spp(\sfH) \cup \rmr(\sfH)$. 
	By Proposition~\ref{prop:gap} there exists $\delta = \delta(\gamma,a,\aa_0,V_0) >0$
	such that for $\|\aa-\ii\aa_0\|_\infty + \|V - V_0\|_\infty \le \delta$ 
	we have	$\wh \delta(\sfH_{\ii\aa_0,V_0}^\Omega,\sfH_{\aa,V}^\Omega) \le 
	\gamma$. 
	Moreover, since $\sfH_{\aa,V}^\Omega$ is $\T$-self-adjoint, see Remark~\ref{rem:WG.sym}, 
	the residual spectrum of $\sfH_{\aa,V}^\Omega$ is empty,~\cf~\cite[Cor.~2.1]{Borisov-2008-62}, 
	thus $\rmr(\sfH_{\aa,V}^\Omega) = \rho(\sfH_{\aa,V}^\Omega)$. 
	Hence, we obtain 
	$\cF\subset\spp(\sfH_{\aa,V}^\Omega)\cup \rho(\sfH_{\aa,V}^\Omega)$.
	Now $J$-self-adjointness of $\sfH_{\aa,V}^\Omega$,
	w.r.t.~$J$ in~\eqref{def:JWG}, and Theorem~\ref{thm:++}~(i),~(ii) imply the claims.
\end{proof}

\begin{remark}\label{rem:wild}
	In particular, for $V_0 \equiv 0$, Theorem~\ref{thm:B_ext} 
	shows that the lowest part of the essential spectrum which is of $++$ type, 	
	remains real for all sufficiently small perturbations respecting the symmetry. 		
	If the bottom of the essential spectrum is not of definite type, 
	such conclusions are not valid as shown below.
	
	If $a=\pi/2$ and $\aa_0 = 1$, then the whole $\sess(\sfH^\Omega_\ii)$ is 
	not of definite type, see Theorem \ref{thm:PTWG.pn} and  Figure~\ref{fig:PTWG.V0}. 
	We consider $\sfH_{\ii + \beta_0}^\Omega$ with $\beta_0 \in \R$, \
	\ie~a perturbation of $\sfH^\Omega_\ii$ in boundary conditions. The  eigenvalues of the 
	new transversal operator $\sfH_{\ii + \beta_0}^\I$ 
	obey the algebraic equation (with $\lm = k^2$)
	\begin{equation}\label{beta.EV}
		(k^2 -1 - \beta_0^2)\sin(\pi k) -  2\beta_0 k\cos(\pi k) = 0,
	\end{equation}
	see \cite[Prop.~4.3]{Krejcirik-2010-43}. 
	While $k = 1$ is clearly the solution of~\eqref{beta.EV} 
	for $\beta_0=0$, it is 		
	not difficult to verify that, for any negative $\beta_0$ with
	sufficiently small $|\beta_0|$, the 
	only two solutions $\kappa_1$, $\kappa_2$, of~\eqref{beta.EV} in the 	
	neighborhood of $k=1$ are non-real. Note also that
	$\kappa_1 \to 1$ as $\beta_0 \to 0$.
	Hence, for $\beta_0 < 0$ with sufficiently small $|\beta_0|$, 
	the essential spectrum of $\sfH_{\ii + \beta_0}^\Omega$ contains
	two non-real branches $\kappa_1^2 + \dR_+$, $\kappa_2^2 + \dR_+$ 
	with $\kappa_2^2 = \overline{\kappa_1^2} \notin \R$. It can also
	be shown that all the other solutions of~\eqref{beta.EV}
	for sufficiently small $|\beta_0|$ are real. 
	Hence, the operator $\sfH_{\ii +\beta_0}^\Omega$ has 
	also one more real branch of the essential spectrum  $\kappa_3^2 + \dR_+$;
	see Figure~\ref{fig:3branches}. 
	\begin{figure}[h!]
\figinit{pt}
% Vertices of the triangle
\figpt 1:(0, -50)
\figpt 2:(0, 50)
\figpt 3:(-40, 0)
\figpt 4:(90, 0)
\figpt 5:(10, -20)
\figpt 6:(70, -20)
\figpt 7:(25, 0)
\figpt 8:(90, 0)
\figpt 9:(10, 20)
\figpt 10:(70, 20)
\figpt 11:(100, -20)\figpt 12:(100, 20)\figpt 13:(120, 0)
\figpt 14:(50, 0)

                    % and the barycenter

% 2. Creation of the graphical file
\figdrawbegin{}
% A closed line to draw the triangle

\figdrawline[1,2]
\figdrawline[3,4]

\figset (width = 1.4)        
\figset (color = \Blackrgb)
\figdrawline [5,6]
\figdrawline [9,10]
\figset (color = \Redrgb)
\figdrawline [7,14]
\figset (color = \Blackrgb)
\figdrawline [14,8]

\figset (color = \Redrgb)
\figdrawcirc 7(1.0)
\figset (color = \Blackrgb)
\figdrawcirc 5(1.0)
\figdrawcirc 9(1.0)

\figset(dash = 8)
\figset (color = \Blackrgb)
\figdrawline [6,11]
\figdrawline [10,12]
\figset(dash = 8)
\figset (color = \Blackrgb)
\figdrawline [8,13]

%\figdrawline [3,1]
%\figdrawline [3,2]
%%\figdrawline [13,11]
%%\figdrawline [13,12]
%\figdrawline[14,11]
%\figdrawline[15,12]
%\figdrawarrow [4,5]
%\figdrawarcell 0 ; 0.5,2.776221382176024 (25,335, 0)
%\figdrawarcell 10 ; 0.3,1.388110691088013 (90,260,0)
%\figdrawarcell 20 ; 0.5,2.776221382176024 (0,360, 0)
%\figdrawarcell 21 ; 0.3,1.65 (90,260,0)
%\figdrawarrowcircP 3;1.2[0,1]
%\figset (dash=8)
%\figdrawline[13,14]
%\figdrawline[13,15]
%\figdrawarcell 10 ; 0.3,1.388110691088013 (-90,90,0)
%\figdrawarcell 0 ; 0.5,2.776221382176024 (-25,25, 0)
%\figdrawarcell 21 ; 0.3,1.65 (-90,90,0)
\figdrawend

%
% 3. Writing text on the figure
\figvisu{\figBoxA}{}
{
\figwritee 2:{\tiny $\Im\lm$} (3)
\figwrites 4:{\tiny $\Re\lm$} (3)
\figwriten 7:{\tiny$\kappa_3^2$} (3)
\figwrites 5:{\tiny$\kappa_2^2$} (3)
\figwriten 9:{\tiny$\kappa_1^2$} (3)
}
\centerline{\box\figBoxA}
\caption{$\s(\sfH_{\ii + \beta_0}^\Omega) = \sess(\sfH_{\ii + \beta_0}^\Omega)$ for $\beta_0 < 0$ with sufficiently small $|\beta_0|$
consists of two non-real branches $\kappa_1^2 + \dR_+$ and $\kappa_2^2 + \dR_+$ and one real branch $\kappa_3^2 + \dR_+$.}
\label{fig:3branches}
\end{figure}
	On the qualitative level, arbitrary small perturbation of $\sfH_\ii^\Omega$
	in the gap metric drastically changes spectral properties of the Hamiltonian. 
\end{remark}	
	
\medskip
\begin{thm}\label{thm:C_ext}
	Let $a > 0$, $\aa_0\in\dR$, $V_0 \in L^\infty(\R;\R)$,
	$\M, \M_\pm,\M_0$ be as in~\eqref{eq:Mmu}, $\aa \in L^\infty(\R;\C)$
	be such that $\aa -\ii\aa_0 \in L^\infty_\infty(\R)$ 
	and let $V \in L^\infty_{\rm \PT}(\Omega)$
	be such that $V(x,y) - V_0(x) \in L^\infty_\infty(\Omega)$. 
	Assume that the interval $[a,b]$ satisfies
	either $[a,b]\cap \M \subset \M_+\setminus (\M_-\cup\M_0)$ or $[a,b]\cap \M \subset 
		\M_-\setminus (\M_+\cup\M_0)$.
	Then there is an open neighborhood 
	$\cU \subset\C$ of $[a,b]$ such that $\OpV$ from Definition~\ref{def:Op} satisfies:
	\begin{myenum}
		\item $\s(\OpV)\cap\cU\subset\dR$;
		\item $\s_\eps(\OpV)\cap\cU\subset 
		\big\{\lm\in \cU \colon |\Im\lm| \le M\eps^{1/m}\big\}$,
		$\eps > 0$, with $m = m(\cU,a,\aa,V)\in\dN$ and
		$M = M(\cU,a,\aa,V) > 0$;
		\item there exists at most finite number of eigenvalues
		$\{\nu_k\}_{k=1}^N$, $N\in\dN_0$, of $\OpV$ in $[a,b]$
		such that for any $[c,d]\subset [a,b] \setminus \cup_{k=1}^N\{\nu_k\}$
		one finds an open neighborhood $\cV \subset\dC$ for which
		$\s_\eps(\OpV)\cap\cV\subset \big\{\lm\in \cV \colon |\Im\lm| \le K\eps
			\big\}$ for any $\eps > 0$ with $K = K(\cV,a,\aa,V) >0$.		
	\end{myenum}
\end{thm}

\begin{proof}
	We give the proof only for the case 
	$[a,b] \cap \M \subset \M_+\setminus (\M_-\cup\M_0)$; the second case is analogous. 
	By Lemma~\ref{lem:PTWG.sp} and Theorem~\ref{thm:PTWG.pn} we have
	$\s(\sfH_{\ii\aa_0,V_0}^\Omega) = \M$ and 
	$\spp(\sfH_{\ii\aa_0,V_0}^\Omega) =	\M_+\setminus (\M_-\cup\M_0)$. 
	By Proposition~\ref{prop:gap} the resolvent difference
	$\sfR_{\aa,V}^\Omega(\lm) - \sfR_{\ii\aa_0,V_0}^\Omega(\lm)$
	is compact for all $\lm\in \rho(\OpV)\cap \rho(\sfH_{\ii\aa_0,V_0}^\Omega)$.
	Hence, by Theorem~\ref{thm:peturb_pi_+}	we have 
	$[a,b]\subset\s_{\pi_+}(\OpV)\cup\rho(\OpV)$. 
	Moreover, the essential spectrum of $\OpV$ 
	(all five definitions in \cite[Sec.~IX]{EE} coincide for $\OpV$) 
	is the same as for $\sfH_{\ii\aa_0,V_0}^\Omega$, \cf~\cite[Thm.~IX.2.4]{EE}. 
	This implies in particular that $[a,b] \subset \overline{\rho(\OpV)}$ and therefore 
	the $J$-self-adjointness of $\OpV$, w.r.t.~$J$ in~\eqref{def:JWG}, and
	Theorem~\ref{thm:pi_+}~(i),~(ii) imply respective
	items of this theorem. Theorem~\ref{thm:pi_+}\,(i), (iii) and 
	Theorem~\ref{thm:++}\,(ii) yield item~(iii).
\end{proof}

\subsection*{Acknowledgments}
The authors thank Jussi Behrndt for drawing their attention to definite type 
spectra and Milo\v{s} Tater for numerical examples. VL thanks Carsten Trunk for discussions. 
The research of PS was supported by SNSF Ambizione project PZ00P2\_154786. 
The research of VL was supported by the Austrian Science Fund (FWF): Project P~25162-N26 and the Czech Science Foundation: Project 14-06818S. 
Both the authors acknowledge the support by the Austria-Czech Republic co-operation grant  CZ01/2013.

\begin{appendix}
\section{Properties of spectra of definite type and of type $\pi$}
%----------------------------------------------------------------------
\label{app:sdef}

In what follows $\cH$ is a Hilbert space and $J$ is a bounded symmetric involution in $\cH$. 
Spectra of definite type or type $\pi$ are always defined w.r.t.~this $J$. 
For $T \in \CH$, the set of points of regular type is denoted by  ${\rm r}(T):=\C \setminus \sap(T)$.
\begin{dfn}[\cite{Azizov-2005-226,Philipp-2014}]\label{def:spi}
	A spectral point $\lambda\in\sap(T)$ is of type $\pi_+$ 
	(or $\pi_-$) w.r.t.~$J$ if there exists a closed
	subspace $\cH_\lambda\subseteq\cH$ with ${\rm codim}\, \cH_\lambda <\infty$ 
	such that every approximate eigensequence $\{f_n\}_n \subset \cH_\lambda\cap\dom T$ 
	corresponding to $\lambda$ satisfies
	\begin{equation*}
		\liminf\limits_{n\rightarrow\infty} (Jf_n,f_n) > 0, 
		\qquad (\text{resp.},~ \limsup\limits_{n\rightarrow\infty} (Jf_n,f_n) < 0).
	\end{equation*}
	The set of all spectral points of type $\pi_\pm$ of $T$ is 
	denoted by $\sigma_{\pi_\pm}(T)$.
\end{dfn}
%
%Some remarkable properties of spectra definite type and of type $\pi$ for $J$-self-adjoint operators are listed in the next two theorems.
%
\begin{thm}\label{thm:++}
\cite[Prop.~3]{Azizov-2005-226}
	Let $T \in \CH$ be such that $T=JT^*J$. Then
	\begin{myenum}
		\item $\sigma_{++}(T)\cup\sigma_{--}(T)\subseteq \dR$.
		\item If $\cF\subset \dR$ is closed and
			  either $\cF\cap\sigma(T)\subseteq \sigma_{++}(T)$ or 
			  $\cF\cap\sigma(T)\subseteq \sigma_{--}(T)$, then
			  there exists an open neighborhood $\cU\subset\dC$ of $\cF$ 
			  and a constant $M = M(\cF,T) >0$ such that 
			  $\cU\setminus \dR\subset\rho(T)$ and
			  $\s_\eps(T)\cap\cU \subset 
			  \big\{\lm\in \cU\colon |\Im\lm| \le M\eps\big\}$ for $\eps >0$.
	\end{myenum}
\end{thm}
\begin{thm}\label{thm:pi_+}
\cite[Thm.~17, Thm.~18, Thm.~20]{Azizov-2005-226}
	Let $T \in \CH$ be such that $T=JT^*J$. Then the following statements hold.
	\begin{myenum}
		\item If $\lambda_0 \in\sigma_{\pi_+}(T)\setminus\sigma_{++}(T)$ 
				($\lambda_0 \in\sigma_{\pi_-}(T)\setminus\sigma_{--}(T)$),
			then $\lambda_0 \in \sigma_{\rm p}(T)$ and there is a 
			corresponding eigenvector $\psi_0$ satisfying
			$(J \psi_0, \psi_0) \leq 0$ 
			(resp., $(J \psi_0, \psi_0) \geq  0$).
		\item Let a closed finite interval $[a,b]$ be such that
			\begin{equation*}
				[a,b]\cap\sigma(T)\subseteq \sigma_{\pi_\pm}(T) 
				\quad \mbox{and} \quad [a,b]\subset \ov{\rho(T)},
			\end{equation*}
			where $\ov{\rho(T)}$ stands for the topological closure of $\rho(T)$ in $\dC$.
			Then there exists an open neighborhood
			 $\cU\subset\dC$ of $[a,b]$ such that:
			\begin{enumerate}[\upshape (a)]
				\item $\s(T)\cap\cU\subset\R$;
				\item there exist constants $m = m(\cU,T) \in \dN$ and 
				$M = M(\cU,T) > 0$ for which
				$\s_\eps(T)\cap\cU 
					\subset 
					\big\{\lm\in \cU\colon |\Im\lm| \le M\eps^{1/m}\big\}$,
					for $\eps >0$;
		                \item 
				there is at most finite number $N \in \N_0$ of 
				exceptional eigenvalues 
				$\{\nu_k\}_{k=1}^N \subset \cU\cap\dR$ of $T$
				such that
				$\big(\cU\cap\s(T)\cap \dR\big)
					\setminus \{\nu_k\}_{k=1}^N
					\subset \sigma_{\pm\pm}(T)$.
%               %
		\end{enumerate}
	\end{myenum}

\end{thm}

\begin{thm}\label{thm:peturb_pi_+}
\cite[Thm.~19]{Azizov-2005-226} 
Let $T_k \in \CH$ be such that $T_k=JT_k^*J$, $k=1,2$.
%    Suppose that a bounded symmetric involution $J$ is dfnd on $\cH$
%    and that $A_k = JA^*_kJ$ for $k=1, 2$.
	Assume that $\rho(T_1)\cap\rho(T_2)\ne \varnothing$ and that
	$ (T_2-\mu)^{-1} - (T_1 -\mu)^{-1}$ is compact for some  $\mu\in\rho(T_1)\cap\rho(T_2)$. Then
	$$(\sigma_{\pi_\pm}(T_2)\cup \rho(T_2)) \cap \R  = (\sigma_{\pi_\pm}(T_1)\cup\rho(T_1)) \cap \R.$$
\end{thm}

\begin{thm}\label{thm:perturb++}
\cite[Thm.~4.5]{Azizov-2011-83}
Let $T_1\in\cC(\cH)$. Let a compact set $\cF\subset\dC$ satisfy  
	$\cF\subset \sigma_{++}(T_1)\cup \rmr(T_1)$, 
	(resp. $\cF\subset \sigma_{--}(T_1)\cup \rmr(T_1)$). 
	Then there exists a constant $\gamma = \gamma(\cF,T_1)\in(0,1)$ such that, 
	for all $T_2\in\cC(\cH)$ satisfying $\wh\delta(T_1,T_2) \le \gamma$,
	\begin{equation*}
		\cF\subset \sigma_{++}(T_2)\cup \rmr(T_2), 
		\quad (\text{resp., } \cF\subset \sigma_{--}(T_2)\cup \rmr(T_2)).
	\end{equation*}
\end{thm}
%----------------------------------------------------------------------

\end{appendix}
%
%\newpage
%
%\bibliographystyle{acm}

%\bibliography{references}
%\bibliography{C:/Data/00Synchronized/references}

%
\end{document}